\documentclass[10pt,reqno]{amsart}
\usepackage{amsmath,amssymb,latexsym,esint,cite,mathrsfs,dsfont}
\usepackage{verbatim,cite,relsize,color,soul}
\usepackage[latin1]{inputenc}
\usepackage{microtype}
\usepackage{color,enumitem,graphicx}
\usepackage[colorlinks=true,urlcolor=blue, citecolor=red,linkcolor=blue,
linktocpage,pdfpagelabels, bookmarksnumbered,bookmarksopen]{hyperref}
\usepackage[hyperpageref]{backref}
\usepackage[english]{babel}
\usepackage[font=small,skip=5pt]{caption}
\usepackage{enumitem}
\usepackage{float}

\usepackage{amsrefs}

\usepackage{tikz}
\usepackage{pgfplots}

\usepackage{geometry}
\numberwithin{equation}{section}

\newtheorem{theorem}{Theorem}[section]
\newtheorem{lemma}[theorem]{Lemma}
\newtheorem{proposition}[theorem]{Proposition}
\newtheorem{corollary}[theorem]{Corollary}

\newtheoremstyle{remarkstyle}
{}{}{}{ }{\bfseries}{.}{ }{\thmname{#1}\thmnumber{ #2}\thmnote{ (#3)}}
\theoremstyle{remarkstyle}
\newtheorem{remark}{Remark}[section]
\newtheorem{definition}{Definition}[section]

\newcommand{\R}{\mathbb R}

\newcommand{\vareps}{\varepsilon}
\DeclareMathOperator*{\loc}{loc}
\DeclareMathOperator*{\rad}{{\bf r}}

\DeclareMathOperator*{\opt}{opt}
\DeclareMathOperator*{\sha}{sha}

\DeclareMathOperator*{\supp}{supp}

\DeclareMathOperator*{\rea}{Re}
\DeclareMathOperator*{\ima}{Im}
\DeclareMathOperator*{\inh}{inh}
\DeclareMathOperator*{\gamc}{{\gamma_c}}
\DeclareMathOperator*{\sigc}{{\sigma_c}}

\newcommand{\scal}[1]{\left\langle #1 \right\rangle}

\title[Inhomogeneous NLS]
{Long time dynamics and blow-up for the focusing inhomogeneous nonlinear Schr\"odinger equation with spatial growing nonlinearity}

\author[V. D. Dinh, M. Majdoub \& T. Sanouni]{Van Duong Dinh, Mohamed Majdoub, and Tarek Saanouni}

\address[V. D. Dinh]{Ecole Normale Sup\'erieure de Lyon \& CNRS, UMPA (UMR 5669), France.\newline
	Department of Mathematics, HCMC University of Education, 280 An Duong Vuong, Ho Chi Minh City, Vietnam.}
\email{\sl contact@duongdinh.com}

\address[M. Majdoub]{Department of Mathematics, College of Science, Imam Abdulrahman Bin Faisal University, P. O. Box 1982, Dammam, Saudi Arabia.\newline
	Basic and Applied Scientific Research Center, Imam Abdulrahman Bin Faisal University, P.O. Box 1982, 31441, Dammam, Saudi Arabia.}
\email{\sl mmajdoub@iau.edu.sa}

\address[T. Saanouni]{Department of Mathematics, College of Sciences and Arts of Uglat Asugour, Qassim University, Buraydah, Kingdom of Saudi Arabia.\newline
	University of Tunis El Manar, Faculty of Sciences of Tunis, LR03ES04 partial differential equations and applications, 2092 Tunis, Tunisia.}
\email{\sl  t.saanouni@qu.edu.sa}
\email{\sl tarek.saanouni@ipeiem.rnu.tn}


\begin{document}
	
	\begin{abstract}
		We investigate the Cauchy problem for the focusing inhomogeneous nonlinear Schr\"odinger equation $i \partial_t u + \Delta u = - |x|^b |u|^{p-1} u$ in the radial Sobolev space $H^1_{\rad}(\R^N)$, where $b>0$ and $p>1$. We show the global existence and energy scattering in the inter-critical regime, i.e., $p>\frac{N+4+2b}{N}$ and $p<\frac{N+2+2b}{N-2}$ if $N\geq 3$. We also obtain blowing-up solutions for the mass-critical and mass-supercritical nonlinearities. The main difficulty, coming from the spatial growing nonlinearity, is overcome by refined Gagliardo-Nirenberg type inequalities. Our proofs are based on improved Gagliardo-Nirenberg inequalities, the Morawetz-Sobolev approach of Dodson and Murphy, radial Sobolev embeddings, and localized virial estimates.
	\end{abstract}
	
	\maketitle
	
	\section{Introduction}
	\label{S1}
	\setcounter{equation}{0}
	
	In this paper, we consider the Cauchy problem for the following focusing inhomogeneous nonlinear Schr\"odinger equation
	\begin{align} \label{INLS}
	i \partial_t u + \Delta u = - |x|^b |u|^{p-1} u, \quad (t,x) \in \R_+ \times \R^N,
	\end{align}
	where $b>0$ and $p>1$. The equation \eqref{INLS} is a special case of a more general inhomogeneous nonlinear Schr\"odinger equation
	\begin{align} \label{INLS-K}
	i \partial_t u + \Delta u = K(x) |u|^{p-1} u
	\end{align}
	which arises in various physical contexts such as the propagation of a laser beam and plasma waves. Here $u$ is the electric field in laser optics and $K$ is proportional to the electric density \cite{Gill, LW}. For $p = 3$, the equation \eqref{INLS} can be viewed as a model of dilute Bose-Einstein condensate when the two-body interactions of the condensate are considered \cite{BPVT,TS}.
	
	The Cauchy problem for \eqref{INLS-K} was first investigated by Merle \cite{Merle} who proves the existence of blow-up solutions in the mass-critical regime and under some assumptions on $K$ including in particular $k_1\leq K(x)\leq k_2$ with $k_1, k_2$ positive constants. Later on, the stability of standing waves was studied in \cite{FW, LWW} for $K(x)=K(\varepsilon |x|)$ with $K \in C^4(\R^N)\cap L^\infty(\R^N)$, $\varepsilon>0$ small, and $p\geq 1+\frac{4}{N}$. Recently, the Cauchy problem for \eqref{INLS-K} with $K(x) = \pm |x|^{-b}$ has attracted a lot of interested in the mathematical community (see e.g., \cite{GS, Guzman, Dinh, AT, KLS, Dinh-NA, Farah, BL, FG-JDE, FG-BBMS, Campos, Dinh-2D, MMZ, CFGM, DK-SIAM, Murphy} and references therein). For instance, we refer to \cite{GS, Guzman, Dinh, AT, KLS} for the local well-posedness results, to \cite{Farah, Dinh-NA, BL} for the existence of blow-up solutions, and to
	\cite{FG-JDE, FG-BBMS, Campos, Dinh-2D, MMZ, CFGM, DK-SIAM, Murphy} for sharp thresholds for scattering versus blow-up.
	
	The main difficulty in studying \eqref{INLS} is the spatial growth of $|x|^b$ at infinity. This prevents us to simply use Sobolev embedding (as for $K(x)$ is bounded) or Hardy inequality (as for $K(x)=\pm |x|^{-b}$) to control the potential energy $\mathlarger{\int} |x|^b|u(t,x)|^{p+1} dx$ for $H^1$-solutions. To handle this nonlinearity, Chen and Guo \cite{CG-DCDS-B, Chen, Chen-CMJ, CG-AM} studied \eqref{INLS} in the subspace of $H^1(\R^N)$ consisting of radial functions, namely
	\[
	H^1_{\rad}(\R^N) := \left\{f \in H^1(\R^N) \ : \ f(x)= f(|x|)\right\}.
	\]
	By means of a Gagliardo-Nirenberg type inequality (see \eqref{GN-ineq}) and the energy method (see \cite[Theorem 3.3.9, p. 71]{Cazenave}), it was shown that \eqref{INLS} is locally well-posed in $H^1_{\rad}(\R^N)$ for
	\begin{align} \label{cond-CG}
	N\geq 2, \quad b>0, \quad p>1+\frac{2b}{N-1}, \quad  p<\frac{N+2}{N-2}+\frac{2b}{N-1} \text{ if } N\geq 3.
	\end{align}
	Moreover, there are conservation laws of mass and energy
	\begin{align*}
	M(u(t)) &=\|u(t)\|^2_{L^2} = M(u_0), \tag{Mass} \\
	E(u(t)) &= \frac{1}{2} \|\nabla u(t) \|^2_{L^2} - \frac{1}{p+1} \int |x|^b |u(t,x)|^{p+1}dx = E(u_0). \tag{Energy}
	\end{align*}
	
	Equation \eqref{INLS} has the following scaling invariance
	\[
	u_\lambda(t,x):= \lambda^{\frac{2+b}{p-1}} u(\lambda^2 t, \lambda x), \quad \lambda>0.
	\]
	A direct computation gives
	\[
	\|u_\lambda(0)\|_{\dot{H}^\gamma} = \lambda^{\gamma+\frac{2+b}{p-1}-\frac{N}{2}} \|u_0\|_{\dot{H}^\gamma}
	\]
	which shows that the critical regularity space is $\dot{H}^{\gamc}$, where
	\begin{align} \label{gamc}
	\gamc:= \frac{N}{2}-\frac{2+b}{p-1}.
	\end{align}
	The case $\gamc=0$ (resp. $\gamc=1$) corresponds to the mass-critical nonlinearity $p=\frac{N+4+2b}{N}$ (resp. the energy-critical nonlinearity $p=\frac{N+2+2b}{N-2}$). When $0<\gamc<1$ or $p>\frac{N+4+2b}{N}$ and $p<\frac{N+2+2b}{N-2}$ if $N\geq 3$, \eqref{INLS} is called mass-supercritical and energy-subcritical (inter-critical for short). In this case, for our later purposes, we define the following exponent
	\begin{align} \label{sigc}
	\sigc:= \frac{1-\gamc}{\gamc} = \frac{4+2b-(N-2)(p-1)}{N(p-1)-4-2b}.
	\end{align}
	As $\frac{N+2}{N-2}+\frac{2b}{N-1}< \frac{N+2+2b}{N-2}$ for $N\geq 3$, there is a gap in the earlier local well posedness results in \cite{CG-DCDS-B, Chen, Chen-CMJ,CG-AM} since the range $\frac{N+2}{N-2}+\frac{2b}{N-1}\leq p \leq \frac{N+2+2b}{N-2}$ is not covered. In Section \ref{S2}, we will fill this gap by using the energy method of Cazenave \cite{Cazenave} and radial Sobolev inequalities due to Cho and Ozawa \cite{CO} (see Propositions \ref{prop-lwp} and \ref{prop-lwp-ener}).
	
	The main purpose of the present paper is to investigate long time dynamics such as global existence, energy scattering, and finite time blow-up for \eqref{INLS}. It is well-known that long time dynamics of \eqref{INLS} is strongly related to the following Gagliardo-Nirenberg inequality:
	\begin{align} \label{GN-ineq}
	\int |x|^b |f(x)|^{p+1} dx \leq C(N,p,b) \|\nabla f\|^{\frac{N(p-1)-2b}{2}}_{L^2} \|f\|^{\frac{4+2b-(N-2)(p-1)}{2}}_{L^2}, \quad f \in H^1_{\rad}(\R^N).
	\end{align}
	In a series of work \cite{CG-DCDS-B, Chen, Chen-CMJ, CG-AM}, Chen and Guo proved \eqref{GN-ineq} for $N, b, p$ satisfying \eqref{cond-CG}. As we will see in Lemma \ref{lem-GN-ineq} that the upper bound for $p$ in \eqref{cond-CG} is not optimal when $N\geq 3$ as there is still a gap below the energy-critical regime. In this paper, we revisit the proof of Gagliardo-Nirenberg inequality and prove that \eqref{GN-ineq} holds with $p$ up to the energy-critical exponent $\frac{N+2+2b}{N-2}$. We also prove the optimality of the lower exponent $1+\frac{2b}{N-1}$. See Section \ref{S2} below.

	We first aim to show the global existence and energy scattering for \eqref{INLS} in the inter-critical regime, i.e., $p>\frac{N+4+2b}{N}$ and $p<\frac{N+2+2b}{N-2}$ if $N\geq 3$. We have the following global existence result below a mass-energy threshold.
	
	\begin{proposition} \label{prop-gwp}
		Let $N\geq 2, b>0, p>\max \left\{\frac{N+4+2b}{N}, 1+\frac{2b}{N-1}\right\}$, and $p<\frac{N+2+2b}{N-2}$ if $N\geq 3$. Let $u_0 \in H^1_{\rad}(\R^N)$ satisfy
		\begin{align} \label{cond-scat-inte}
		\begin{aligned}
		E(u_0)(M(u_0))^{\sigc} &< E(Q) (M(Q))^{\sigc}, \\
		\|\nabla u_0\|_{L^2} \|u_0\|^{\sigc}_{L^2} &< \|\nabla Q\|_{L^2} \|Q\|^{\sigc}_{L^2},
		\end{aligned}
		\end{align}
		where $Q$ is the unique positive radial solution to
		\begin{align} \label{equ-Q}
		-\Delta Q + Q - |x|^b |Q|^{p-1}Q =0.
		\end{align}
		Then the corresponding solution to \eqref{INLS} with initial data $\left. u\right|_{t=0}=u_0$ exists globally in time, i.e., $T^*=\infty$.
	\end{proposition}

	\begin{remark}
		The existence and uniqueness of positive radial solution to \eqref{equ-Q} will be addressed in Section \ref{S2}.
	\end{remark}

	Once solutions exist globally in time, a nature question arises that do these solutions scatter at infinity? In this direction, we have the following theorem.
	
	\begin{theorem} \label{theo-scat-inte}
		Let $N\geq 2$, $b>0$, $p>\frac{N+4}{N}+\frac{2b}{N-1}$, and $p<\frac{N+2+2b}{N-2}$ if $N\geq 3$. Let $u_0 \in H^1_{\rad}(\R^N)$ satisfy \eqref{cond-scat-inte}. Then the corresponding global solution to \eqref{INLS} scatters in $H^1_{\rad}(\R^N)$ in the sense that there exists $u^+ \in H^1_{\rad}(\R^N)$ such that
		\[
		\lim_{t\rightarrow \infty} \|u(t) - e^{it\Delta} u^+\|_{H^1} =0.
		\]
	\end{theorem}

	\begin{remark}
	There is a gap $\max\left\{\frac{N+4+2b}{N}, 1+\frac{2b}{N-1}\right\} <p \leq \frac{N+4}{N} + \frac{2b}{N-1}$. We believe that this gap is technical due to our current method.
	\end{remark}

	In the proof of energy scattering given in Theorem \ref{theo-scat-inte}, we avoid the concentration-compactness way, pioneered by Kenig and Merle \cite{KM}, which requires building some heavy machinery in order to obtain the desired space-time bounds. Indeed, we give a simpler method, based on Tao's scattering criterion \cite{Tao}, and Dodson and Murphy's Virial/Morawetz inequalities \cite{DM} (see also \cite{Arora, ADM, DK}). Due to the spatial growth of nonlinearity, we make an intensive use of radial Sobolev embeddings. We refer the reader to Section \ref{S4} for more details.
	
	We are also interested in showing the existence of finite time blow-up solutions to \eqref{INLS}. To our knowledge, there are no results concerning the existence of finite time blow-up solutions with data in $H^1_{\rad}(\R^N)$. Some blow-up results with data in $\Sigma_{\rad}(\R^N):= H^1_{\rad}(\R^N)\cap L^2(\R^N,|x|^2 dx)$ were derived in \cite{CG-DCDS-B, Zhu}. Our contributions in this direction are the following results.
	
	\begin{theorem}[Mass-critical blow-up solutions] \label{theo-blow-mass-appl}
		Let $N\geq 3$, $0<b \leq N-2$, and $p = \frac{N+4+2b}{N}$. Let $u_0 \in H^1_{\rad}(\R^N)$ be such that $E(u_0)<0$. Then the corresponding solution to \eqref{INLS} with initial data $\left. u\right|_{t=0}=u_0$ blows up in finite time.
	\end{theorem}
	
	\begin{remark}
	The upper bound on $b$ (hence the lower bound on $N$) is technical and comes from the choice of cutoff function (see Lemma \ref{lem-vareps}). This restriction can be relaxed to $N\geq 2$ and $0<b \leq 2(N-1)$ if we assume that $u_0\in \Sigma_{\rad}(\R^N)$ since
	\begin{align} \label{viri-iden}
	\frac{d^2}{dt^2}\int |x|^2 |u(t,x)|^2dx=16 E(u_0)+\frac{4(4+2b-N(p-1))}{p+1}\int |x|^b|u(t,x)|^{p+1}dx=16 E(u_0).
	\end{align}
	Here $0<b \leq 2(N-1)$ ensures $p\geq 1+\frac{2b}{N-1}$ which is needed for the existence of local solutions.
	\end{remark}

	\begin{theorem}[Inter-critical blow-up solutions] \label{theo-blow-inter}
		Let $N\geq 2$, $b>0$, $p> \frac{N+4+2b}{N}$, $p<\frac{N+2+2b}{N-2}$ if $N\geq 3$, and $p\leq 5$. Let $u_0 \in H^1_{\rad}(\R^N)$ satisfy either $E(u_0)<0$ or $E(u_0) \geq0$ and  
		\begin{align} \label{cond-blow-inte}
		\begin{aligned}
		E(u_0) (M(u_0))^{\sigc} &< E(Q)(M(Q))^{\sigc}, \\
		\|\nabla u_0\|_{L^2} \|u_0\|^{\sigc}_{L^2} &> \|\nabla Q\|_{L^2} \|Q\|^{\sigc}_{L^2},
		\end{aligned}
		\end{align}
		where $Q$ is the unique positive radial solution to \eqref{equ-Q}. Then the corresponding solution to \eqref{INLS} with initial data $\left. u\right|_{t=0} =u_0$ blows up in finite time.
	\end{theorem}
	
	\begin{remark}
		The restriction $p\leq 5$ is technical as it is needed in an application of Young's inequality (see Lemma \ref{lem-viri-est-inte}). Thanks to \eqref{viri-iden}, this restriction can be removed if we consider initial data in $\Sigma_{\rad}(\R^N)$ satisfying either $E(u_0)<0$ or $E(u_0)\geq 0$ and \eqref{cond-blow-inte}. Note that the assumption $p\leq 5$ implies that $b<2(N-1)$ which ensures $\frac{N+4+2b}{N}>1+\frac{2b}{N-1}$.
	\end{remark}

	\begin{theorem}[Energy-critical blow-up solutions] \label{theo-blow-ener}
		Let $N\geq 4$, $0<b \leq 2(N-3)$, and $p=\frac{N+2+2b}{N-2}$. Let $u_0 \in H^1_{\rad}(\R^N)$ satisfy either $E(u_0)<0$ or $E(u_0) \geq0$ and
		\begin{align} \label{cond-blow-ener}
		\begin{aligned}
		E(u_0) &< E(W), \\
		\|\nabla u_0\|_{L^2} &> \|\nabla W\|_{L^2},
		\end{aligned}
		\end{align}
		where
		\begin{align} \label{W}
		W(x)= \left(1+\frac{|x|^{2+b}}{(N+b)(N-2)}\right)^{-\frac{N-2}{2+b}},
		\end{align}
		is the unique positive radial solution to
		\begin{align} \label{equ-W}
		-\Delta W - |x|^b |W|^{\frac{4+2b}{N-2}} W=0.
		\end{align}
		Then the corresponding solution to \eqref{INLS}  with initial data $\left. u\right|_{t=0}=u_0$ blows up in finite time.
	\end{theorem}
	
	\begin{remark}
		The restrictions on $N$ and $b$ are technical to ensure $p=\frac{N+2+2b}{N-2} \leq 5$. The latter is needed in an application of Young's inequality. As in the inter-critical case, these restrictions can be removed by considering initial data $u_0\in \Sigma_{\rad}(\R^N)$ satisfying either $E(u_0)<0$ or $E(u_0)\geq 0$ and \eqref{cond-blow-ener}.
	\end{remark}
	
	The proofs of blow-up solutions are based on an idea of Ogawa and Tsutsumi \cite{OT-JDE, OT-PAMS} using virial estimates and radial Sobolev embeddings. Due to the presence of a spatial growing nonlinearity, some careful estimates are needed in our analysis.
	
	\begin{remark}
	Finally, we mention that the study of standing waves for \eqref{INLS} will be the subject of a forthcoming work. 
	\end{remark}
	
	This paper is organized as follows. In Section \ref{S2}, we study the existence of optimizers for the Gagliardo-Nirenberg inequality \eqref{GN-ineq}. We also prove the regularity, exponential decay, and the uniqueness of positive radial solutions to \eqref{equ-Q}. In Section \ref{S3}, we show the local well-posedness for \eqref{INLS} with energy-subcritical and energy-critical nonlinearities. Section \ref{S4} is devoted to the proof of energy scattering given in Theorem \ref{theo-scat-inte}. Finally, in Section \ref{S5}, we give the proofs of blow-up solutions for \eqref{INLS} in the mass-critical, mass and energy inter-critical, and energy-critical regimes.

	\section{Gagliardo-Nirenberg inequality revisited}
	\label{S2}
	\setcounter{equation}{0}

	In this section, we revisit the Gagliardo-Nirenberg inequality \eqref{GN-ineq}. More precisely, we have the following result.
	\begin{theorem} \label{theo-GN-ineq}
		Let $N\geq 2$, $b>0$, $p>1+\frac{2b}{N-1}$, and $p<\frac{N+2+2b}{N-2}$ if $N\geq 3$. Then the optimal constant in the Gagliardo-Nirenberg inequality \eqref{GN-ineq} is achieved. Moreover,
		\[
		C_{\opt} = \int |x|^b |Q(x)|^{p+1} dx \div \left[\|\nabla Q\|^{\frac{N(p-1)-2b}{2}}_{L^2} \|Q\|^{\frac{4+2b-(N-2)(p-1)}{2}}_{L^2} \right],
		\]
		where $Q \in H^1_{\rad}(\R^N)$ is the unique positive radial solution to \eqref{equ-Q}.
	\end{theorem}
	
	\begin{remark}
		This result extends the one in \cite[Theorem 1.1]{CG-AM}, where the optimal constant is showed for $N,b,p$ satisfying \eqref{cond-CG}.
	\end{remark}
	
	\begin{remark}
		In \cite{Zhu} (see the proof of Proposition 3.1 given there), this result was proved for $N\geq 3$ and $1+\frac{2b}{N-2}<p<1+\frac{4+2b}{N-2}$. Comparing to \eqref{cond-CG}, there is still a gap between $\frac{N+2}{N-2}+\frac{2b}{N-1}$ and $1+\frac{2b}{N-2}$ for $b$ sufficiently large. Thus our result is an improvement of the one in \cite{Zhu}.
	\end{remark}
	
	Next, we show the regularity and exponential decay of non-trivial solutions to \eqref{equ-Q}.
	
	\begin{proposition} \label{prop-phi}
		Let $N\geq 2, b>0, p>1+\frac{2b}{N-1}$, and $p<\frac{N+2+2b}{N-2}$ if $N\geq 3$. Let $\phi \in H^1_{\rad}(\R^N)$ be a non-trivial solution to \eqref{equ-Q}. Then the following properties hold:
		\begin{itemize}
			\item $\phi \in C^2(\R^N)$ and $|D^\beta \phi(x)| \rightarrow 0$ as $|x| \rightarrow \infty$ for all $|\beta|\leq 2$.
			\item There exists $C>0$ such that
			\[
			e^{C|x|} \left(|\phi(x)| + |\nabla \phi(x)|\right) \in L^\infty(\R^N).
			\]
		\end{itemize}
	\end{proposition}
	By applying the uniqueness criteria of Shioji and Watanabe \cite{SW}, we prove the following uniqueness of positive radial solutions to \eqref{equ-Q}.
	
	\begin{proposition} \label{prop-unique}
		Let $N\geq 2, b>0, p>1+\frac{2b}{N-1}$, and $p<\frac{N+2+2b}{N-2}$ if $N\geq 3$. Then there exists a unique positive radial solution to \eqref{equ-Q}.
	\end{proposition}
	
	Before giving the proofs of Theorem \ref{theo-GN-ineq}, Propositions \ref{prop-phi} and \ref{prop-unique}, let us start by showing that the upper bound for $p$ given in \eqref{cond-CG} is not optimal when $N\geq 3$.
	
	\begin{lemma} \label{lem-GN-ineq}
		Let $N\geq 2$, $b>0$, $p \geq 1+\frac{2b}{N-1}$, and $p\leq \frac{N+2+2b}{N-2}$ if $N\geq 3$. Then there exists $C(N,p,b)>0$ such that
		\[
		\int |x|^b |f(x)|^{p+1} dx \leq C(N,p,b) \|\nabla f\|^{\frac{N(p-1)-2b}{2}}_{L^2} \|f\|^{\frac{4+2b-(N-2)(p-1)}{2}}_{L^2}, \quad \forall f \in H^1_{\rad}(\R^N).
		\]
	\end{lemma}
	
	\begin{proof}
		We only consider $N\geq 3$ since the case $N=2$ was showed in \cite{CG-DCDS-B}. Thanks to the radial Sobolev inequality (see e.g., \cite{Strauss}): for $N\geq 2$ and $f \in H^1_{\rad}(\R^N)$,
		\begin{align} \label{est-Strauss}
		|x|^{\frac{N-1}{2}} |f(x)| \leq C(N) \|\nabla f\|^{\frac{1}{2}}_{L^2} \|f\|^{\frac{1}{2}}_{L^2},
		\end{align}	
		we have
		\begin{align} \label{est-1}
		\int |x|^b |f(x)|^{2+\frac{2b}{N-1}} dx &= \int \left( |x|^{\frac{N-1}{2}} |f(x)|\right)^{\frac{2b}{N-1}} |f(x)|^2 dx \nonumber \\
		&\leq C(N,b) \|\nabla f\|^{\frac{b}{N-1}}_{L^2} \|f\|^{2+\frac{b}{N-1}}_{L^2}.
		\end{align}
		From the radial Sobolev inequality (see e.g., \cite[Appendix]{BL}): for $N\geq 3$ and $f \in H^1_{\rad}(\R^N)$,
		\begin{align} \label{est-BL}
		|x|^{\frac{N-2}{2}} |f(x)| \leq C(N) \|\nabla f\|_{L^2}
		\end{align}
		and the Sobolev embedding $\dot{H}^1(\R^N) \subset L^{\frac{2N}{N-2}}(\R^N)$, we see that
		\begin{align} \label{est-2}
		\int |x|^b |f(x)|^{\frac{2N+2b}{N-2}} dx = \int \left( |x|^{\frac{N-2}{2}} |f(x)|\right)^{\frac{2b}{N-2}} |f(x)|^{\frac{2N}{N-2}} dx \leq C(N,b)\|\nabla f\|^{\frac{2N+2b}{N-2}}_{L^2}.
		\end{align}
		As $2+\frac{2b}{N-1} \leq p+1 \leq \frac{2N+2b}{N-2}$, we interpolate between \eqref{est-1} and \eqref{est-2} to get
		\begin{align*}
		\int |x|^b |f(x)|^{p+1} dx &\leq \left( \int |x|^b |f(x)|^{2+\frac{2b}{N-1}} dx\right)^\theta \left(\int |x|^b |f(x)|^{\frac{2N+2b}{N-2}} dx\right)^{1-\theta} \\
		& \leq C(N,p,b) \|\nabla f\|^{\frac{b}{N-1} \theta + \frac{2N+2b}{N-2}(1-\theta)}_{L^2} \|f\|^{\left(2+\frac{b}{N-1}\right)\theta}_{L^2},
		\end{align*}
		where $\theta \in [0,1]$ is such that
		\[
		p+1 = \left(2+\frac{2b}{N-1}\right)\theta + \frac{2N+2b}{N-2}(1-\theta).
		\]
		A direct calculation yields
		\[
		\frac{b}{N-1} \theta + \frac{2N+2b}{N-2}(1-\theta) = \frac{N(p-1)-2b}{2}, \quad \left(2+\frac{b}{N-1}\right)\theta = \frac{4+2b-(N-2)(p-1)}{2}.
		\]
		The proof is complete.
	\end{proof}
	
	The upper bound $p=\frac{N+2+2b}{N-2}$ is optimal for $N\geq 3$ since it corresponds to the energy critical regularity. The following result shows that the lower bound $1+\frac{2b}{N-1}$ is indeed optimal for the Gagliardo-Nirenberg inequality \eqref{GN-ineq}. 
	
	\begin{lemma} \label{lem-GN-opt}
		Let $N\geq 2$, $b>0$, and $1<p<1+\frac{2b}{N-1}$. Then
			\begin{equation}
			\label{Opt-GNI}
			\sup\left\{ \frac{\displaystyle\int |x|^b |f(x)|^{p+1} dx}{\|f\|_{H^1}^{p+1}} : f \in H^1(\R^N)\backslash \{0\}\right\}=\infty.
			\end{equation}
	\end{lemma}
	
	\begin{proof}
		Define $f_k(x)=\psi(|x|-k)$ where $0\neq \psi\in C_0^\infty(\R)$ with $\supp(\psi)\subset[0,1]$. One can easily verify that
			\begin{align*}
				\|f_k\|_{H^1} ~& \lesssim ~ k^{\frac{N-1}{2}},\\
				\int  |x|^b |f_k(x)|^{p+1} dx ~& \gtrsim~ k^{b+N-1}.
			\end{align*}
			Hence
			\[
			\frac{\displaystyle\int |x|^b |f_k(x)|^{p+1} dx}{\|f_k\|_{H^1}^{p+1}} ~\gtrsim~ k^{\frac{N-1}{2}\left(1+\frac{2b}{N-1}-p\right)}\to \infty \mbox{ as } k\to\infty,
			\]
			since $p<1+\frac{2b}{N-1}$. This finishes the proof.
	\end{proof}
	
	We are next interested in finding optimizers for \eqref{GN-ineq}. To do this, we first show the following compactness result.
	
	\begin{lemma}[Compact embedding] \label{lem-comp-embe}
	Let $N\geq 2$, $b>0$, $p>1+\frac{2b}{N-1}$, and $p<\frac{N+2+2b}{N-2}$ if $N\geq 3$. Then
	\[
	H^1_{\rad}(\R^N) \hookrightarrow L^{p+1}(\R^N, |x|^b dx)
	\]
	is compact.
	\end{lemma}

	\begin{remark}
	This compactness result was proved in \cite[Lemma 2.3]{CG-AM} for $p>1+\frac{2b}{N-1}$ and $p<\frac{N+2}{N-2}+\frac{2b}{N-1}$ if $N\geq 3$. Here we extend this result to $\frac{N+2}{N-2} +\frac{2b}{N-1} \leq p< \frac{N+2+2b}{N-2}$ in dimensions $N\geq 3$.
	\end{remark}

	\begin{proof}[Proof of Lemma \ref{lem-comp-embe}]
	We only consider the case $N\geq 3$ and $\frac{N+2}{N-2} +\frac{2b}{N-1} \leq p< \frac{N+2+2b}{N-2}$. Let $(f_n)_n \subset H^1_{\rad}(\R^N)$ be such that $f_n \rightharpoonup 0$ weakly in $H^1_{\rad}(\R^N)$. We will show that up to a subsequence,
	\begin{align} \label{prof-comp-embe}
	\int |x|^b |f_n(x)|^{p+1} dx \rightarrow 0 \text{ as } n \rightarrow \infty.
	\end{align}
	Since $f_n \rightharpoonup 0$ weakly in $H^1_{\rad}(\R^N)$, up to a subsequence, we have $f_n \rightarrow 0$ a.e. $x \in \R^N$ and $f_n \rightarrow 0$ strongly in $L^r_{\loc}(\R^N)$ for all $1\leq r <\frac{2N}{N-2}$. Let $\vareps>0$. We estimate
	\[
	\int |x|^b |f_n(x)|^{p+1} dx = \left(\int_{|x| \leq \vareps} + \int_{\vareps \leq |x| \leq \frac{1}{\vareps}} +\int_{|x| \geq \frac{1}{\vareps}} \right) |x|^b |f_n(x)|^{p+1} dx = (\text{I}) + (\text{II}) + (\text{III}).
	\]
	For $(\text{I})$, we pick $c>0$ such that $p+1= \frac{2N+2c}{N-2}$ which is possible since $p+1>\frac{2N}{N-2}$. From \eqref{est-2}, we have
	\[
	(\text{I}) = \int_{|x|\leq \vareps} |x|^{b-c} |x|^c |f_n(x)|^{\frac{2N+2c}{N-2}} dx \lesssim \vareps^{b-c} \|\nabla f_n\|^{\frac{2N+2c}{N-2}}_{L^2} \rightarrow 0 \text{ as } \vareps \rightarrow 0.
	\]
	Here we have used the fact that $(f_n)_n$ is bounded uniformly in $H^1(\R^N)$ and $c<b$ as $p<\frac{N+2+2b}{N-2}$. For $(\text{III})$, we use \eqref{est-Strauss} and $(N-1)(p-1) >2b$ to get
	\begin{align*}
	(\text{III}) &= \int_{|x| \geq \frac{1}{\vareps}} |x|^{b-\frac{(N-1)(p-1)}{2}} \left( |x|^{\frac{N-1}{2}} |f_n(x)|\right)^{p-1} |f_n(x)|^2 dx \\
	&\lesssim \vareps^{\frac{(N-1)(p-1)}{2}-b} \|\nabla f_n\|^{\frac{p-1}{2}}_{L^2} \|f_n\|^{\frac{p+3}{2}}_{L^2} \rightarrow 0 \text{ as } \vareps \rightarrow 0.
	\end{align*}
	For $(\text{II})$, we infer from \eqref{est-Strauss} that
	\begin{align*}
	(\text{II}) &=\int_{\vareps \leq |x| \leq \frac{1}{\vareps}} |x|^{b-\frac{(N-1)(p-1)}{2}} \left( |x|^{\frac{N-1}{2}} |f_n(x)|\right)^{p-1} |f_n(x)|^2 dx \\
	&\leq \vareps^{b-\frac{(N-1)(p-1)}{2}} \left(\sup_{\vareps \leq |x| \leq \frac{1}{\vareps}} |x|^{\frac{N-1}{2}} |f_n(x)\right)^{p-1} \|f_n\|^2_{L^2\left(\vareps \leq |x| \leq \frac{1}{\vareps}\right)} \\
	&\lesssim \vareps^{b-\frac{(N-1)(p-1)}{2}} \|\nabla f_n\|^{\frac{p-1}{2}}_{L^2} \|f_n\|^{\frac{p-1}{2}}_{L^2} \|f_n\|^2_{L^2\left(\vareps \leq |x| \leq \frac{1}{\vareps}\right)} \rightarrow 0 \text{ as } n \rightarrow \infty.
	\end{align*}
	Here for fixed $\vareps>0$, we have $f_n \rightarrow 0$ strongly in $L^2\left(\vareps \leq |x| \leq \frac{1}{\vareps}\right)$. Collecting the above estimates, we prove \eqref{prof-comp-embe}.
	\end{proof}

	\begin{proof} [Proof of Theorem \ref{theo-GN-ineq}]
		The proof is similar to that of \cite[Theorem 2.1]{CG-AM} using Lemma \ref{lem-comp-embe}. It was proved in \cite{CG-AM} that there exists $\phi \in H^1_{\rad}(\R^N)$ such that
		\[
		- A \Delta \phi + B \phi - \frac{p+1}{C_{\opt}} |x|^b |\phi|^{p-1} \phi =0
		\]
		and
		\[
		C_{\opt} = W(\phi):= \int |x|^b |\phi(x)|^{p+1} dx \div \left[\|\nabla \phi\|^A_{L^2} \|\phi\|^B_{L^2} \right],
		\]
		where
		\[
		A:= \frac{N(p-1)-2b}{2}, \quad B:= \frac{4+2b-(N-2)(p-1)}{2}.
		\]
		Using the fact that $|\nabla |\phi|(x)|\leq |\nabla \phi(x)|$ a.e. $x \in \R^N$, we see that $|\phi|$ is also an optimizer for \eqref{GN-ineq}. Setting $\phi(x) = \lambda Q(\mu x)$ with $\lambda, \mu>0$ such that
		\[
		\lambda^{p-1}= \frac{B C_{\opt}}{p+1} \left(\frac{B}{A}\right)^{b/2}, \quad \mu^2 = \frac{B}{A},
		\]
		we see that $Q$ is a solution to \eqref{equ-Q} and $C_{\opt} = W(Q)$. As $Q \geq 0$, we have
		\[
		\Delta Q - Q = - |x|^b|Q|^{p-1} Q \leq 0 \text{ on } \R^N.
		\]
		By the maximum principle (see e.g., \cite[Theorem 3.5]{GT}), $Q$ is either positive or identically zero. Therefore, there exists an optimizer for \eqref{GN-ineq} which is a positive and radially symmetric solution to \eqref{equ-Q}. The uniqueness of such a solution is given in Proposition \ref{prop-unique}. The proof is complete.
	\end{proof}
	
	We also have the following Pohozaev's identity due to \cite{Chen}.
	
	\begin{lemma}[\cite{Chen}] \label{lem-poho-iden}
		Let $N\geq 2$, $b>0$, $p>1+\frac{2b}{N-1}$, and $p<\frac{N+2+2b}{N-2}$ if $N\geq 3$. Let $Q \in H^1_{\rad}(\R^N)$ be a non-trivial solution to \eqref{equ-Q}. Then
		\begin{align} \label{poho-iden}
		\|\nabla Q\|^2_{L^2} = \frac{N(p-1)-2b}{4+2b-(N-2)(p-1)} \|Q\|^2_{L^2} = \frac{N(p-1)-2b}{2(p+1)} \int |x|^b |Q(x)|^{p+1} dx.
		\end{align}
	\end{lemma}
	
	\begin{lemma} \label{lem-regu}
		Let $N\geq 2$, $b>0$, $p>1+\frac{2b}{N-1}$, and $p<\frac{N+2+2b}{N-2}$ if $N\geq 3$. Let $\phi \in H^1_{\rad}(\R^N)$ be a non-trivial solution to \eqref{equ-Q}. Then $\phi \in C^2(\R^N)$ and $|D^\beta \phi(x)| \rightarrow 0$ as $|x| \rightarrow \infty$ for all $|\beta|\leq 2$.
	\end{lemma}
	
	\begin{proof} We follow an argument of \cite{Cazenave} and proceed in several steps.
		
		{\bf Step 1.} We claim that $\phi \in L^q_{\rad}(\R^N)$ for all $2\leq q <\infty$. Assume this claim for the moment, we have $|x|^b |\phi|^{p-1} \phi \in L^q_{\rad}(\R^N)$ for all $2\leq q<\infty$. As
		\begin{align} \label{equ-phi}
		-\Delta \phi +\phi -|x|^b |\phi|^{p-1}\phi=0,
		\end{align}
		the elliptic regularity yields $\phi \in W^{2,q}_{\rad}(\R^N)$ for all $2\leq q<\infty$, hence $\partial_j \phi \in W^{1,q}_{\rad}(\R^N)$ for all $j=1,\cdots, N$ and all $2\leq q<\infty$. We will prove the claim by considering several cases.
		
		$\bullet$ We first consider the case $N\geq 2$, $b>0$, $p>1+\frac{2b}{N-1}$, and $p<\frac{N+2}{N-2} +\frac{2b}{N-1}$ if $N\geq 3$. Observe that if $\phi \in L^r_{\rad}(\R^N)$ for some $r>\alpha+1$, then $|x|^b |\phi|^{p-1} \phi \in L^{\frac{r}{\alpha+1}}_{\rad}(\R^N)$, where $\alpha=p-1-\frac{2b}{N-1}$. In fact, by \eqref{est-Strauss},
		\[
		\||x|^b|\phi|^{p-1} \phi\|_{L^{\frac{r}{\alpha+1}}} = \left\|\left(|x|^{\frac{N-1}{2}} |\phi|\right)^{\frac{2b}{N-1}} |\phi|^\alpha \phi \right\|_{L^{\frac{r}{\alpha+1}}} \lesssim \|\nabla \phi\|^{\frac{b}{N-1}}_{L^2} \|\phi\|^{\frac{b}{N-1}}_{L^2} \|\phi\|^{\alpha+1}_{L^r}.
		\]
		From \eqref{equ-phi}, we infer that $\phi \in W^{2,\frac{r}{\alpha+1}}_{\rad}(\R^N)$. By Sobolev embedding, we have
		\begin{align} \label{regu-q}
		\phi \in L^q_{\rad}(\R^N) \text{ for all } q\geq \frac{r}{\alpha+1} \text{ such that } \frac{1}{q} \geq \frac{\alpha+1}{r} - \frac{2}{N}.
		\end{align}
		Define for $n\geq 0$, $q_0=\alpha+2$ and
		\[
		\frac{1}{q_{n+1}} = \frac{\alpha+1}{q_n}-\frac{2}{N}, \quad n\geq 0.
		\]
		In particular, we have
		\begin{align*}
		\frac{1}{q_{n+1}} = (\alpha+1) \left(\frac{\alpha+1}{q_{n-1}} - \frac{2}{N}\right) - \frac{2}{N} &= (\alpha+1)^2 \left(\frac{1}{q_{n-1}} -\frac{2(\alpha+2)}{N(\alpha+1)^2}\right) \\
		&=(\alpha+1)^2 \left(\frac{1}{q_{n-1}} -\frac{2}{N\alpha} +\frac{2}{N\alpha(\alpha+1)^2}\right).
		\end{align*}
		By induction, we get
		\[
		\frac{1}{q_n}= (\alpha+1)^n \left(\frac{1}{\alpha+2} -\frac{2}{N\alpha} +\frac{2}{N\alpha(\alpha+1)^n}\right).
		\]
		As $(N-2)\alpha<4$ due to $p<\frac{N+2}{N-2}+\frac{2b}{N-1}$, we have
		\[
		\frac{1}{q_{n+1}} - \frac{1}{q_n} = -(\alpha+1)^n \left(\frac{\alpha}{\alpha+2} -\frac{2}{N}\right) <0.
		\]
		This shows that $\frac{1}{q_n}$ is decreasing in $n$ and $\frac{1}{q_n} \rightarrow -\infty$ as $n\rightarrow \infty$. As $q_0=\alpha+2$, there exists $k\geq 0$ such that
		\[
		\frac{1}{q_n} >0 \text{ for } 0 \leq n \leq k \text{ and } \frac{1}{q_{k+1}} \leq 0.
		\]
		Since $\phi \in H^1_{\rad}(\R^N)$, we have $\phi \in L^{q_0}_{\rad}(\R^N)$. If $\phi \in L^{q_n}_{\rad}(\R^N)$ for some $n\leq k-1$, then by \eqref{regu-q},
		\[
		\phi \in L^q_{\rad}(\R^N) \text{ for all } q\geq \frac{q_n}{\alpha+1} \text{ such that } \frac{1}{q} \geq \frac{\alpha+1}{q_n} - \frac{2}{N} = \frac{1}{q_{n+1}},
		\]
		hence $\phi \in L^{q_{n+1}}_{\rad}(\R^N)$. This shows that $\phi \in L^{q_k}_{\rad}(\R^N)$ and, repeating the above argument, we get $\phi \in L^{q_{k+1}}_{\rad}(\R^N)$. Hence the claim is proved in this case.
		
		$\bullet$ We next consider the case $N\geq 3$, $b>0$, and $1+\frac{2b}{N-2}<p<\frac{N+2+2b}{N-2}$. Repeating the same reasoning as above with $\alpha=p-1-\frac{2b}{N-2}$ and using \eqref{est-BL} instead of \eqref{est-Strauss}, we can conclude the claim in this case.
		
		$\bullet$ Now let us consider the general case for $N\geq 3$. Comparing between $\frac{N+2}{N-2}+\frac{2b}{N-1}$ and $1+\frac{2b}{N-2}$, we see that if $0<b < 2(N-1)$, then we have the claim for all $1+\frac{2b}{N-1}<p<\frac{N+2+2b}{N-2}$. In the case $b\geq 2(N-1)$, the claim was proved for
		\[
		1+\frac{2b}{N-1} <p<\frac{N+2}{N-2} +\frac{2b}{N-1}, \quad 1+\frac{2b}{N-2} < p < \frac{N+2+2b}{N-2}.
		\]
		To fill the gap on $p \in \left[\frac{N+2}{N-2}+\frac{2b}{N-1}, 1+\frac{2b}{N-2}\right]$ when $b\geq 2(N-1)$, we first recall the following radial Sobolev embedding due to Cho and Ozawa \cite{CO}: for $N\geq 3$, $\frac{1}{2} \leq s \leq 1$, and $f \in H^1_{\rad}(\R^N)$,
		\begin{align} \label{est-CO}
		|x|^{\frac{N-2s}{2}} |f(x)|\leq C(N,s) \|\nabla f\|^s_{L^2} \|f\|^{1-s}_{L^2}.
		\end{align}
		We then take $\frac{1}{2}<s_1<1$ such that $\frac{N+2}{N-2}+\frac{2b}{N-1} = 1+\frac{2b}{N-2s_1}$. It is possible since $1+\frac{2b}{N-1}<\frac{N+2}{N-2}+\frac{2b}{N-1}< 1+\frac{2b}{N-2}$. Repeating the same argument as in the first case with $\alpha=p-1-\frac{2b}{N-2s_1}$ and using \eqref{est-CO}, we can prove the claim with $1+\frac{2b}{N-2s_1} \leq p<\frac{N+2}{N-2} +\frac{2b}{N-2s_1}$. Note that the case $\alpha=0$ works as well. If $2(N-1)<b <\frac{N-2s_1}{1-s_1}$, we have $\frac{N+2}{N-2} + \frac{2b}{N-2s_1} > 1+\frac{2b}{N-2}$. The claim is thus proved for all $1+\frac{2b}{N-1}<p<\frac{N+2+2b}{N-2}$. Otherwise, if $b\geq \frac{N-2s_1}{1-s_1}$, we take $s_2 \in (s_1,1)$ so that $\frac{N+2}{N-2} +\frac{2b}{N-2s_1} = 1+\frac{2b}{N-2s_2}$. Repeating the same line of arguments, we proved the claim for $1+\frac{2b}{N-2s_2} \leq p <\frac{N+2}{N-2} +\frac{2b}{N-2s_2}$. If $\frac{N-2s_1}{1-s_1} <b<\frac{N-2s_2}{1-s_2}$, we are done. Otherwise, we repeat the above argument until $\frac{N-2s_{k-1}}{1-s_{k-1}} <b<\frac{N-2s_k}{1-s_k}$. Note that this process will be terminated in finite steps since $s_k \rightarrow 1$ (i.e., $\frac{N-2s_k}{1-s_k}\gg 1$ for $k$ large) and $b>0$ is given. The claim is now proved.
		
		{\bf Step 2.} For each $j=1, \cdots, N$, we have
		\[
		-\Delta \partial_j \phi + \partial_j \phi = \partial_j( |x|^b |\phi|^{p-1} \phi) \sim |x|^b |\phi|^{p-1} \partial_j \phi + |x|^{b-1} |\phi|^{p-1} \phi.
		\]
		As $\phi, \partial_j \phi \in L^q_{\rad}(\R^N)$ for all $2\leq q<\infty$, we infer that $|x|^b |\phi|^{p-1} \partial_j \phi \in L^q_{\rad}(\R^N)$ for all $2 \leq q<\infty$.

		When $b\geq 1$, we have from \eqref{est-Strauss} that
		\[
		\||x|^{b-1} |\phi|^{p-1} \phi\|_{L^q} = \left\| \left( |x|^{\frac{N-1}{2}}|\phi|\right)^{\frac{2(b-1)}{N-1}} |\phi|^{\alpha+\frac{2}{N-1}} \phi\right\|_{L^q} \lesssim \|\nabla \phi\|^{\frac{b-1}{N-1}}_{L^2} \|\phi\|^{\frac{b-1}{N-1}}_{L^2} \|\phi\|^{\alpha+1+\frac{2b}{N-1}}_{L^{q\left(\alpha+1+\frac{2b}{N-1}\right)}}
		\]
		which shows $|x|^{b-1} |\phi|^{p-1} \phi \in L^q_{\rad}(\R^N)$ for all $2\leq q<\infty$.
		
		When $0<b<1$, we take $N<q^*<\frac{N}{1-b}$. We will show that $|x|^{b-1}|\phi|^{p-1}\phi \in L^q_{\rad}(\R^N)$ for all $2\leq q \leq q^*$. In fact, we have
		\[
		\||x|^{-(1-b)} |\phi|^{p-1} \phi\|_{L^q} \leq \||x|^{-(1-b)} |\phi|^{p-1} \phi\|_{L^q(B(0,1))} + \||x|^{-(1-b)} |\phi|^{p-1} \phi\|_{L^q(B^c(0,1))},
		\]
		where $B(0,1)$ is the unit ball centered at the origin. On $B^c(0,1)$, we have
		\[
		\||x|^{-(1-b)} |\phi|^{p-1} \phi\|_{L^q(B^c(0,1))} \leq \||\phi|^{p-1} \phi\|_{L^q}\leq \|\phi\|^p_{L^{qp}}.
		\]
		On $B(0,1)$, we estimate
		\[
		\||x|^{-(1-b)} |\phi|^{p-1} \phi\|_{L^q(B(0,1))} \leq \||x|^{-(1-b)}\|_{L^\gamma(B(0,1))} \||\phi|^{p-1}\phi\|_{L^\rho} \lesssim \|\phi\|^p_{L^{\rho p}}
		\]
		provided that $\gamma, \rho \geq 1$ satisfy $\frac{1}{q}=\frac{1}{\gamma} +\frac{1}{\rho}$ and $\||x|^{-(1-b)}\|_{L^\gamma(B(0,1))}<\infty$. These conditions are satisfied by taking
		$\gamma=q^*+\vareps$ with some $\vareps>0$ sufficiently small and $\rho =\frac{q\gamma}{\gamma-q}$.
		
		In any case, we have proved that $|x|^{b-1} |\phi|^{p-1} \phi \in L^q_{\rad}(\R^N)$ for all $2\leq q \leq q^*$ for some $q^*>N$. Thus we get $\partial_j \phi \in W^{2,q}_{\rad}(\R^N)$ for all $2\leq q \leq q^*$ and all $j=1,\cdots, N$. Hence $\phi \in W^{3,q}_{\rad}(\R^N)$ for all $2\leq q\leq q^*$. By Morrey's inequality, $\phi \in C^{2,\delta}(\R^N)$ for all $0<\delta<1$. In particular, $|D^\beta \phi(x)| \rightarrow 0$ as $|x| \rightarrow \infty$ for all $|\beta|\leq 2$.
	\end{proof}
	
	\begin{lemma} \label{lem-deca}
		Let $N\geq 2$, $b>0$, and $1+\frac{2b}{N-1}<p<\frac{N+2+2b}{N-2}$. Let $\phi \in H^1_{\rad}(\R^N)$ be a non-trivial solution to \eqref{equ-Q}. Then there exists $C>0$ such that
		\[
		e^{C|x|} (|\phi(x)| + |\nabla \phi(x)|) \in L^\infty(\R^N).
		\]
	\end{lemma}
	
	\begin{proof}
		Let $\vareps>0$. We define $\theta_\vareps(x):= e^{\frac{|x|}{1+\vareps |x|}}$. We see that $\theta_\vareps$ is bounded, Lipschitz continuous, and $|\nabla \theta_\vareps(x)| \leq \theta_\vareps(x) \leq e^{|x|}$ for all $x \in \R^N$. Taking the scalar product with $\theta_\vareps \phi$ in $H^1(\R^N)$, we infer from \eqref{equ-Q} that
			\[
			\rea \int \nabla \phi \cdot \nabla(\theta_\vareps \overline{\phi}) dx + \int \theta_\vareps |\phi|^2 dx =\int \theta_\vareps |x|^b |\phi|^{p+1} dx.
			\]
			Using $\nabla (\theta_\vareps \overline{\phi}) = \overline{\phi} \nabla \theta_\vareps + \theta_\vareps \nabla \overline{\phi}$, we have
			\[
			\rea (\nabla \phi \cdot \nabla (\theta_\vareps \overline{\phi})) \geq \theta_\vareps |\nabla \phi|^2 - \theta_\vareps |\phi||\nabla \phi|.
			\]
			By the Cauchy-Schwarz's inequality, we get
			\[
			\rea (\nabla \phi \cdot \nabla (\theta_\vareps \overline{\phi})) \geq \frac{1}{2} \theta_\vareps |\nabla \phi|^2 - \frac{1}{2} \theta_\vareps |\phi|^2.
			\]
			Thus
			\[
			\int \theta_\vareps (|\nabla \phi|^2 +|\phi|^2) dx \leq 2 \int \theta_\vareps |x|^b |\phi|^{p+1} dx \lesssim \|\nabla \phi\|^{\frac{b}{N-1}}_{L^2} \|\phi\|^{\frac{b}{N-1}}_{L^2} \int \theta_\vareps |\phi|^{p+1-\frac{2b}{N-1}} dx.
			\]
			As $|\phi(x)| \rightarrow 0$ as $|x| \rightarrow \infty$, we have $|\phi(x)|^{p-1-\frac{2b}{N-1}} \leq \frac{1}{100}$ for all $|x| \geq R$. It follows that
			\begin{align*}
			\int \theta_\vareps (|\nabla \phi|^2 +|\phi|^2) dx &\lesssim \int_{|x|\leq R} \theta_\vareps |\phi|^{p+1-\frac{2b}{N-1}} dx + \int_{|x|\geq R} \theta_\vareps |\phi|^{p+1-\frac{2b}{N-1}} dx \\
			&\lesssim \int_{|x| \leq R} e^{|x|} |\phi|^{p+1-\frac{2b}{N-1}} dx + \frac{1}{100} \int \theta_\vareps |\phi|^2 dx.
			\end{align*}
			Hence
			\[
			\int \theta_\vareps (|\nabla \phi|^2 +|\phi|^2) dx \lesssim \int_{|x| \leq R} e^{|x|} |\phi|^{p+1-\frac{2b}{N-1}} dx <\infty.
			\]
			Letting $\vareps \rightarrow 0$, we get
			\begin{align} \label{est-phi}
			\int e^{|x|} (|\nabla \phi|^2 +|\phi|^2) dx <\infty.
			\end{align}
			As $|\nabla \phi| \in L^\infty(\R^N)$ (see Lemma \ref{lem-regu}), 
			we denote $L=\|\nabla \phi\|_{L^\infty}+1$. For $x \in \R^N$, we define
			\[
			B(x):= \left\{ y\in \R^N : |x-y| \leq \frac{|\phi(x)|}{2L}\right\}.
			\]
			By the choice of $L$, we have $|\phi(x)-\phi(y)|\leq L|x-y|$. Thus for any $y \in B(x)$,
			\begin{align*}
			|\phi(x)|^2 \leq (|\phi(y)|+L|x-y|)^2 \leq 2|\phi(y)|^2 +2L^2|x-y|^2 \leq 2|\phi(y)|^2+\frac{1}{2}|\phi(x)|^2,
			\end{align*}
			hence $|\phi(x)|^2\leq 4|\phi(y)|^2$ for all $y \in B(x)$. Integrating this inequality over $B(x)$, we have
			\[
			C(N)\left(\frac{|\phi(x)|}{2L}\right)^N |\phi(x)|^2 \leq 4 \int_{B(x)} |\phi(y)|^2 dy
			\]
			or
			\[
			|\phi(x)|^{N+2} \leq \frac{4(2L)^N}{C(N)} \int_{B(x)} |\phi(y)|^2 dy.
			\]
			For $y \in B(x)$, we have $|x| \leq |y| +\frac{\|\phi\|_{L^\infty}}{2L}$. From this, we see that
			\begin{align*}
			e^{|x|} |\phi(x)|^{N+2} &\leq \frac{4(2L)^N}{C(N)} e^{|x|} \int_{B(x)} |\phi(y)|^2 dy \\
			&\leq \frac{4(2L)^N}{C(N)} e^{\frac{\|\phi\|_{L^\infty}}{2L}} \int_{B(x)} e^{|y|} |\phi(y)|^2 dy \\
			&\leq \frac{4(2L)^N}{C(N)} e^{\frac{\|\phi\|_{L^\infty}}{2L}} \int e^{|y|} |\phi(y)|^2 dy.
			\end{align*}
			This together with \eqref{est-phi} yield $e^{\frac{|x|}{N+2}} |\phi(x)| <\infty$. A similar argument goes for $\nabla \phi$ and the proof is complete.
		\end{proof}
	
	\begin{proof}[Proof of Proposition \ref{prop-phi}]
		It follows directly from Lemmas \ref{lem-regu} and \ref{lem-deca}.
	\end{proof}	

	We next recall the uniqueness result of positive radial solution to an ordinary differential equation due to Shioji and Watanabe \cite{SW}. We consider
	\begin{align} \label{ordi-equa}
	\left\{
	\begin{aligned}
	\phi'' + \frac{f'(r)}{f(r)} \phi' &+ g(r) \phi + h(r) \phi^p =0, \quad 0<r<\infty, \\
	\phi(0)>0, \quad &\phi(r) \rightarrow 0 \text{ as } r\rightarrow \infty,
	\end{aligned}
	\right.
	\end{align}
	where $f,g,h:(0,\infty) \rightarrow \R$ are given functions. We say that $\phi$ is a positive solution to \eqref{ordi-equa} if $\phi \in C([0,\infty)) \cap C^2((0,\infty))$ and $\phi$ satisfies \eqref{ordi-equa}.
	
	We introduce the following functions
	\begin{align*}
	\alpha(r) &= (f(r))^{\frac{2(p+1)}{p+3}} (h(r))^{-\frac{2}{p+3}}, \\
	\beta(r) &= -\frac{1}{2} \alpha'(r) + \frac{f'(r)}{f(r)} \alpha(r), \\
	\gamma(r) &= -\beta'(r) +\frac{f'(r)}{f(r)} \beta(r), \\
	G(r) &= -\beta(r) g(r) + \frac{1}{2}\gamma'(r) + \frac{1}{2} (\alpha(r) g(r))'.
	\end{align*}
	
	\begin{theorem}[\cite{SW}] \label{theo-SW}
		Let $p>1$, $f,h \in C^3((0,\infty))$ are positive functions, and $g \in C^1((0,\infty))$. Assume the following conditions:
		\begin{enumerate}[leftmargin=7mm,itemsep=2pt]
			\item $\limsup_{r\rightarrow 0} f(r) <\infty$.
			\item $\lim_{r \rightarrow 0} \frac{1}{f(r)} \mathlarger{\int}_0^r f(\tau)(|g(\tau)| + h(\tau)) d\tau =0$.
			\item There exists $R>0$ such that
			\begin{enumerate}
				\item $f(g + h) \in L^1((0,R))$.
				\item $r \mapsto f(r) (|g(r)| + h(r)) \mathlarger{\int}_r^R \frac{1}{f(\tau)} d\tau \in L^1((0,R))$.
				\item $1/f \notin L^1((0,R))$.
			\end{enumerate}
			\item $\limsup_{r\rightarrow 0} \alpha(r) <\infty$, $\limsup_{r\rightarrow 0} |\beta(r)| <\infty$, and $\lim_{r\rightarrow 0} \alpha(r) g(r) =\lim_{r\rightarrow 0} \alpha(r) h(r) =0$.
			\item $\lim_{r\rightarrow 0} \gamma(r) \in [0,\infty]$.
			\item There exists $k\in [0,\infty]$ such that $G(r) \geq 0 \text{ on } (0,k)$ and $G(r) \leq 0$ on $(k,\infty)$.
			\item $G^- \not \equiv 0$, where $G^-(r) = \min \{G(r),0\}$ for $r\in (0,\infty)$.
		\end{enumerate}
		Then \eqref{ordi-equa} has at most one positive solution.
	\end{theorem}
	
	\begin{proof}[Proof of Proposition \ref{prop-unique}]		
	By Theorem \ref{theo-GN-ineq}, there exists a positive radial solution to \eqref{equ-Q}. By Lemmas \ref{lem-regu}, we have for $N\geq 2$, $b>0$, $p>1+\frac{2b}{N-1}$, and $p<\frac{N+2+2b}{N-2}$ if $N\geq 3$ that $\phi \in C([0,\infty)) \cap C^2((0,\infty))$. We also have
	\[
	\phi'' + \frac{N-1}{r} \phi' - \phi + r^b \phi^p =0,
	\]
	hence $\phi$ solves \eqref{ordi-equa} with
	\[
	f(r) = r^{N-1}, \quad g(r)=-1, \quad h(r) =r^b.
	\]
	Let us check the assumptions of Theorem \ref{theo-SW}. To this end, we compute
	\begin{align*}
	\alpha(r) &= r^{\frac{2(N-1)(p+1)-2b}{p+3}}, \\
	\beta(r) &= \frac{2(N-1) +b}{p+3} r^{\frac{2(N-1)(p+1)-2b}{p+3}-1}, \\
	\gamma(r) &= \frac{2(N-1)+b}{p+3} \times \frac{2N+2b-(N-2)(p+1)}{p+3} r^{\frac{2(N-1)(p+1)-2b}{p+3}-2}, \\
	G(r) &= [-Cr^2+D]r^{\frac{2(N-1)(p+1)-2b}{p+3}-3},
	\end{align*}
	where
	\begin{align*}
	C&= \frac{(N-1)(p-1)-2b}{p+3}, \\
	D&= \frac{2(N-1)+b}{p+3} \times \frac{2N+2b-(N-2)(p+1)}{p+3}\times \frac{(N-2)(p+1)-2-b}{p+3}.
	\end{align*}
	
	It is obvious that the first assumption is satisfied. We have
	\[
	\frac{1}{f(r)} \int_0^r f(\tau) (|g(\tau)| + h(\tau)) d\tau = \frac{r}{N} + \frac{r^{b+1}}{N+b},
	\]
	hence the second assumption is satisfied. We also have for any $R>0$,
	\[
	f(r)(|g(r)| + h(r)) = r^{N-1}(1+r^b) \in L^1((0,R))
	\]
	and
	\[
	f(r)(|g(r)| + h(r)) \int_r^R \frac{1}{f(\tau)} d\tau = (1+r^b) \frac{r - R^{2-N}r^{N-1}}{N-2} \in L^1((0,R))
	\]
	and
	\[
	\frac{1}{f(r)} = r^{1-N} \notin L^1((0,R)).
	\]
	Thus the third assumption is fulfilled. The fourth assumption is readily to verify since
	\[
	\frac{2(N-1)(p+1)-2b}{p+3}-1 = \frac{(2N-3)(p+1)-2-2b}{p+3}>0
	\]
	for $p>1+\frac{2b}{N-1}$ and $N\geq 2$.
	To check the fifth assumption, we compute
	\[
	\frac{2(N-1)(p+1)-2b}{p+3}-2 = \frac{2((N-2)(p+1)-2-b)}{p+3}.
	\]
	When $N=2$, it is negative. When $N\geq 3$, it is positive since $p+1>2+\frac{2b}{N-1}\geq \frac{2+b}{N-2}$. This shows that $\lim_{r \rightarrow 0} \gamma(r)=\infty$ if $N=2$ and $\lim_{r \rightarrow 0} \gamma(r)=0$ if $N\geq 3$, hence the fifth assumption is fulfilled. From the assumption on $p$, we see that $C>0$ for all $N\geq 2$ and $D<0$ if $N=2$ and $D>0$ if $N\geq 3$. Thus the sixth condition is satisfied for $k=0$ if $N=2$ and some $k>0$ if $N\geq 3$. Finally, the last assumption is also satisfied as $C>0$. Applying Theorem \ref{theo-SW}, we end the proof of Proposition \ref{prop-unique}.
	\end{proof}
\section{Local well-posedness}
	\label{S3}
	\setcounter{equation}{0}
	In this section, we study the local well-posedness for \eqref{INLS} with data in $H^1_{\rad}(\R^N)$. We have the following local existence in the energy-subcritical case.
	
	\begin{proposition} \label{prop-lwp}
		Let $N\geq 2$, $b>0$, $p\geq 1+\frac{2b}{N-1}$, and $p<\frac{N+2+2b}{N-2}$ if $N\geq 3$. Then \eqref{INLS} is locally well-posed in $H^1_{\rad}(\R^N)$ in the sense that: for every $u_0 \in H^1_{\rad}(\R^N)$,
		\begin{itemize}[leftmargin=5mm]
			\item There exist $T^*\in (0,\infty]$ and a unique solution $u \in C([0,T^*), H^1_{\rad}(\R^N))$ to \eqref{INLS} with initial data $\left.u\right|_{t=0} = u_0$.
			\item The maximal time of existence satisfies the blow-up alternative: if $T^*<\infty$, then $$\lim_{t \nearrow T^*} \|u(t)\|_{H^1} = \infty.$$
			\item If $u_{0,n} \rightarrow u_0$ strongly in $H^1_{\rad}(\R^N)$ and $0<T <T^*$, then the maximal solution $u_n$ of \eqref{INLS} with initial data $\left. u_n\right|_{t=0} =u_{0,n}$ is defined on $[0,T]$ for $n$ large and $u_n \rightarrow u$ in $C([0,T],H^1_{\rad}(\R^N))$ as $n\rightarrow \infty$.
		\end{itemize}
		Moreover, there are conservation laws of mass and energy, i.e.,
		\[
		M(u(t)) = M(u_0), \quad E(u(t)) = E(u_0)\quad \mbox{for all}\quad t\in [0,T^*).
		\]
		
	\end{proposition}
	
	\begin{remark}
		This result extends the one in \cite{CG-AM, CG-DCDS-B}, where the local well-posedness for \eqref{INLS} was stated with $N,b,p$ satisfying \eqref{cond-CG}.
	\end{remark}
	
	To prove Proposition \ref{prop-lwp}, we start with the following lemmas.
	
	\begin{lemma} \label{lem-lwp-1}
		Let $N\geq 2$, $b>0$, $p\geq 1+\frac{2b}{N-1}$, and $p<\frac{N+2}{N-2} + \frac{2b}{N-1}$ if $N\geq 3$. Then \eqref{INLS} is locally well-posed in $H^1_{\rad}(\R^N)$ in the sense of Proposition \ref{prop-lwp}. Moreover, there are conservation laws of mass and energy.
	\end{lemma}
	
	\begin{proof}
		We first show the existence of $H^1_{\rad}$-solution to \eqref{INLS}. To this end, we use the energy method of Cazenave (see \cite[Theorem 3.3.9]{Cazenave}). It suffices to find $\rho, r \in \left[2, \frac{2N}{N-2}\right)$ such that
		\begin{align} \label{est-rho-r}
		\||x|^b|u|^{p-1} u - |x|^b |v|^{p-1} v\|_{L^{\rho'}} \leq C(M) \|u-v\|_{L^r}
		\end{align}
		for all $u,v \in H^1_{\rad}(\R^N)$ satisfying $\|u\|_{H^1} + \|v\|_{H^1} \leq M$. By H\"older's inequality and
		the radial Sobolev inequality \eqref{est-Strauss}, we have
		\begin{align*}
		\||x|^b|u|^{p-1} u - |x|^b |v|^{p-1} v\|_{L^{\rho'}} &\lesssim \||x|^b (|u|^{p-1} + |v|^{p-1})|u-v|\|_{L^{\rho'}} \\
		&= \Big\|\Big(\Big( |x|^{\frac{N-1}{2}} |u|\Big)^{\frac{2b}{N-1}} |u|^{p-1-\frac{2b}{N-1}}  \\
		&\mathrel{\phantom{= \Big\|\Big(\Big( }}+ \Big( |x|^{\frac{N-1}{2}} |v|\Big)^{\frac{2b}{N-1}} |v|^{p-1-\frac{2b}{N-1}} \Big) |u-v|\Big\|_{L^{\rho'}} \\
		&\lesssim \Big( \|\nabla u\|^{\frac{b}{N-1}}_{L^2} \|u\|^{\frac{b}{N-1}}_{L^2} \|u\|^{\left(p-1-\frac{2b}{N-1}\right)n}_{L^{\left(p-1-\frac{2b}{N-1}\right)n}} \\
		&\mathrel{\phantom{\lesssim \Big\|\Big(\Big( }} +  \|\nabla v\|^{\frac{b}{N-1}}_{L^2} \|v\|^{\frac{b}{N-1}}_{L^2} \|v\|^{\left(p-1-\frac{2b}{N-1}\right)n}_{L^{\left(p-1-\frac{2b}{N-1}\right)n}} \Big) \|u-v\|_{L^r}
		\end{align*}
		provided that $n\geq 1$ satisfies
		\[
		\frac{1}{\rho'} = \frac{1}{n} +\frac{1}{r}.
		\]
		We now choose $\rho=r=p+1-\frac{2b}{N-1}$. We have $\left(p-1-\frac{2b}{N-1}\right)n =p+1-\frac{2b}{N-1}$ and
		\begin{align} \label{est-r-1}
		\||x|^b|u|^{p-1} u - |x|^b |v|^{p-1} v\|_{L^{r'}} \lesssim \Big(\|u\|^{\frac{2b}{N-1}}_{H^1} \|u\|^{p+1-\frac{2b}{N-1}}_{L^r} + \|v\|^{\frac{2b}{N-1}}_{H^1} \|v\|^{p+1-\frac{2b}{N-1}}_{L^r}\Big)\|u-v\|_{L^r}.
		\end{align}
		Since $r \in \left[2,\frac{2N}{N-2}\right)$, we infer from the standard Gagliardo-Nirenberg inequality that
		\begin{align*}
		\||x|^b|u|^{p-1} u - |x|^b |v|^{p-1} v\|_{L^{r'}} \lesssim \Big(\|u\|^{p-1}_{H^1} + \|v\|^{p-1}_{H^1} \Big) \|u-v\|_{L^r}.
		\end{align*}
		From this, we obtain \eqref{est-rho-r}. As a result, there exist $T^*\in (0,\infty]$ and a unique solution $u \in C([0,T^*), H^1_{\rad}(\R^N)) \cap C^1([0,T^*), H^{-1}_{\rad}(\R^N))$ to \eqref{INLS}, where $H^{-1}_{\rad}(\R^N)$ is the dual space of $H^1_{\rad}(\R^N)$. The uniqueness of $H^1_{\rad}$-solution follows from \eqref{est-rho-r} and \cite[Proposition 4.2.3]{Cazenave}. The proof is complete.	
	\end{proof}
	
	\begin{lemma} \label{lem-lwp-2}
		Let $N\geq 3$, $b>0$, $p\geq 1+\frac{2b}{N-2}$, and $p<\frac{N+2+2b}{N-2}$. Then \eqref{INLS} is locally well-posed in $H^1_{\rad}(\R^N)$ in the sense of Proposition \ref{prop-lwp}. Moreover, there are conservation laws of mass and energy.
	\end{lemma}
	
	\begin{proof}
		The proof is very similar to that of Lemma \ref{lem-lwp-1}, however, we use the radial Sobolev embedding \eqref{est-BL} instead of \eqref{est-Strauss}. In particular, we have
		\begin{align*}
		\||x|^b|u|^{p-1} u - |x|^b |v|^{p-1} v\|_{L^{\rho'}} &\lesssim \||x|^b (|u|^{p-1} + |v|^{p-1})|u-v|\|_{L^{\rho'}} \\
		&= \Big\|\Big(\Big( |x|^{\frac{N-2}{2}} |u|\Big)^{\frac{2b}{N-2}} |u|^{p-1-\frac{2b}{N-2}}  \\
		&\mathrel{\phantom{= \Big\|\Big(\Big( }}+ \Big( |x|^{\frac{N-2}{2}} |v|\Big)^{\frac{2b}{N-2}} |v|^{p-1-\frac{2b}{N-2}} \Big) |u-v|\Big\|_{L^{\rho'}} \\
		&\lesssim \Big( \|\nabla u\|^{\frac{2b}{N-2}}_{L^2} \|u\|^{\left(p-1-\frac{2b}{N-2}\right)n}_{L^{\left(p-1-\frac{2b}{N-2}\right)n}} \\
		&\mathrel{\phantom{\lesssim \Big\|\Big(\Big( }} +  \|\nabla v\|^{\frac{2b}{N-2}}_{L^2}  \|v\|^{\left(p-1-\frac{2b}{N-2}\right)n}_{L^{\left(p-1-\frac{2b}{N-2}\right)n}} \Big) \|u-v\|_{L^r}
		\end{align*}
		provided that $n\geq 1$ satisfies
		\[
		\frac{1}{\rho'} = \frac{1}{n} +\frac{1}{r}.
		\]
		We choose $\rho=r=p+1-\frac{2b}{N-2}$ to get
		\begin{align} \label{est-r-2}
		\||x|^b|u|^{p-1} u - |x|^b |v|^{p-1} v\|_{L^{r'}} \lesssim \Big(\|u\|^{\frac{2b}{N-2}}_{H^1} \|u\|^{p+1-\frac{2b}{N-2}}_{L^r} + \|v\|^{\frac{2b}{N-2}}_{H^1} \|v\|^{p+1-\frac{2b}{N-2}}_{L^r}\Big)\|u-v\|_{L^r}.
		\end{align}
		By the standard Gagliardo-Nirenberg inequality with $r \in \left[2,\frac{2N}{N-2}\right)$, we have
		\begin{align*}
		\||x|^b|u|^{p-1} u - |x|^b |v|^{p-1} v\|_{L^{r'}} \lesssim \Big(\|u\|^{p-1}_{H^1} + \|v\|^{p-1}_{H^1} \Big) \|u-v\|_{L^r}.
		\end{align*}
		The proof then follows by the same argument as in Lemma \ref{lem-lwp-1}.
	\end{proof}

	We are now able to prove Proposition \ref{prop-lwp}.
	
	\begin{proof}[Proof of Proposition \ref{prop-lwp}]
		From Lemma \ref{lem-lwp-1}, we have the local well-posedness (LWP) for $N= 2$, $b>0$, and $p\geq 1+2b$. Let us consider the case $N\geq 3$. By Lemmas \ref{lem-lwp-1} and \ref{lem-lwp-2}, we have the LWP for
		\[
		N\geq 3, \quad 0<b\leq 2(N-1), \quad 1+\frac{2b}{N-1} \leq p<\frac{N+2+2b}{N-2},
		\]
		and
		\[
		N\geq 3, \quad b> 2(N-1), \quad 1+\frac{2b}{N-1}\leq p<\frac{N+2}{N-2}+\frac{2b}{N-1} \quad \text{or} \quad 1+\frac{2b}{N-2}\leq p<\frac{N+2+2b}{N-2}.
		\]
		
		\begin{center}
			\begin{figure}[H]
				\begin{tikzpicture}
				\draw [-] (-6,0) -- (-5,0);
				\draw [->] (5,0) -- (7,0);
				\draw [blue,fill] (-5,0) circle [radius=1pt] node[below] {$1+\frac{2b}{N-1}$};
				\draw [red,fill] (-2,0) circle [radius=1pt] node[below] {$1+\frac{2b}{N-2}$};
				\draw [blue,fill] (2,0) circle [radius=1pt] node[below] {$\frac{N+2}{N-2} +\frac{2b}{N-1}$};
				\draw [red,fill] (5,0) circle [radius=1pt] node[below] {$\frac{N+2+2b}{N-2}$};
				\draw [thick, blue] (-5,0.01) -- (2,0.01);
				\draw (-4, 0.04) node[above] {{\color{blue} LWP}};
				\draw [thick, red] (-2,-0.01) -- (5,-0.01);
				\draw (4, 0.04) node[above] {{\color{red} LWP}};
				\end{tikzpicture}
				\caption{LWP for $N\geq3$ and $0<b\leq 2(N-1)$}
			\end{figure}
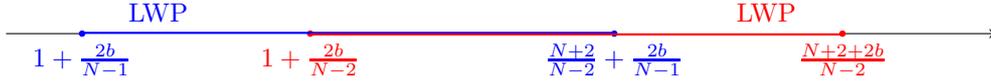
		\end{center}
		
		\begin{center}
			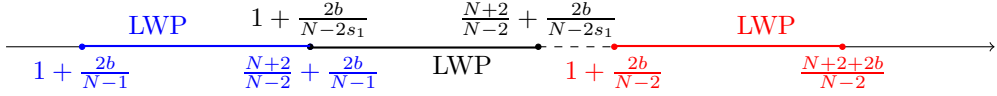
\begin{figure}[H]
				\begin{tikzpicture}
				\draw [-] (-6,0) -- (-5,0);
				\draw [->] (5,0) -- (7, 0);
				\draw [dashed] (1,0) -- (2,0);
				\draw [blue,fill] (-5,0) circle [radius=1pt] node[below] {$1+\frac{2b}{N-1}$};
				\draw [blue,fill] (-2,0) circle [radius=1pt] node[below] {$\frac{N+2}{N-2} +\frac{2b}{N-1}$};
				\draw [red,fill] (2,0) circle [radius=1pt] node[below] {$1+\frac{2b}{N-2}$};
				\draw [red,fill] (5,0) circle [radius=1pt] node[below] {$\frac{N+2+2b}{N-2}$};
				\draw [fill] (-2,0) circle [radius=1pt] node[above] {$1+\frac{2b}{N-2s_1}$};
				\draw [fill] (1,0) circle [radius=1pt] node[above] {$\frac{N+2}{N-2}+\frac{2b}{N-2s_1}$};
				\draw [thick, blue] (-5,0.01) -- (-2,0.01);
				\draw (-4, 0.04) node[above] {{\color{blue} LWP}};
				\draw [thick] (-2,-0.01) -- (1,-0.01);
				\draw (0, 0) node[below] {LWP};
				\draw [thick, red] (2,0.01) -- (5,0.01);
				\draw (4, 0.04) node[above] {{\color{red} LWP}};
				\end{tikzpicture}
				\caption{LWP for $N\geq3$ and $b> 2(N-1)$}
			\end{figure}
		\end{center}
		
		To fill the gap $p \in \left[\frac{N+2}{N-2}+\frac{2b}{N-1},1+\frac{2b}{N-2} \right)$ when $b>2(N-1)$, we proceed as in the proof of Lemma \ref{lem-regu}. We take $\frac{1}{2}<s_1<1$ such that $\frac{N+2}{N-2} +\frac{2b}{N-1} = 1+\frac{2b}{N-2s_1}$. Repeating the same argument as in the proofs of Lemmas \ref{lem-lwp-1} and \ref{lem-lwp-2} with \eqref{est-CO} instead of \eqref{est-Strauss} or \eqref{est-BL}, we can prove the LWP for \eqref{INLS} with
		\[
		1+\frac{2b}{N-2s_1} \leq p<\frac{N+2}{N-2} +\frac{2b}{N-2s_1}.
		\]
		If $2(N-1)<b\leq \frac{N-2s_1}{1-s_1}$, we have $\frac{N+2}{N-2} + \frac{2b}{N-2s_1} \geq 1+\frac{2b}{N-2}$. Thus we get the LWP for all $1+\frac{2b}{N-1} \leq p<\frac{N+2+2b}{N-2}$. Otherwise, if $b>\frac{N-2s_1}{1-s_1}$, we take $s_2 \in (s_1,1)$ so that $\frac{N+2}{N-2} + \frac{2b}{N-2s_1} =1+\frac{2b}{N-2s_2}$. Repeating the same reasoning, we prove the LWP for \eqref{INLS} with
		\[
		1+\frac{2b}{N-2s_2} \leq p<\frac{N+2}{N-2} +\frac{2b}{N-2s_2}.
		\]
		If $\frac{N-2s_1}{1-s_1}<b\leq \frac{N-2s_2}{1-s_2}$, we are done. Otherwise, we repeat the above argument until $\frac{N-2s_{k-1}}{1-s_{k-1}} < b \leq \frac{N-2s_k}{1-s_k}$. Since $s_k \rightarrow 1$ and $b>0$ is given, the above process will end in finite steps. The proof is complete.
	\end{proof}
	
	The above argument using the energy method does not work in the energy-critical case $p=\frac{N+2+2b}{N-2}$. Indeed, in such a case, there is no real numbers $\rho,r\in [2,\frac{2N}{N-2})$ satisfying \eqref{est-rho-r}. In this case, we will show the local well-posedness by using Strichartz estimates and the contraction mapping argument. Let us start by recalling the notion of Schr\"odinger admissible pairs.
	
	
	\begin{definition}
		A pair $(q,r)$ is said to be Schr\"odinger admissible if
		\[
		\frac{2}{q}+\frac{N}{r}=\frac{N}{2}, \quad \left\{
		\renewcommand*{\arraystretch}{1.2}
		\begin{array}{ll}
		r \in \left[2,\frac{2N}{N-2}\right] &\text{if } N\geq 3, \\
		r \in [2, \infty) &\text{if } N=2, \\
		r\in [2, \infty] &\text{if } N=1.
		\end{array}
		\right.
		\]
	\end{definition}
	
	\begin{proposition}[Strichartz estimates \cite{Cazenave}]
		Let $N\geq 1$ and $(q,r), (m,n)$ Schr\"odinger admissible pairs. Then there exists a constant $C>0$ such that
		\[
		\|e^{it\Delta} f\|_{L^q(\R,L^r)} \leq C\|f\|_{L^2}
		\]
		for any $f \in L^2(\R^N)$. Moreover, for any interval $I\subset \R$ containing $0$, there exists $C>0$ independent of $I$ such that
		\[
		\left\| \int_0^t e^{i(t-s)\Delta} F(s) ds \right\|_{L^q(I,L^r)} \leq C \|F\|_{L^{m'}(I,L^{n'})}
		\]
		for any $F \in L^{m'}(I, L^{n'}(\R^N))$.
	\end{proposition}
	
	We have the following local well-posedness in the energy-critical case.
	
	\begin{proposition} \label{prop-lwp-ener}
		Let $N\geq 3$, $b>0$, and $p=\frac{N+2+2b}{N-2}$. Denote
		\[
		q:=\frac{2N}{N-2}, \quad r:=\frac{2N^2}{N^2-2N+4}.
		\]
		Then for every $u_0 \in H^1_{\rad}(\R^N)$, there exist $T^*\in (0,\infty]$ and a unique solution
		\[
		u \in C([0,T^*), H^1_{\rad}(\R^N)) \cap L^q_{\loc}([0,T^*), W^{1,r}_{\rad}(\R^N))
		\]
		to \eqref{INLS} with initial data $\left. u \right|_{t=0} = u_0$. The maximal time of existence satisfies the blow-up alternative: if $T^*<\infty$, then $\|u\|_{L^p([0,T^*), W^{1,r})} =\infty$. Moreover, there are conservation laws of mass and energy. In addition, if $\|u_0\|_{H^1} <\vareps$ for some $\vareps>0$ sufficiently small, then $T^*=\infty$ and the solution scatters in $H^1_{\rad}(\R^N)$.
	\end{proposition}
	
	\begin{proof}
		Let us consider
		\[
		Y(I):= \left\{ u \in C(I,H^1_{\rad}(\R^N)) \cap L^q(I, W^{1,r}_{\rad}(\R^N)) : \|u\|_{L^\infty(I, H^1)} \leq L, \|u\|_{L^q(I, W^{1,r})} \leq M \right\}
		\]
		equipped with the distance
		\[
		d(u,v):= \|u-v\|_{L^q(I,L^r)},
		\]
		where $I=[0,T]$ and $L, M,T>0$ to be chosen later. By the Duhamel formula, it suffices to show that the functional
		\[
		\Phi(u)(t):= e^{it\Delta} u_0 + i \int_0^t e^{i(t-s)\Delta} |x|^b |u(s)|^{\frac{4+2b}{N-2}} u(s) ds=: u_{\hom}(t) + u_{\inh}(t)
		\]
		is a contraction on the complete metric space $(Y(I),d)$. Since $(q,r)$ is a Schr\"odinger admissible pair, we have from Strichartz estimates that
		\[
		\|u_{\hom}\|_{L^\infty(I, H^1) \cap L^q(I,W^{1,r})} \lesssim \|u_0\|_{H^1}.
		\]
		In particular, we have $\|u_{\hom}\|_{L^q(I,W^{1,r})} \leq \vareps$ for some $\vareps>0$ sufficiently small provided that $T>0$ is taken small enough or $\|u_0\|_{H^1}$ is small. In the second case, we can take $I=[0,\infty)$. Also, by Strichartz estimates, we have
		\begin{align*}
		\|u_{\inh}\|_{L^\infty(I,H^1) \cap L^q(I,W^{1,r})} &\lesssim \|\scal{\nabla} (|x|^b|u|^{\frac{4+2b}{N-2}} u)\|_{L^{q'}(I,L^{r'})} \\
		&\sim \||x|^b |u|^{\frac{4+2b}{N-2}} u\|_{L^{q'}(I,L^{r'})} + \|\nabla(|x|^b |u|^{\frac{4+2b}{N-2}} u)\|_{L^{q'}(I,L^{r'})}.
		\end{align*}
		Let us estimate the second term. By the Leibniz's rule and the chain rule, we have
		\[
		\nabla(|x|^b|u|^{\frac{4+2b}{N-2}} u) \sim |x|^b |u|^{\frac{4+2b}{N-2}} \nabla u + |x|^b |u|^{\frac{4+2b}{N-2}} (|x|^{-1} u).
		\]
		By \eqref{est-BL}, H\"older's inequality, and Sobolev embedding, we have
		\begin{align*}
		\||x|^b |u|^{\frac{4+2b}{N-2}} \nabla u\|_{L^{r'}} &= \left\|\left( |x|^{\frac{N-2}{2}} |u|\right)^{\frac{2b}{N-2}} |u|^{\frac{4}{N-2}} \nabla u\right\|_{L^{r'}} \\
		&\lesssim \|\nabla u\|^{\frac{2b}{N-2}} \left\||u|^{\frac{4}{N-2}} \nabla u\right\|_{L^{r'}} \\
		&\lesssim \|\nabla u\|^{\frac{2b}{N-2}} \|u\|^{\frac{4}{N-2}}_{L^n} \|\nabla u\|_{L^r} \\
		&\lesssim \|\nabla u\|^{\frac{2b}{N-2}}_{L^2} \|\nabla u\|^{\frac{N+2}{N-2}}_{L^r},
		\end{align*}
		where $n=\frac{2N^2}{N^2-4N+4}$. It follows that
		\[
		\||x|^b |u|^{\frac{4+2b}{N-2}} \nabla u\|_{L^{q'}(I,L^{r'})} \lesssim \|\nabla u\|^{\frac{2b}{N-2}}_{L^\infty(I,L^2)} \|\nabla u\|_{L^{q}(I,L^r)}^{\frac{N+2}{N-2}}.
		\]
		By the Hardy's inequality (see e.g., \cite{OK}): for $N\geq 2$ and $1<r<N$,
		\begin{align} \label{Hardy-ineq}
		\||x|^{-1} f\|_{L^r}\leq \frac{r}{N-r} \|\nabla f\|_{L^r},
		\end{align}
		we have
		\begin{align*}
		\||x|^b |u|^{\frac{4+2b}{N-2}} |x|^{-1} u\|_{L^{r'}} &= \left\|\left( |x|^{\frac{N-2}{2}} |u|\right)^{\frac{2b}{N-2}} |u|^{\frac{4}{N-2}} |x|^{-1} u\right\|_{L^{r'}} \\
		&\lesssim \|\nabla u\|^{\frac{2b}{N-2}} \left\||u|^{\frac{4}{N-2}} |x|^{-1} u\right\|_{L^{r'}} \\
		&\lesssim \|\nabla u\|^{\frac{2b}{N-2}} \|u\|^{\frac{4}{N-2}}_{L^n} \||x|^{-1} u\|_{L^r} \\
		&\lesssim \|\nabla u\|^{\frac{2b}{N-2}}_{L^2} \|\nabla u\|^{\frac{N+2}{N-2}}_{L^r},
		\end{align*}
		where we readily check that $r=\frac{2N^2}{N^2-2N+4} \in (1,N)$. Thus we get
		\[
		\||x|^b |u|^{\frac{4+2b}{N-2}} |x|^{-1} u\|_{L^{q'}(I,L^{r'})} \lesssim \|\nabla u\|^{\frac{2b}{N-2}}_{L^\infty(I,L^2)} \|\nabla u\|_{L^{q}(I,L^r)}^{\frac{N+2}{N-2}},
		\]
		hence
		\[
		\|\nabla(|x|^b |u|^{\frac{4+2b}{N-2}} u)\|_{L^{q'}(I,L^{r'})} \lesssim \|\nabla u\|^{\frac{2b}{N-2}}_{L^\infty(I,L^2)} \|\nabla u\|_{L^{q}(I,L^r)}^{\frac{N+2}{N-2}}.
		\]
		A similar estimate goes for $\||x|^b |u|^{\frac{4+2b}{N-2}} u\|_{L^{q'}(I,L^{r'})}$ and we obtain
		\[
		\|u_{\inh}\|_{L^\infty(I,H^1)\cap L^q(I,W^{1,r})} \lesssim \|u\|^{\frac{2b}{N-2}}_{L^\infty(I,H^1)} \|u\|^{\frac{N+2}{N-2}}_{L^q(I,W^{1,r})}.
		\]
		On the other hand, we have
		\begin{align*}
		\|\Phi(u)-\Phi(v)\|_{L^q(I,L^r)} &\lesssim \||x|^b(|u|^{p-1} u - |v|^{p-1} v)\|_{L^{q'}(I,L^{r'})} \\
		&\lesssim \left(\|\nabla u\|^{\frac{2b}{N-2}}_{L^\infty(I,L^2)} \|\nabla u\|^{\frac{4}{N-2}}_{L^q(I,L^r)} + \|\nabla v\|^{\frac{2b}{N-2}}_{L^\infty(I,L^2)} \|\nabla v\|^{\frac{4}{N-2}}_{L^q(I,L^r)} \right) \|u-v\|_{L^q(I,L^r)}.
		\end{align*}
		Hence
		\[
		\|\Phi(u)-\Phi(v)\|_{L^q(I,L^r)} \lesssim \left(\| u\|^{\frac{2b}{N-2}}_{L^\infty(I,H^1)} \|u\|^{\frac{4}{N-2}}_{L^q(I,W^{1,r})} + \|v\|^{\frac{2b}{N-2}}_{L^\infty(I,H^1)} \|v\|^{\frac{4}{N-2}}_{L^q(I,W^{1,r})} \right) \|u-v\|_{L^q(I,L^r)}.
		\]
		This shows that there exists $C>0$ such that for all $u,v \in Y(I)$,
		\begin{align*}
		\|\Phi(u)\|_{L^\infty(I,H^1)} &\leq C\|u_0\|_{H^1} + C L^{\frac{2b}{N-2}} M^{\frac{N+2}{N-2}}, \\
		\|\Phi(u)\|_{L^q(I,W^{1,r})} &\leq \vareps + C L^{\frac{2b}{N-2}} M^{\frac{N+2}{N-2}}, \\
		d(\Phi(u), \Phi(v)) &\leq C L^{\frac{2b}{N-2}} M^{\frac{4}{N-2}} d(u,v).
		\end{align*}
		We take $L=2C\|u_0\|_{H^1}$. We then choose $M>0$ and $\vareps>0$ small so that
		\[
		CL^{\frac{2b}{N-2}} M^{\frac{4}{N-2}} \leq \frac{1}{2}, \quad M \leq C\|u_0\|_{H^1}, \quad \vareps + \frac{M}{2} \leq M.
		\]
		This shows that $\Phi$ is a contraction on $(Y(I),d)$ and the existence of solutions to \eqref{INLS} follows. The blow-up alternative comes from standard argument (see e.g., \cite{CW, Cazenave}). We thus omit the details.
		
		It remains to show the scattering for $\|u_0\|_{H^1}$ sufficiently small. Note that the solution exists globally in this case and
		\[
		\|u\|_{L^q([0,\infty),W^{1,r})} <\infty.
		\]
		Now let $0<t_1<t_2$. We have
		\begin{align*}
		\|e^{-it_2\Delta} u(t_2) - e^{-it_1\Delta} u(t_1)\|_{H^1} &=\left\| i \int_{t_1}^{t_2} e^{-is \Delta} |x|^b |u(s)|^{\frac{4+2b}{N-2}} u(s) ds \right\|_{H^1} \\
		&= \left\| i \int_{t_1}^{t_2} \scal{\nabla} (|x|^b |u(s)|^{\frac{4+2b}{N-2}} u(s))\right\|_{L^2} \\
		&\lesssim \|\scal{\nabla}(|x|^b|u|^{\frac{4+2b}{N-2}} u)\|_{L^{q'}((t_1,t_2),L^{r'})} \\
		&\lesssim \|u\|_{L^\infty((t_1,t_2),H^1)}^{\frac{2b}{N-2}} \|u\|^{\frac{N+2}{N-2}}_{L^q((t_1,t_2),W^{1,r})} \rightarrow 0
		\end{align*}
		as $t_1, t_2 \rightarrow \infty$. This shows that $(e^{-it\Delta} u(t))_{t\rightarrow \infty}$ is a Cauchy sequence in $H^1$, hence the limit
		\[
		u^+:=\lim_{t\rightarrow \infty} e^{-it\Delta} u(t)
		\]
		exists in $H^1_{\rad}(\R^N)$. We also have
		\[
		u(t)-e^{it\Delta} u^+ = -i\int_t^\infty e^{i(t-s)\Delta} |x|^b |u(s)|^{\frac{4+2b}{N-2}} u(s) ds.
		\]
		Repeating the above estimate, we prove as well that
		\[
		\|u(t)-e^{it\Delta} u^+\|_{H^1} \rightarrow 0 \text{ as } t \rightarrow \infty.
		\]
		The proof is complete.					
	\end{proof}

	\section{Global existence and energy scattering}
	\label{S4}
	\setcounter{equation}{0}
	
	This section is devoted to the proof of the global existence and energy scattering given in Proposition \ref{prop-gwp} and Theorem \ref{theo-scat-inte}.
	
	\subsection{Strichartz estimates for non Schr\"odinger admissible pairs}
	We recall the following inhomogeneous Strichartz estimates for non Schr\"odinger admissible pairs which play an essential role in the proof of energy scattering.
	
	\begin{lemma} [\cite{CW-CMP}]
		Let $N\geq 1$ and $I\subset \R$ be an interval containing 0. Let $(q,r)$ be a Schr\"odinger admissible pair with $r>2$. Fix $k>\frac{q}{2}$ and define $m$ by
		\begin{align} \label{cond-km}
		\frac{1}{k}+\frac{1}{m} = \frac{2}{q}.
		\end{align}
		Then there exist a constant $C>0$ such that
		\begin{align} \label{stri-est-non-adm}
		\left\|\int_0^t e^{i(t-s)\Delta} F(s) ds \right\|_{L^k(I,L^r)} \leq C \|F\|_{L^{m'}(I,L^{r'})}
		\end{align}
		for any $F \in L^{m'}(I,L^{r'})$.
	\end{lemma}

	\subsection{Properties of ground states}
	Let $Q$ be the unique positive radial solution to \eqref{equ-Q}. Since $Q$ is an optimizer for the Gagliardo-Nirenberg inequality: for $N\geq 2$, $b>0$, $p>1+\frac{2b}{N-1}$, and $p<\frac{N+2+2b}{N-2}$ if $N\geq 3$,
	\begin{align} \label{GN-opti}
	\int |x|^b |f(x)|^{p+1} dx \leq C_{\opt} \|\nabla f\|^{\frac{N(p-1)-2b}{2}}_{L^2} \|f\|^{\frac{4+2b-(N-2)(p-1)}{2}}_{L^2}, \quad f \in H^1_{\rad}(\R^N),
	\end{align}
	we have
	\[
	C_{\opt} = \int |x|^b |Q(x)|^{p+1}dx \div \left[\|\nabla Q\|_{L^2}^{\frac{N(p-1)-2b}{2}} \|Q\|^{\frac{4+2b-(N-2)(p-1)}{2}}_{L^2} \right].
	\]
	Thanks to the Pohozaev's identity \eqref{poho-iden}, we see that
	\begin{align}
	E(Q) &= \frac{1}{2}\|\nabla Q\|^2_{L^2}-\frac{1}{p+1}\int |x|^b |Q(x)|^{p+1}dx \nonumber \\
	&= \frac{N(p-1)-4-2b}{2(N(p-1)-2b)} \|\nabla Q\|^2_{L^2} = \frac{N(p-1)-4-2b}{4(p+1)} \int |x|^b |Q(x)|^{p+1}dx \label{E-Q}
	\end{align}
	and
	\begin{align} \label{opti-cons}
	C_{\opt} = \frac{2(p+1)}{N(p-1)-2b} \left( \|\nabla Q\|_{L^2} \|Q\|^{\sigc}_{L^2}\right)^{-\frac{N(p-1)-4-2b}{2}},
	\end{align}
	where $\sigc$ is as in \eqref{sigc}.

	\subsection{Morawetz estimates}
	In this subsection, we derive some Morawetz estimates related to \eqref{INLS}. Let us start with the following virial identity (see e.g., \cite{Cazenave})
	
	\begin{lemma} [Virial identity] \label{lem-viri-iden}
		Let $N\geq 2$, $b>0$, $p\geq 1+\frac{2b}{N-1}$, and $p\leq \frac{N+2+2b}{N-2}$ if $N\geq 3$. Let $\varphi: \R^N \rightarrow \R$ be a sufficiently smooth and decaying function. Let $u \in C([0,T^*), H^1_{\rad}(\R^N))$ be a solution to \eqref{INLS}. Define
		\begin{align} \label{V-varphi}
		V_{\varphi}(t) := \int \varphi(x)|u(t,x)|^2 dx.
		\end{align}
		Then we have
		\[
		V'_{\varphi}(t) =  2 \ima \int \nabla \varphi(x) \cdot \nabla u(t,x) \overline{u}(t,x) dx
		\]
		and
		\begin{align*}
		V''_{\varphi}(t) &= - \int \Delta^2 \varphi(x) |u(t,x)|^2 dx + 4 \rea \sum_{j,k=1}^N \int \partial^2_{jk} \varphi(x) \partial_j\overline{u}(t,x) \partial_k u(t,x) dx \\
		&\mathrel{\phantom{=}} - \frac{2(p-1)}{p+1} \int \Delta \varphi(x) |x|^b |u(t,x)|^{p+1} dx + \frac{4}{p+1}\int \nabla \varphi(x) \cdot \nabla(|x|^b)|u(t,x)|^{p+1}dx.
		\end{align*}
	\end{lemma}
	
	Let $\zeta: [0,\infty) \rightarrow [0,2]$ be a smooth function satisfying
	\begin{align} \label{zeta}
	\zeta(r) = \left\{
	\begin{array}{ccl}
	2 & \text{if} & 0\leq r \leq 1, \\
	0 &\text{if} & r\geq 2.
	\end{array}
	\right.
	\end{align}
	We define the function $\theta: [0,\infty) \rightarrow [0,\infty)$ by
	\[
	\theta(r):= \int_0^r \int_0^s \zeta(z)dz ds.
	\]
	Given $R>0$, we define a radial function
	\begin{align} \label{varphi-R}
	\varphi_R(x) = \varphi_R(r):= R^2 \theta(r/R), \quad r=|x|.
	\end{align}
	It is easy to check that
	\begin{align*}
	2 \geq \varphi''_R(r) \geq 0, \quad 2- \frac{\varphi'_R(r)}{r} \geq 0, \quad 2N - \Delta \varphi_R(x) \geq 0, \quad \forall r\geq 0, \quad \forall x \in \R^N.
	\end{align*}
	
	\begin{lemma}[Coercivity] \label{lem-coer}
		Let $N\geq 2$, $b>0$, $p>\max\left\{\frac{N+4+2b}{N}, 1+\frac{2b}{N-1}\right\}$, and $p<\frac{N+2+2b}{N-2}$ if $N\geq 3$. Let $u_0 \in H^1_{\rad}(\R^N)$ satisfy \eqref{cond-scat-inte}. Then the corresponding solution to \eqref{INLS} with initial data $\left.u\right|_{t=0}=u_0$ satisfies
		\begin{align} \label{est-solu-scat}
		\|\nabla u(t)\|_{L^2}\|u(t)\|^{\sigc}_{L^2}< \|\nabla Q\|_{L^2}\|Q\|^{\sigc}_{L^2}
		\end{align}
		for all $t\in [0,T^*)$. In particular, the solution exists globally in time, i.e., $T^*=\infty$. Moreover, there exist $\delta=\delta(u_0,Q)>0$ and $R_0=R_0(u_0,Q)>0$ such that for all $R\geq R_0$,
		\begin{align} \label{coer-est}
		\int |\nabla(\chi_R u(t,x))|^2 dx - \frac{N(p-1)-2b}{2(p+1)}\int |x|^b |\chi_R u(t,x)|^{p+1}dx \geq \delta \int |x|^b|\chi_R u(t,x)|^{p+1} dx
		\end{align}
		for all $t\in [0,\infty)$, where $\chi_R(x)=\chi(x/R)$ with $\chi \in C^\infty_0(\R^N)$ satisfying $0\leq \chi \leq 1$,
		\begin{align} \label{chi}
		\chi(x) =\left\{
		\begin{array}{cc l}
		1 &\text{if} & |x| \leq 1/2, \\
		0 &\text{if} & |x| \geq 1.
		\end{array}
		\right.
		\end{align}
	\end{lemma}
	
	\begin{proof}
		Multiplying both side of the energy functional by $(M(u(t))^{\sigc}$, it follows from the sharp Gagliardo-Nirenberg inequality \eqref{GN-opti} that
		\begin{align}
		E(u(t)) (M(u(t)))^{\sigc} &= \frac{1}{2} \left( \|\nabla u(t)\|_{L^2} \|u(t)\|^{\sigc}_{L^2}\right)^2 - \frac{1}{p+1} \left(\int |x|^b |u(t,x)|^{p+1} dx\right) \|u(t)\|^{2\sigc}_{L^2} \nonumber \\
		&\geq \frac{1}{2} \left( \|\nabla u(t)\|_{L^2} \|u(t)\|^{\sigc}_{L^2}\right)^2 - \frac{C_{\opt}}{p+1} \|\nabla u(t)\|^{\frac{N(p-1)-2b}{2}}_{L^2} \|u(t)\|^{\frac{4+2b-(N-2)(p-1)}{2} +2{\sigc}}_{L^2} \nonumber\\
		& = F\left(\|\nabla u(t)\|_{L^2} \|u(t)\|^{\sigc}_{L^2} \right), \label{est-E}
		\end{align}
		where
		\[
		F(\lambda) = \frac{1}{2} \lambda^2 - \frac{C_{\opt}}{p+1} \lambda^{\frac{N(p-1)-2b}{2}}.
		\]
		Using \eqref{E-Q} and \eqref{opti-cons}, we have
		\[
		F \left( \|\nabla Q\|_{L^2} \|Q\|^{\sigc}_{L^2} \right) = \frac{N(p-1)-4-2b}{2(N(p-1)-2b)} \left( \|\nabla Q\|_{L^2} \|Q\|^{\sigc}_{L^2}  \right)^2 = E(Q) (M(Q))^{\sigc}.
		\]
		From the first condition in \eqref{cond-scat-inte} and the conservation laws of mass and energy, we deduce
		\[
		F\left(\|\nabla u(t)\|_{L^2} \|u(t)\|^{\sigc}_{L^2} \right) \leq E(u_0) (M(u_0))^{\sigc} < E(Q) (M(Q))^{\sigc} = F \left( \|\nabla Q\|_{L^2} \|Q\|^{\sigc}_{L^2} \right)
		\]
		for all $t \in [0,T^*)$. By the second condition in \eqref{cond-scat-inte}, the continuity argument implies that
		\[
		\|\nabla u(t)\|_{L^2} \|u(t)\|^{\sigc}_{L^2} < \|\nabla Q\|_{L^2} \|Q\|^{\sigc}_{L^2}
		\]
		for all $t\in [0,T^*)$. This shows \eqref{est-solu-scat}. By the conservation of mass, the blow-up alternative implies that the solution exists globally in time, i.e., $T^*=\infty$. We thus prove Proposition \ref{prop-gwp}.
		
		To see \eqref{coer-est}, we take $\vartheta = \vartheta(u_0,Q)>0$ so that
		\begin{align} \label{coer-est-prof-1}
		E(u_0) (M(u_0))^{\sigc} \leq (1-\vartheta) E(Q) (M(Q))^{\sigc}.
		\end{align}
		As
		\begin{align*}
		E(Q)(M(Q))^{\sigc} &=\frac{N(p-1)-4-2b}{2(N(p-1)-2b)} \left(\|\nabla Q\|_{L^2} \|Q\|^{\sigc}_{L^2} \right)^2 \\
		&= \frac{N(p-1)-4-2b}{4(p+1)} C_{\opt} \left(\|\nabla Q\|_{L^2} \|Q\|^{\sigc}_{L^2} \right)^{\frac{N(p-1)-2b}{2}},
		\end{align*}
		we infer from \eqref{est-E} and \eqref{coer-est-prof-1} that
		\begin{multline} \label{est-G}
		\frac{N(p-1)-2b}{N(p-1)-4-2b} \left(\frac{\|\nabla u(t)\|_{L^2} \|u(t)\|^{\sigc}_{L^2}}{\|\nabla Q\|_{L^2} \|Q\|^{\sigc}_{L^2}} \right)^2  \\ - \frac{4}{N(p-1)-4-2b} \left(\frac{\|\nabla u(t)\|_{L^2} \|u(t)\|^{\sigc}_{L^2}}{\|\nabla Q\|_{L^2} \|Q\|^{\sigc}_{L^2}} \right)^{\frac{N(p-1)-2b}{2}} \leq 1-\vartheta.
		\end{multline}
		Consider the function
		\[
		G(\lambda) := \frac{N(p-1)-2b}{N(p-1)-4-2b} \lambda^2 - \frac{4}{N(p-1)-4-2b} \lambda^{\frac{N(p-1)-2b}{2}}, \quad 0<\lambda<1.
		\]
		It is easy to see that $G$ is strictly increasing in $(0,1)$ with $G(0) = 0$ and $G(1) = 1$. It follows from \eqref{est-G} that there exists $\rho>0$ depending on $\vartheta$ such that 
		\[
		G(\lambda)\leq 1-\vartheta\Longrightarrow \lambda<1-2\rho,
		\] 
		which shows
		\begin{align} \label{coer-est-prof-2}
		\|\nabla u(t)\|_{L^2} \|u(t)\|^{\sigc}_{L^2} < (1-2\rho)\|\nabla Q\|_{L^2} \|Q\|^{\sigc}_{L^2}
		\end{align}
		for all $t\in [0,\infty)$.
		
		By the definition of $\chi_R$, we have $\|\chi_R u(t)\|_{L^2} \leq \|u(t)\|_{L^2}$. On the other hand, using
		\begin{align*}
		\int |\nabla (\chi f)|^2 dx 
		= \int \chi^2 |\nabla f|^2 dx - \int \chi \Delta \chi |f|^2 dx,
		\end{align*}
		we also have
		\[
		\|\nabla (\chi_R u(t))\|^2_{L^2} \leq \|\nabla u(t)\|^2_{L^2} + O\left(R^{-2} \|u(t)\|^2_{L^2} \right).
		\]
		Thus
		\begin{align*}
		\|\nabla(\chi_R u(t))\|_{L^2} \|\chi_R u(t)\|^{\sigc}_{L^2} &\leq \left( \|\nabla u(t)\|^2_{L^2} + O \left(R^{-2} \|u(t)\|^2_{L^2}\right)\right)^{\frac{1}{2}} \|u(t)\|^{\sigc}_{L^2} \\
		&\leq \|\nabla u(t)\|_{L^2} \|u(t)\|^{\sigc}_{L^2} + O \left(R^{-1} \|u(t)\|_{L^2}^{{\sigc}+1}\right) \\
		&<(1-2\rho)\|\nabla Q\|_{L^2} \|Q\|^{\sigc}_{L^2} + O\left(R^{-1} \|u_0\|_{L^2}^{{\sigc}+1}\right) \\
		&<(1-\rho)\|\nabla Q\|_{L^2} \|Q\|^{\sigc}_{L^2}
		\end{align*}
		provided that $R\geq R_0$, where $R_0>0$ is sufficiently large depending on $\rho, \|u_0\|_{L^2}$, hence on $u_0,Q$.
		
		Finally, \eqref{coer-est} follows from the following fact: if
		\begin{align} \label{coer-est-prof-3}
		\|\nabla f\|_{L^2} \|f\|_{L^2}^{\sigc} <(1-\rho) \|\nabla Q\|_{L^2} \|Q\|^{\sigc}_{L^2},
		\end{align}
		then there exists $\delta =\delta(\rho)>0$ such that
		\begin{align} \label{coer-est-prof-4}
		\|\nabla f\|^2_{L^2} - \frac{N(p-1)-2b}{2(p+1)} \int |x|^b |f(x)|^{p+1}dx \geq \delta \int |x|^b |f(x)|^{p+1}dx.
		\end{align}
		To see \eqref{coer-est-prof-4}, we first have from the Gagliardo-Nirenberg inequality, \eqref{poho-iden}, and \eqref{coer-est-prof-3} that
		\begin{align*}
		E(f) &\geq \frac{1}{2} \|\nabla f\|^2_{L^2} - \frac{C_{\opt}}{p+1} \|\nabla f\|^{\frac{N(p-1)-2b}{2}}_{L^2} \|f\|^{\frac{4+2b-(N-2)(p-1)}{2}}_{L^2} \\
		&= \frac{1}{2} \|\nabla f\|^2_{L^2} \left(1 - \frac{2C_{\opt}}{p+1} \|\nabla f\|^{\frac{N(p-1)-4-2b}{2}}_{L^2} \|f\|^{\frac{4+2b-(N-2)(p-1)}{2}} _{L^2}\right) \\
		&=\frac{1}{2} \|\nabla f\|^2_{L^2} \left(1 - \frac{2C_{\opt}}{p+1} \left(\|\nabla f\|_{L^2} \|f\|^{\sigc} _{L^2}\right)^{\frac{N(p-1)-4-2b}{2}}\right) \\
		&> \frac{1}{2} \|\nabla f\|^2_{L^2} \left(1 - \frac{2C_{\opt}}{p+1} (1-\rho)^{\frac{N(p-1)-4-2b}{2}} \left(\|\nabla Q\|_{L^2} \|Q\|^{\sigc} _{L^2} \right)^{\frac{N(p-1)-4-2b}{2}} \right) \\
		&=\frac{1}{2} \|\nabla f\|^2_{L^2} \left( 1- \frac{4}{N(p-1)-2b}(1-\rho)^{\frac{N(p-1)-4-2b}{2}} \right).
		\end{align*}
		This implies in particular that
		\[
		\|\nabla f\|^2_{L^2} \geq \frac{N(p-1)-2b}{2(p+1)(1-\rho)^{\frac{N(p-1)-4-2b}{2}}} \int |x|^b |f(x)|^{p+1}dx.
		\]
		We now set $K(f)$ the left hand side of \eqref{coer-est-prof-4}. We have that
		\begin{align*}
		K(f) &= \frac{N(p-1)-2b}{2} E(f) - \frac{N(p-1)-4-2b}{4} \|\nabla f\|^2_{L^2} \\
		&\geq \frac{N(p-1)-2b}{4} \|\nabla f\|^2_{L^2} \left(1-\frac{4}{N(p-1)-2b} (1-\rho)^{\frac{N(p-1)-4-2b}{2}} \right) - \frac{N(p-1)-4-2b}{4} \|\nabla f\|^2_{L^2} \\
		&=(1-(1-\rho)^{\frac{N(p-1)-4-2b}{2}}) \|\nabla f\|^2_{L^2} \\
		&\geq \frac{(N(p-1)-2b)(1-(1-\rho)^{\frac{N(p-1)-4-2b}{2}})}{2(p+1) (1-\rho)^{\frac{N(p-1)-4-2b}{2}}} \int |x|^b |f(x)|^{p+1}dx.
		\end{align*}
		This proves \eqref{coer-est-prof-4} and the proof is complete.
	\end{proof}
	
	\begin{lemma} [Morawetz estimate] \label{lem-mora-esti}
		Let $N\geq 2$, $b>0$, $p>\max\left\{\frac{N+4+2b}{N}, 1+\frac{2b}{N-1}\right\}$, and $p<\frac{N+2+2b}{N-2}$ if $N\geq 3$. Let $u_0 \in H^1_{\rad}(\R^N)$ satisfy \eqref{cond-scat-inte}. Then the corresponding solution to \eqref{INLS} with initial data $\left.u\right|_{t=0} = u_0$ satisfies for any time interval $I \subset [0,\infty)$,
		\begin{align} \label{mora-est-I}
		\int_I \int |x|^b|u(t,x)|^{p+1} dx dt \leq C(u_0,Q) |I|^{\beta}, \quad \beta:= \max\left\{ \frac{1}{3}, \frac{2}{(N-1)\alpha+2}\right\}
		\end{align}
		for some constant $C(u_0,Q)>0$ depending only on $u_0$ and $Q$, where $\alpha=p-1-\frac{2b}{N-1}$.
	\end{lemma}
	
	\begin{proof}
		Let $\varphi_R$ be as in \eqref{varphi-R}. Define $V_{\varphi_R}(t)$ as in \eqref{V-varphi} and set $M_{\varphi_R}(t):=V'_{\varphi_R}(t)$. By the Cauchy-Schwarz inequality, \eqref{est-solu-scat}, and the conservation of mass, we have
		\begin{align} \label{est-M}
		\left|M_{\varphi_R}(t)\right| \lesssim \|\nabla \varphi_R\|_{L^\infty} \|u(t)\|_{L^2} \|\nabla u(t)\|_{L^2} \lesssim R
		\end{align}
		for all $t\in [0,\infty)$. Here the implicit constant depends only on $u_0$ and $Q$. By Lemma \ref{lem-viri-iden} and the fact that $\varphi_R(x)=|x|^2$ for $|x|\leq R$, we have
		\begin{align*}
		M'_{\varphi_R}(t) &= - \int \Delta^2 \varphi_R |u(t)|^2 dx + 4 \rea \sum_{j,k=1}^N \int \partial^2_{jk} \varphi_R \partial_j \overline{u}(t) \partial_k u(t) dx \\
		&\mathrel{\phantom{=}} - \frac{2(p-1)}{p+1} \int \Delta \varphi_R |x|^b|u(t)|^{p+1} dx + \frac{4}{p+1}\int \nabla \varphi_R \cdot \nabla(|x|^b)|u(t)|^{p+1}dx \\
		& = 8 \left( \int_{|x|\leq R} |\nabla u(t)|^2 dx - \frac{N(p-1)-2b}{2(p+1)} \int_{|x|\leq R} |x|^b |u(t)|^{p+1} dx \right) \\
		&\mathrel{\phantom{=}} - \int \Delta^2 \varphi_R |u(t)|^2 dx + 4 \rea \sum_{j,k=1}^N \int_{|x|> R} \partial^2_{jk} \varphi_R\partial_j\overline{u}(t) \partial_k u(t) dx \\
		&\mathrel{\phantom{=}} - \frac{2(p-1)}{p+1} \int_{|x|>R} \Delta \varphi_R |x|^b|u(t)|^{p+1} dx + \frac{4}{p+1}\int_{|x|>R} \nabla \varphi_R \cdot \nabla(|x|^b)|u(t)|^{p+1}dx.
		\end{align*}
		Since $\|\Delta^2 \varphi_R\|_{L^\infty} \lesssim R^{-2}$, the conservation of mass implies
		\[
		\left|\int \Delta^2 \varphi_R |u(t)|^2 dx\right| \lesssim R^{-2}.
		\]
		Since $u$ is radial, we have
		\[
		\rea \sum_{j,k=1}^N \int_{|x|>R} \partial^2_{jk} \varphi_R \partial_j\overline{u}(t) \partial_k u(t) dx = \int_{|x|>R} \varphi''_R |\partial_r u(t)|^2dx \geq 0.
		\]
		Moreover, since $\|\Delta \varphi_R\|_{L^\infty} \lesssim 1$ and $\|x \cdot \nabla \varphi_R\|_{L^\infty} \lesssim |x|^2$, we have from \eqref{est-Strauss}, \eqref{est-solu-scat}, and the conservation of mass that
		\begin{align*}
		\left| \int_{|x|>R} \left(\Delta \varphi_R |x|^b + \nabla \varphi_R \cdot \nabla (|x|^b)\right) |u(t)|^{p+1} dx \right| &\lesssim \int_{|x|>R} |x|^b |u(t)|^{p+1} dx \\
		&\lesssim \|u(t)\|^{\frac{2b}{N-1}}_{H^1} \int_{|x|>R} |u(t)|^{p+1-\frac{2b}{N-1}} dx \\
		&\lesssim \|u(t)\|^{\frac{2b}{N-1}}_{H^1} \left(\sup_{|x|>R} |u(t,x)| \right)^\alpha \|u(t)\|^2_{L^2} \\
		&\lesssim  R^{-\frac{(N-1)\alpha}{2}} \|u(t)\|^{p+1}_{H^1} \lesssim R^{-\frac{(N-1)\alpha}{2}}.
		\end{align*}
		We thus have
		\begin{multline*}
		\frac{d}{dt} M_{\varphi_R}(t) \geq 8 \left( \int_{|x|\leq R} |\nabla u(t)|^2 dx - \frac{N(p-1)-2b}{2(p+1)} \int_{|x|\leq R} |x|^b |u(t)|^{p+1} dx \right)  \\
		+ O\left(R^{-2} + R^{-\frac{(N-1)\alpha}{2}}\right).
		\end{multline*}
		Now let $\chi_R$ be as in Lemma \ref{lem-coer}. We have
		\begin{align*}
		\int |\nabla(\chi_R u(t))|^2 dx &= \int \chi_R^2 |\nabla u(t)|^2 dx - \int \chi_R \Delta(\chi_R) |u(t)|^2 dx \\
		&=\int_{|x|\leq R} |\nabla u(t)|^2 dx - \int_{R/2 \leq |x| \leq R} (1-\chi^2_R) |\nabla u(t)|^2 dx \\
		&\mathrel{\phantom{=\int_{|x|\leq R} |\nabla u(t)|^2 dx}}- \int \chi_R \Delta(\chi_R) |u(t)|^2 dx \\
		\int |x|^b|\chi_R u(t)|^{p+1} dx &= \int_{|x| \leq R} |x|^b|u(t)|^{p+1} dx - \int_{R/2\leq |x| \leq R}  (1-\chi_R^{p+1}) |x|^b |u(t)|^{p+1} dx.
		\end{align*}
		It follows that
		\begin{align*}
		\int_{|x| \leq R} |\nabla u(t)|^2 dx &- \frac{N(p-1)-2b}{2(p+1)} \int_{|x|\leq R} |x|^b|u(t)|^{p+1} dx \\
		&= \int |\nabla(\chi_R u(t))|^2 dx - \frac{N(p-1)-2b}{2(p+1)} \int |x|^b|\chi_R u(t)|^{p+1} dx \\
		&\mathrel{\phantom{=}} + \int_{R/2\leq |x|\leq R} (1-\chi^2_R) |\nabla u(t)|^2 dx + \int \chi_R \Delta(\chi_R) |u(t)|^2 dx \\
		&\mathrel{\phantom{=}} - \frac{N(p-1)-2b}{2(p+1)} \int_{R/2 \leq |x| \leq R} (1-\chi_R^{p+1}) |x|^b|u(t)|^{p+1} dx.
		\end{align*}
		Thanks to the fact that $0\leq \chi_R \leq 1$, $\|\Delta(\chi_R)\|_{L^\infty} \lesssim R^{-2}$ and the radial Sobolev embedding, we get
		\begin{align*}
		\int_{|x| \leq R} |\nabla u(t)|^2 & dx - \frac{N(p-1)-2b}{2(p+1)} \int_{|x|\leq R} |x|^b|u(t)|^{p+1} dx \\
		&\geq \int |\nabla(\chi_R u(t))|^2 dx - \frac{N(p-1)-2b}{2(p+1)} \int |x|^b|\chi_R u(t)|^{p+1} dx + O\left(R^{-2} +R^{-\frac{(N-1)\alpha}{2}}\right).
		\end{align*}
		We thus obtain
		\[
		\frac{d}{dt}M_{\varphi_R}(t) \geq 8 \left(\int |\nabla(\chi_R u(t))|^2 dx - \frac{N(p-1)-2b}{2(p+1)} \int |x|^b|\chi_R u(t)|^{p+1} dx \right) + O\left(R^{-2} +R^{-\frac{(N-1)\alpha}{2}}\right).
		\]
		By Lemma \ref{lem-coer} and \eqref{est-M}, there exist $\delta =\delta(u_0,Q)>0$ and $R_0=R_0(u_0,Q)>0$ such that for all $R\geq R_0$,
		\[
		8\delta \int |x|^b |\chi_R u(t,x)|^{p+1}dx\leq \frac{d}{dt} M_{\varphi_R}(t) + O\left(R^{-2} + R^{-\frac{(N-1)\alpha}{2}}\right)
		\]
		which implies for any time interval $I \subset \R$,
		\[
		8\delta \int_I \int |x|^b |\chi_R u(t,x)|^{p+1}dx dt \leq \sup_{t\in I} |M_{\varphi_R}(t)| + O\left(R^{-2} + R^{-\frac{(N-1)\alpha}{2}}\right) |I|.
		\]
		It follows from the definition of $\chi_R$ and \eqref{est-M} that
		\[
		\int_I \int_{|x|\leq R/2} |x|^b|u(t,x)|^{p+1} dx dt \lesssim R + \left(R^{-2} + R^{-\frac{(N-1)\alpha}{2}} \right)|I|.
		\]
		On the other hand, by radial Sobolev embeddings,
		\begin{align*}
		\int_{|x|\geq R/2} |x|^b |u(t,x)|^{p+1} dx & \leq \|u(t)\|^{\frac{2b}{N-1}}_{H^1}\int_{|x|\geq R/2} |u(t,x)|^{\alpha+2} dx \\
		&\leq \|u(t)\|^{\frac{2b}{N-1}}_{H^1}\left(\sup_{|x|\geq R/2} |u(t,x)|^\alpha \right) \|u(t)\|^2_{L^2} \lesssim R^{-\frac{(N-1)\alpha}{2}}.
		\end{align*}
		We thus get
		\[
		\int_I \int |x|^b |u(t,x)|^{p+1} dx dt \lesssim R + \left(R^{-2} + R^{-\frac{(N-1)\alpha}{2}} \right)|I| \lesssim R+ R^{-\sigma} |I|,
		\]
		where
		\[
		\sigma:= \min \left\{ 2, \frac{(N-1)\alpha}{2}\right\}.
		\]
		Taking $R=|I|^{\frac{1}{1+\sigma}}$, we get for $|I|$ sufficiently large,
		\[
		\int_I \int |x|^b |u(t,x)|^{p+1}dx dt \lesssim |I|^{\frac{1}{1+\sigma}} = |I|^{\beta},
		\]
		where $\beta$ is as in \eqref{mora-est-I}. The proof is complete.
	\end{proof}
	
	\begin{lemma} \label{lem-ener-evac}
		Let $N\geq 2$, $b>0$, $p>\max\left\{\frac{N+4+2b}{N}, 1+\frac{2b}{N-1}\right\}$, and $p<\frac{N+2+2b}{N-2}$ if $N\geq 3$. Let $u_0 \in H^1_{\rad}(\R^N)$ satisfy \eqref{cond-scat-inte}. Then the corresponding solution to \eqref{INLS} with initial data $\left.u\right|_{t=0} = u_0$ satisfies
		\begin{align} \label{ener-evac}
		\liminf_{t\rightarrow \infty} \int |x|^b |u(t,x)|^{p+1}dx =0.
		\end{align}
	\end{lemma}
	
	\begin{proof}
		Assume by contradiction that \eqref{ener-evac} does not hold. Then there exist $t_0>0$ and $\varrho>0$ such that
		\[
		\int |x|^b |u(t,x)|^{p+1}dx \geq \varrho
		\]
		for all $t\geq t_0$. This implies in particular that for every $I \subset [t_0, \infty)$,
		\[
		\int_I \int |x|^b |u(t,x)|^{p+1}dx dt \geq \varrho |I|
		\]
		which contradicts \eqref{mora-est-I} for $|I|$ large as $\beta<1$.
	\end{proof}
	
	\begin{corollary} \label{coro-mora-est-1}
		Let $N\geq 2$, $b>0$, $p>\max\left\{\frac{N+4+2b}{N}, 1+\frac{2b}{N-1}\right\}$, and $p<\frac{N+2+2b}{N-2}$ if $N\geq 3$. Let $u_0 \in H^1_{\rad}(\R^N)$ satisfy \eqref{cond-scat-inte}. Then there exists $t_n \rightarrow \infty$ such that the corresponding solution to \eqref{INLS} with initial data $\left.u\right|_{t=0}=u_0$ satisfies for any $R>0$,
		\begin{align} \label{small-L2}
		\lim_{n\rightarrow \infty} \int_{|x| \leq R} |u(t_n,x)|^2 dx =0.
		\end{align}
	\end{corollary}
	
	\begin{proof}
		By \eqref{ener-evac}, there exists $t_n\rightarrow \infty$ such that
		\[
		\lim_{n\rightarrow \infty} \int |x|^b |u(t_n,x)|^{p+1}dx =0.
		\]
		Let $R>0$. By H\"older's inequality, we see that
		\begin{align*}
		\int_{|x| \leq R} |u(t_n,x)|^2 dx & = \int_{|x|\leq R} |x|^{-\frac{2b}{p+1}} |x|^{\frac{2b}{p+1}} |u(t_n,x)|^2 dx \\
		&\leq \left( \int_{|x| \leq R} |x|^{-\frac{2b}{p-1}} dx \right)^{\frac{p-1}{p+1}} \left( \int_{|x| \leq R} |x|^b|u(t_n,x)|^{p+1} dx \right)^{\frac{2}{p+1}} \\
		&\lesssim R^{\frac{N(p-1)-2b}{p+1}} \left( \int |x|^b|u(t_n,x)|^{p+1} dx  \right)^{\frac{2}{p+1}} \rightarrow 0 \text{ as } n \rightarrow \infty.
		\end{align*}
	\end{proof}
	
	\subsection{Proof of Theorem \ref{theo-scat-inte}}
	In this subsection, we give the proof of Theorem \ref{theo-scat-inte}. Recall that we assume the following conditions:
	\[
	N\geq 2, \quad b>0,\quad p> \frac{N+4}{N} + \frac{2b}{N-1}, \quad p<\frac{N+2+2b}{N-2} \text{ if } N\geq 3.
	\] 
	First we prove the following theorem.
	\begin{theorem} \label{theo-scat-inte-1}
		Let $N\geq 2$, $b>0$, $p>1+\frac{2b}{N-1}+\frac{4}{N}$, and $p<1+\frac{2b}{N-1} + \frac{4}{N-2}$ if $N\geq 3$. Let $u_0 \in H^1_{\rad}(\R^N)$ satisfy \eqref{cond-scat-inte}. Then the corresponding global solution to \eqref{INLS} scatters in $H^1_{\rad}(\R^N)$. In particular, Theorem \ref{theo-scat-inte} holds when $N=2$.
	\end{theorem}
	To prove this result, we denote
	\begin{align}\label{alpha-1}
	\alpha:= p-1-\frac{2b}{N-1}
	\end{align}
	which, under the assumption of Theorem \ref{theo-scat-inte-1}, implies that $\alpha>\frac{4}{N}$ and $\alpha<\frac{4}{N-2}$ if $N\geq 3$.
	
	The following estimates are a direct consequence of H\"older's inequality 
	\begin{lemma} \label{lem-non-est-1}
		Let $\alpha$ be as in \eqref{alpha-1} and denote
		\begin{align} \label{qrkm-1}
		q:=\frac{4(\alpha+2)}{N\alpha}, \quad r:=\alpha+2, \quad k:=\frac{2\alpha(\alpha+2)}{4-(N-2)\alpha}, \quad m:=\frac{2\alpha(\alpha+2)}{N\alpha^2+(N-2)\alpha-4}.
		\end{align}
		We have
		\begin{align} \label{non-est-1}
		\begin{aligned}
		\||u|^\alpha u\|_{L^{m'}(I,L^{r'})} &\lesssim \|u\|^{\alpha+1}_{L^k(I,L^r)}, \\
		\||u|^\alpha u\|_{L^{q'}(I,L^{r'})} &\lesssim \|u\|^\alpha_{L^k(I,L^r)} \|u\|_{L^q(I,L^r)}, \\
		\|\nabla(|u|^\alpha u)\|_{L^{q'}(I,L^{r'})} &\lesssim \|u\|^\alpha_{L^k(I,L^r)} \|\nabla u\|_{L^q(I,L^r)}.
		\end{aligned}
		\end{align}
	\end{lemma}

	\begin{lemma}[Small data scattering] \label{lem-smal-data-1}
		Let $N,b$, and $p$ be as in Theorem \ref{theo-scat-inte-1}. Suppose that $u \in C([0,\infty), H^1_{\rad}(\R^N))$ is a solution to \eqref{INLS} satisfying $\|u\|_{L^\infty([0,\infty), H^1)} <\infty$. Then there exists $\delta>0$ such that if
		\[
		\|e^{i(t-T)\Delta} u(T)\|_{L^k([T,\infty), L^r)} <\delta
		\]
		for some $T>0$, where $k$ and $r$ are as in  \eqref{qrkm-1}, then $u$ scatters in $H^1_{\rad}(\R^N)$.
	\end{lemma}
		
	\begin{proof}
		We proceed in two steps.
		
		{\bf Step 1.} We first show that the solution satisfies
		\begin{align} \label{claim-uT}
		\begin{aligned}
		\|u\|_{L^k([T,\infty), L^r)} &\leq 2 \|e^{i(t-T)\Delta} u(T)\|_{L^k([T,\infty), L^r)}, \\
		\|\scal{\nabla} u\|_{L^q([T,\infty),L^r)} &\leq 2 C\|u(T)\|_{H^1},
		\end{aligned}
		\end{align}
		for some constant $C>0$, where $q, r, k$ are as in \eqref{qrkm-1}. To see this, we consider
		\begin{align*}
		Y:= \Big\{ u \in C(I,H^1_{\rad}(\R^N)) \cap L^k(I,L^r_{\rad}(\R^N)) &\cap L^q(I,W^{1,r}_{\rad}(\R^N)) \\
		&: \|u\|_{L^k(I, L^r)} \leq M, \quad \|\scal{\nabla} u\|_{L^q(I,L^r)} \leq L \Big\}
		\end{align*}
		equipped with the distance
		\[
		d(u,v) := \|u-v\|_{L^k(I,L^r)} + \|u-v\|_{L^q(I,L^r)},
		\]
		where $I=[T,\infty)$ and $M,L>0$ will be chosen later. We will show that the functional
		\[
		\Phi(u)(t):= e^{i(t-T)\Delta} u(T) + i \int_T^t e^{i(t-s)\Delta} |x|^b |u(s)|^{p-1} u(s) ds
		\]
		is a contraction on the complete metric space $(Y,d)$. Since $(q,r)$ is a Schr\"odinger admissible pair, $k,m,q$ satisfy \eqref{cond-km}, and $k>\frac{q}{2}$ due to the fact that $\alpha>\frac{4}{N}$, we see that \eqref{stri-est-non-adm} holds for this choice of exponents. By \eqref{stri-est-non-adm}, \eqref{est-Strauss}, and \eqref{non-est-1}, we have
		\begin{align*}
		\|\Phi(u)\|_{L^k(I,L^r)} &\leq \|e^{i(t-T)\Delta} u(T)\|_{L^k(I,L^r)} + C \||x|^b |u|^{p-1} u\|_{L^{m'}(I,L^{r'})} \\
		&= \|e^{i(t-T)\Delta} u(T)\|_{L^k(I,L^r)} + C \left\|\left(|x|^{\frac{N-1}{2}} |u|\right)^{\frac{2b}{N-1}} |u|^\alpha u\right\|_{L^{m'}(I,L^{r'})} \\
		&\leq \|e^{i(t-T)\Delta} u(T)\|_{L^k(I,L^r)} + C \|u\|^{\frac{2b}{N-1}}_{L^\infty(I,H^1)} \| |u|^\alpha u\|_{L^{m'}(I,L^{r'})} \\
		&\leq \|e^{i(t-T)\Delta} u(T)\|_{L^k(I,L^r)} + C \|u\|^{\frac{2b}{N-1}}_{L^\infty(I,H^1)} \|u\|^{\alpha+1}_{L^k(I,L^r)}.
		\end{align*}
		By Strichartz estimates, \eqref{est-Strauss}, and \eqref{non-est-1}, we have
		\begin{align*}
		\|\Phi(u)\|_{L^q(I,L^r)} &\leq \|e^{i(t-T)\Delta} u(T)\|_{L^q(I,L^r)} + C \||x|^b |u|^{p-1} u\|_{L^{q'}(I,L^{r'})} \\
		&\leq \|u(T)\|_{L^2} + C \left\|\left(|x|^{\frac{N-1}{2}} |u|\right)^{\frac{2b}{N-1}} |u|^\alpha u\right\|_{L^{q'}(I,L^{r'})} \\
		&\leq \|u(T)\|_{L^2} + C \|u\|^{\frac{2b}{N-1}}_{L^\infty(I,H^1)} \| |u|^\alpha u\|_{L^{q'}(I,L^{r'})} \\
		&\leq \|u(T)\|_{L^2} + C \|u\|^{\frac{2b}{N-1}}_{L^\infty(I,H^1)} \|u\|^\alpha_{L^k(I,L^r)} \|u\|_{L^q(I,L^r)}.
		\end{align*}
		We next have
		\[
		\|\nabla \Phi(u)\|_{L^q(I,L^r)} \leq \|\nabla e^{i(t-T)\Delta} u(T)\|_{L^q(I,L^r)} + C \|\nabla(|x|^b |u|^{p-1} u)\|_{L^{q'}(I,L^{r'})}
		\]
		Observe that
		\begin{align*}
		\nabla (|x|^b |u|^{p-1}u) &\sim |x|^b |u|^{p-1} \nabla u + |x|^{b-1} |u|^{p-1} u \\
		&\sim \left(|x|^{\frac{N-1}{2}} |u|\right)^{\frac{2b}{N-1}} |u|^\alpha \nabla u + \left(|x|^{\frac{N-1}{2}} |u|\right)^{\frac{2b}{N-1}} |x|^{-1} |u|^\alpha u.
		\end{align*}
		By \eqref{est-Strauss} and \eqref{non-est-1}, we have
		\begin{align*}
		\||x|^b |u|^{p-1} \nabla u\|_{L^{q'}(I,L^{r'})} &\lesssim \|u\|^{\frac{2b}{N-1}}_{L^\infty(I,H^1)} \|u\|^\alpha_{L^k(I,L^r)} \|\nabla u\|_{L^q(I,L^r)}.
		\end{align*}
		By Hardy's inequality \eqref{Hardy-ineq} with the fact that $1<r'<N$, we see that
		\begin{align*}
		\||x|^{b-1} |u|^{p-1} u\|_{L^{q'}(I,L^{r'})} &\lesssim \|u\|^{\frac{2b}{N-1}}_{L^\infty(I,H^1)} \||x|^{-1} |u|^\alpha u\|_{L^{q'}(I,L^{r'})} \\
		&\lesssim \|u\|^{\frac{2b}{N-1}}_{L^\infty(I,H^1)} \|\nabla(|u|^\alpha u)\|_{L^{q'}(I,L^{r'})} \\
		&\lesssim \|u\|^{\frac{2b}{N-1}}_{L^\infty(I,H^1)} \|u\|^\alpha_{L^k(I,L^r)} \|\nabla u\|_{L^q(I,L^r)}.
		\end{align*}
		Collecting the above estimates, we get
		\[
		\|\nabla \Phi(u)\|_{L^{q'}(I,L^{r'})} \lesssim \|u\|^{\frac{2b}{N-1}}_{L^\infty(I,H^1)} \|u\|^\alpha_{L^k(I,L^r)} \|\nabla u\|_{L^q(I,L^r)}.
		\]
		It follows that
		\[
		\|\scal{\nabla} \Phi(u)\|_{L^{q'}(I,L^{r'})} \lesssim \|u\|^{\frac{2b}{N-1}}_{L^\infty(I,H^1)} \|u\|^\alpha_{L^k(I,L^r)} \|\scal{\nabla} u\|_{L^q(I,L^r)}.
		\]
		We also have
		\begin{align*}
		\|\Phi(u)-\Phi(v)\|_{L^k(I,L^r)} &\lesssim \||x|^b(|u|^{p-1} u-|v|^{p-1} v)\|_{L^{m'}(I,L^{r'})} \\
		&\lesssim \||x|^b (|u|^{p-1} + |v|^{p-1})|u-v|\|_{L^{m'}(I,L^{r'})} \\
		&\lesssim \left(\|u\|^{\frac{2b}{N-1}}_{L^\infty(I,H^1)} \|u\|^\alpha_{L^k(I,L^r)} +\|v\|^{\frac{2b}{N-1}}_{L^\infty(I,H^1)} \|v\|^\alpha_{L^k(I,L^r)}\right) \|u-v\|_{L^k(I,L^r)}
		\end{align*}
		and
		\begin{align*}
		\|\Phi(u)-\Phi(v)\|_{L^q(I,L^r)} &\lesssim \||x|^b(|u|^{p-1} u - |v|^{p-1} v)\|_{L^{q'}(I,L^{r'})}\\
		&\lesssim \||x|^b(|u|^{p-1} + |v|^{p-1})|u-v|\|_{L^{q'}(I,L^{r'})} \\
		&\lesssim \left(\|u\|^{\frac{2b}{N-1}}_{L^\infty(I,H^1)} \|u\|^\alpha_{L^k(I,L^r)} +\|v\|^{\frac{2b}{N-1}}_{L^\infty(I,H^1)} \|v\|^\alpha_{L^k(I,L^r)}\right) \|u-v\|_{L^q(I,L^r)}.
		\end{align*}
		As $H^1$-norm of the solution is bounded uniformly on $[0,\infty)$, there exists $C>0$ such that for all $u,v \in Y$,
		\begin{align*}
		\|\Phi(u)\|_{L^k(I,L^r)} &\leq \|e^{i(t-T)\Delta} u(T)\|_{L^k(I,L^r)} + C M^{\alpha+1}, \\
		\|\scal{\nabla} \Phi(u)\|_{L^q(I,L^r)} &\leq C \|u(T)\|_{H^1} + CM^\alpha L
		\end{align*}
		and
		\[
		d(\Phi(u), \Phi(v)) \leq CM^\alpha d(u,v).
		\]
		By choosing $M=2\|e^{i(t-T)\Delta} u(T)\|_{L^k(I,L^r)}, L=2 C\|u(T)\|_{H^1}$, and taking $M>0$ sufficiently small so that $CM^\alpha \leq \frac{1}{2}$, we obtain that $\Phi$ is a contraction mapping on $(Y,d)$. This shows \eqref{claim-uT}.
		
		{\bf Step 2.} Now let $0<t_1<t_2$. We estimate
		\begin{align*}
		\|e^{-it_2\Delta} u(t_2) - e^{-it_1\Delta} u(t_1)\|_{H^1} &=\left\| i \int_{t_1}^{t_2} e^{-is \Delta} |x|^b |u(s)|^{p-1} u(s) ds \right\|_{H^1} \\
		&= \left\| i \int_{t_1}^{t_2} \scal{\nabla} (|x|^b |u(s)|^{p-1} u(s))\right\|_{L^2} \\
		&\lesssim \|\scal{\nabla}(|x|^b|u|^{p-1} u)\|_{L^{q'}((t_1,t_2),L^{r'})} \\
		&\lesssim \|u\|_{L^\infty((t_1,t_2),H^1)}^{\frac{2b}{N-1}} \|u\|^\alpha_{L^k((t_1,t_2),L^r)} \|\scal{\nabla} u\|_{L^q((t_1,t_2),L^r)}  \rightarrow 0
		\end{align*}
		as $t_1, t_2 \rightarrow \infty$. A similar argument as in the proof of Proposition \ref{prop-lwp-ener} shows the energy scattering. The proof is complete.
	\end{proof}

	\begin{proposition} \label{prop-scat-1}
		Let $N,b$, and $p$ be as in Theorem \ref{theo-scat-inte-1}. Let $u_0 \in H^1_{\rad}(\R^N)$ satisfy \eqref{cond-scat-inte}. Then for $\vareps>0$ sufficiently small, there exists $T=T(\vareps, u_0,Q)$ sufficiently large such that the corresponding solution to \eqref{INLS} with initial data $\left. u\right|_{t=0}= u_0$ satisfies
		\begin{align} \label{small-data}
		\|e^{i(t-T)\Delta} u(T)\|_{L^k([T,\infty),L^r)} \lesssim \vareps^\mu
		\end{align}
		for some $\mu>0$, where $k$ and $r$ are as in \eqref{qrkm-1}.
	\end{proposition}
	
	\begin{proof}
		
		We will consider separately two cases: $N\geq 3$ and $N=2$.
		
		\vspace{3mm}
		
		\noindent {\bf \underline{Case 1.} $N\geq 3$.}
		
		\vspace{3mm}
		
		Let $T>0$ be a large parameter depending on $\vareps, u_0$ and $Q$ to be chosen later. For $T>\vareps^{-\sigma}$ with some $\sigma>0$ to be chosen later, we use the Duhamel formula to write
		\begin{align} \label{duhamel}
		\begin{aligned}
		e^{i(t-T)\Delta} u(T) &= e^{it\Delta} u_0 + i \int_0^T e^{i(t-s)\Delta} |x|^b |u(s)|^{p-1} u(s) ds \\
		&= e^{it\Delta} u_0 + F_1(t) + F_2(t),
		\end{aligned}
		\end{align}
		where
		\[
		F_1(t):= i \int_I e^{i(t-s)\Delta} |x|^b |u(s)|^{p-1} u(s) ds, \quad F_2(t):= i\int_J e^{i(t-s)\Delta} |x|^b |u(s)|^{p-1} u(s) ds
		\]
		with $I:=  [T-\vareps^{-\sigma},T]$ and $J:= [0,T-\vareps^{-\sigma}]$.
		
		\noindent {\bf Step 1. Estimate the linear part.}
		By Strichartz estimates, Sobolev embeddings, we have
		\[
		\|e^{it\Delta} u_0\|_{L^k(\R,L^r)} \lesssim \||\nabla|^{\delta} e^{it\Delta} u_0\|_{L^k(\R,L^l)} \lesssim \|u_0\|_{H^1} <\infty,
		\]
		where
		\begin{align} \label{defi-l}
		l=\frac{2N\alpha(\alpha+2)}{N\alpha^2 +4(N-1)\alpha-8}, \quad \delta = \frac{N\alpha-4}{2\alpha}.
		\end{align}
		Here $(k,l)$ is a Schr\"odinger admissible pair. As $\alpha>\frac{4}{N}$ and $\alpha<\frac{4}{N-2}$ if $N\geq 3$, we readily check that $l \geq r$ and $\delta \in (0,1)$. By the monotone convergence, we may find $T>\vareps^{-\sigma}$ so that
		\begin{align} \label{est-linear}
		\|e^{it\Delta} u_0\|_{L^k([T,\infty),L^r)} \lesssim \vareps.
		\end{align}
		
		\noindent {\bf Step 2. Estimate $F_1$.} By Strichartz estimates \eqref{stri-est-non-adm}, \eqref{non-est-1}, \eqref{est-Strauss}, and Sobolev embedding, we have
		\begin{align*}
		\|F_1\|_{L^k([T,\infty),L^r)} \lesssim \||x|^b|u|^{p-1} u\|_{L^{m'}(I,L^{r'})} &\lesssim \|u\|^{\frac{2b}{N-1}}_{L^\infty(I,H^1)} \||u|^\alpha u\|_{L^{m'}(I,L^{r'})} \\
		&\lesssim   \|u\|^{\alpha+1}_{L^k(I,L^r)} \lesssim |I|^{\frac{\alpha+1}{k}} \|u\|^{\alpha+1}_{L^\infty(I,L^r)},
		\end{align*}
		where we have used the fact that $\|u\|_{L^\infty([0,\infty), H^1)} \leq C(u_0,Q)<\infty$.	We estimate $\|u\|_{L^\infty(I,L^r)}$ as follows. Fix $R = \max \left\{\vareps^{-2-\sigma}, \vareps^{-\frac{4-(N-2)\alpha}{(N-1)\alpha}}\right\}$, we have from \eqref{small-L2} (by enlarging $T$ if necessary) that
		\[
		\int_{|x|\leq R} |u(T,x)|^2 dx \lesssim \vareps^2.
		\]
		By the definition of $\chi_R$,
		\[
		\int \chi_R(x) |u(T,x)|^2 dx \lesssim \vareps^2.
		\]
		Using the fact that
		\begin{align*}
		\left| \frac{d}{dt} \int \chi_R(x) |u(t,x)|^2 dx \right| &= \left| 2 \int \nabla \chi_R(x) \cdot \ima (\overline{u}(t,x) \nabla u(t,x)) dx\right| \\
		&\leq 2 \|\nabla \chi_R\|_{L^\infty} \|u(t)\|_{L^2} \|\nabla u(t)\|_{L^2} \lesssim R^{-1}
		\end{align*}
		for all $t\in \R$, we have for any $t\in I$,
		\begin{align*}
		\int \chi_R(x)|u(t,x)|^2 dx &= \int \chi_R(x) |u(T,x)|^2 dx - \int_t^T \left(\frac{d}{ds} \int \chi_R(x) |u(s,x)|^2 dx \right) ds \\
		&\leq \int \chi_R(x)|u(T,x)|^2 dx +CR^{-1}(T-t) \\
		&\leq  C \vareps^2 + CR^{-1} \vareps^{-\sigma} \leq 2C \vareps^2
		\end{align*}
		for some constant $C=C(u_0,Q)>0$. This shows that
		\[
		\|\chi_R u\|_{L^\infty(I,L^2)} \lesssim \vareps,
		\]
		where we have used the fact $\chi_R^2 \leq \chi_R$ since $0\leq \chi_R \leq 1$.
		By H\"older's inequality, the radial Sobolev embedding \eqref{est-Strauss},
		\begin{align*}
		\|u\|_{L^\infty(I,L^r)} &\leq \|\chi_R u\|_{L^\infty(I,L^r)} + \|(1-\chi_R) u\|_{L^\infty(I,L^r)} \\
		&\leq \|\chi_R u\|^{\frac{4-(N-2)\alpha}{2(\alpha+2)}}_{L^\infty(I,L^2)} \|\chi_R u\|^{\frac{N\alpha}{2(\alpha+2)}}_{L^\infty(I,L^{\frac{2N}{N-2}})} \\
		&\mathrel{\phantom{\leq }}+ \|(1-\chi_R) u\|^{\frac{\alpha}{\alpha+2}}_{L^\infty(I,L^\infty)} \|(1-\chi_R) u\|^{\frac{2}{\alpha+2}}_{L^\infty(I,L^2)} \\
		&\lesssim \vareps^{\frac{4-(N-2)\alpha}{2(\alpha+2)}} + R^{-\frac{(N-1)\alpha}{2(\alpha+2)}} \lesssim \vareps^{\frac{4-(N-2)\alpha}{2(\alpha+2)}}.
		\end{align*}
		It follows that
		\begin{align*}
		\|F_1\|_{L^k([T,\infty),L^r)} \lesssim \vareps^{-\frac{(\alpha+1)\sigma}{k}} \vareps^{\frac{(4-(N-2)\alpha)(\alpha+1)}{2(\alpha+2)}} = \vareps^{(\alpha+1)\left[-\frac{\sigma}{k} + \frac{4-(N-2)\alpha}{2(\alpha+2)} \right]}.
		\end{align*}
		By the definition of $k$, we see that
		\begin{align} \label{est-F1}
		\|F_1\|_{L^k([T,\infty),L^r)} \lesssim \vareps^{\frac{(\alpha+1)(4-(N-2)\alpha)(\alpha-\sigma)}{2\alpha(\alpha+2)}}.
		\end{align}
		
		\noindent {\bf Step 3. Estimate $F_2$.} We estimate
		\[
		\|F_2\|_{L^k([T,\infty),L^r)} \leq \|F_2\|^{\theta}_{L^k([T,\infty),L^l)} \|F_2\|^{1-\theta}_{L^k([T,\infty),L^n)}
		\]
		where $l$ is as in \eqref{defi-l}, $\theta \in (0,1)$ and $n>r$ satisfy
		\[
		\frac{1}{r}=\frac{\theta}{l} +\frac{1-\theta}{n}.
		\]
		Using the fact $(k,l)$ is a Schr\"odinger admissible pair and
		\[
		F_2(t)=e^{i(t-T+\vareps^{-\sigma})\Delta} u(T-\vareps^{-\sigma}) - e^{it\Delta} u_0,
		\]
		Strichartz estimates imply
		\[
		\|F_2\|_{L^k([T,\infty),L^l)} \lesssim 1.
		\]
		On the other hand, by the dispersive estimates and Sobolev embeddings, we have for any $t\geq T$,
		\begin{align*}
		\|F_2(t)\|_{L^n} &\lesssim \int_J (t-s)^{-\frac{N}{2} \left(1-\frac{2}{n}\right)} \||x|^b|u(s)|^{p-1} u(s)\|_{L^{n'}} ds \\
		&\lesssim \int_0^{T-\vareps^{-\sigma}} (t-s)^{-\frac{N}{2}\left(1-\frac{2}{n}\right)} \|u(s)\|^{\frac{2b}{N-1}}_{H^1} \||u(s)|^\alpha u(s) \|_{L^{n'}} ds \\
		&\lesssim \int_0^{T-\vareps^{-\sigma}} (t-s)^{-\frac{N}{2}\left(1-\frac{2}{n}\right)} \|u(s)\|^{\alpha+1}_{L^{n'(\alpha+1)}} ds \\
		&\lesssim (t-T+\vareps^{-\sigma})^{-\frac{N}{2}\left(1-\frac{2}{n}\right) +1}
		\end{align*}
		provided that
		\[
		n'(\alpha+1) \in \left[2,\frac{2N}{N-2}\right], \quad \frac{N}{2}\left(1-\frac{2}{n}\right) -1 >0.
		\]
		This gives
		\[
		\|F_2\|_{L^k([T,\infty),L^n)} \lesssim \left(\int_T^{\infty} (t-T+\vareps^{-\sigma})^{-\left[\frac{N}{2}\left(1-\frac{2}{n}\right) -1\right]k} dt \right)^{\frac{1}{k}} \lesssim \vareps^{\sigma \left[\frac{N}{2}\left(1-\frac{2}{n}\right) - 1 - \frac{1}{k}\right]}
		\]
		provided that
		\[
		\frac{N}{2}\left(1-\frac{2}{n}\right) - 1 - \frac{1}{k}>0.
		\]
		We thus obtain
		\begin{align} \label{est-F2}
		\|F_2\|_{L^k([T,\infty),L^r)} \lesssim \vareps^{\sigma \left[\frac{N}{2}\left(1-\frac{2}{n}\right) - 1 - \frac{1}{k}\right] (1-\theta)}.
		\end{align}
		The above estimate holds true provided
		\[
		n>r, \quad n'(\alpha+1) \in \left[2,\frac{2N}{N-2}\right], \quad \frac{N}{2}\left(1-\frac{2}{n}\right) - 1 - \frac{1}{k}>0.
		\]
		We will choose a suitable $n$ satisfying the above conditions. By the choice of $r$ and $k$, the above conditions become
		\begin{align} \label{condition}
		0\leq \frac{1}{n} <\frac{1}{\alpha+2}, \quad \frac{1}{n} \in \left[ \frac{1-\alpha}{2},\frac{N+2-(N-2)\alpha}{2N}\right], \quad \frac{1}{n}<\frac{(N-2)(\alpha^2+3\alpha)-4}{2N\alpha(\alpha+2)}.
		\end{align}
		
		In the case $\alpha>1$, we take $\frac{1}{n}=0$ or $n=\infty$.
		
		In the case $\alpha\leq 1$, we take $\frac{1}{n}=\frac{1-\alpha}{2}$ or $n=\frac{2}{1-\alpha}$. The first two conditions in \eqref{condition} are satisfied. Let us check the last one which is equivalent to
		\begin{align} \label{alpha}
		N\alpha^3 +2(N-1)\alpha^2 + (N-6)\alpha -4 >0.
		\end{align}
		Since
		\[
		N\alpha^3+2(N-1)\alpha^2 +(N-6)\alpha-4 = (\alpha+1)(N\alpha^2 + (N-2)\alpha -4),
		\]
		it is clear that \eqref{alpha} is fulfilled for $\alpha>\frac{4}{N}$.
		
		\noindent {\bf Step 4. Conclusion.} By \eqref{duhamel}, we get from \eqref{est-linear}, \eqref{est-F1} and \eqref{est-F2} that for $\sigma>0$ sufficiently small, there exists $\mu=\mu(\sigma)>0$ such that
		\[
		\|e^{i(t-T)\Delta} u(T)\|_{L^k([T,\infty),L^r)} \lesssim \vareps^{\mu}.
		\]

		\vspace{3mm}
		
		\noindent \underline{\bf Case 2.} $N = 2$.
		
		\vspace{3mm}
		
		Note that we assume $\alpha> 2$ in this case. Note that the last condition in \eqref{condition} is not applicable for $N=2$. To overcome this difficulty, we use the space time estimate \eqref{mora-est-I} as follows. By the dispersive estimate and H\"older's inequality, we see that for $t\geq T$,
		\begin{align*}
		\|F_2(t)\|_{L^\infty} \lesssim \int_J (t-s)^{-1} \||x|^b|u(s)|^{p-1} u(s)\|_{L^1} ds = \int_J (t-s)^{-1} \||x|^b |u(s)|^{\alpha+1+2b}\|_{L^1}ds.
		\end{align*}
		By H\"older's inequality and the radial Sobolev embedding \eqref{est-Strauss} with $N=2$, we estimate
		\begin{align*}
		\||x|^b |u(s)|^{\alpha+1+2b}\|_{L^1} &\leq \||x|^b |u(s)|^{\frac{(\alpha-1)(\alpha+2)}{\alpha} +2b}\|_{L^{\frac{\alpha}{\alpha-1}}} \||u(s)|^{\frac{2}{\alpha}}\|_{L^{\alpha}} \\
		& = \left(\int |x|^{\frac{\alpha b}{\alpha-1}} |u(s,x)|^{\alpha+2 +\frac{2\alpha b}{\alpha-1}} dx \right)^{\frac{\alpha-1}{\alpha}} \|u(s)\|^{\frac{2}{\alpha}}_{L^2} \\
		& = \left( \int \left(|x|^{\frac{1}{2}} |u(s,x)|\right)^{\frac{2b}{\alpha-1}} |x|^b |u(s,x)|^{p+1} dx\right)^{\frac{\alpha-1}{\alpha}} \|u(s)\|^{\frac{2}{\alpha}}_{L^2} \\
		&\lesssim \|u(s)\|^{\frac{2b}{\alpha}}_{H^1} \left(\int |x|^b |u(s,x)|^{p+1}dx\right)^{\frac{\alpha-1}{\alpha}} \|u(s)\|^2_{L^2}.
		\end{align*}
		Since $\|u\|_{L^\infty([0,\infty),H^1)} \leq C(u_0,Q)<\infty$, we have
		\begin{align*}
		\|F_2(t)\|_{L^\infty} &\lesssim \int_J (t-s)^{-1} \left(\int |x|^b |u(s,x)|^{p+1}dx\right)^{\frac{\alpha-1}{\alpha}} ds \\
		&\lesssim \|(t-s)^{-1}\|_{L^\alpha_s(J)} \left\| \left(\int |x|^b |u(s,x)|^{p+1}dx\right)^{\frac{\alpha-1}{\alpha}} \right\|_{L^{\frac{\alpha}{\alpha-1}}_s(J)} \\
		&\lesssim \|(t-s)^{-1}\|_{L^\alpha_s(J)} \left( \int_J \int |x|^b |u(s,x)|^{p+1}dx ds \right)^{\frac{\alpha-1}{\alpha}}.
		\end{align*}
		We see that for $t\geq T$,
		\begin{align*}
		\|(t-s)^{-1}\|_{L^\alpha_s(J)} &=\left(\int_0^{T-\vareps^{-\sigma}} (t-s)^{-\alpha} ds \right)^{\frac{1}{\alpha}} \\
		&=\left( \left.\frac{(t-s)^{-\alpha +1}}{\alpha -1} \right|_{s=0}^{s=T-\vareps^{-\sigma}} \right)^{\frac{1}{\alpha}} \\
		&\approx \left( (t-T+\vareps^{-\sigma})^{-\alpha+1} - t^{-\alpha+1}\right)^{\frac{1}{\alpha}} \\
		&\lesssim (t-T+\vareps^{-\sigma})^{-\frac{\alpha-1}{\alpha}},
		\end{align*}
		where $t\geq t-T+\vareps^{-\sigma}$ as $T>\vareps^{-\sigma}$. On the other hand, by \eqref{mora-est-I},
		\[
		\int_J \int |x|^b |u(s,x)|^{p+1}dx ds\lesssim |J|^{\beta} \lesssim T^{\beta},
		\]
		where $\beta= \max \left\{\frac{1}{3},\frac{2}{\alpha+2}\right\}$. It yields for $t\geq T$,
		\[
		\|F_2(t)\|_{L^\infty} \lesssim (t-T+\vareps^{-\sigma})^{-\frac{\alpha-1}{\alpha}} T^{\frac{(\alpha-1)\beta}{\alpha}}.
		\]
		It follows that
		\begin{align*}
		\|F_2\|_{L^k([T,\infty),L^\infty)} &\lesssim T^{\frac{(\alpha-1)\beta}{\alpha}} \left(\int_T^{\infty} (t-T+\vareps^{-\sigma})^{-\frac{(\alpha-1)k}{\alpha}} dt \right)^{\frac{1}{k}} \\
		&\lesssim T^{\frac{(\alpha-1)\beta}{\alpha}} \left(\left.(t-T+\vareps^{-\sigma})^{-\frac{(\alpha-1)k}{\alpha}+1}\right|_{t=T}^{t=\infty}\right)^{\frac{1}{k}} \\
		&\lesssim T^{\frac{(\alpha-1)\beta}{\alpha}} \vareps^{\sigma\left(\frac{\alpha-1}{\alpha}-\frac{1}{k}\right)}.
		\end{align*}
		We thus get
		\begin{align} \label{est-F2-2d}
		\|F_2\|_{L^k([T,\infty),L^r)} \lesssim \left[T^{\frac{(\alpha-1)\beta}{\alpha}} \vareps^{\sigma\left(\frac{\alpha-1}{\alpha}-\frac{1}{k}\right)}\right]^{1-\frac{l}{r}}= \left(T^{\frac{(\alpha-1)\beta}{\alpha}} \vareps^{\frac{(\alpha^2+\alpha-4)\sigma}{\alpha(\alpha+2)}}\right)^{\frac{\alpha^2-4}{\alpha^2+2\alpha-4}}.
		\end{align}
		Collecting \eqref{duhamel}, \eqref{est-linear}, \eqref{est-F1}, and \eqref{est-F2-2d}, we have
		\[
		\|e^{i(t-T)\Delta} u(T)\|_{L^k([T,\infty),L^r)} \lesssim \vareps + \vareps^{\frac{2(\alpha+1)(\alpha-\sigma)}{\alpha(\alpha+2)}} + \left(T^{\frac{(\alpha-1)\beta}{\alpha}} \vareps^{\frac{(\alpha^2+\alpha-4)\sigma}{\alpha(\alpha+2)}}\right)^{\frac{\alpha^2-4}{\alpha^2+2\alpha-4}}.
		\]
		By taking $T=\vareps^{-a\sigma}$ with some $a>1$ to be chosen shortly (it ensures $T>\vareps^{-\sigma}$) and choosing $\sigma>0$ small enough, we obtain
		\begin{align} \label{est-eps-mu}
		\|e^{i(t-T)\Delta} u(T)\|_{L^k([T,\infty),L^r)} \lesssim \vareps^\mu
		\end{align}
		for some $\mu>0$. The above estimate requires
		\[
		\frac{\alpha^2+\alpha-4}{\alpha(\alpha+2)} - \frac{a(\alpha-1)\beta}{\alpha}>0 \quad \text{or} \quad a<\frac{\alpha^2+\alpha-4}{\beta(\alpha+2)(\alpha-1)}.
		\]
		It remains to show that
		\begin{align} \label{cond-a}
		\frac{\alpha^2+\alpha-4}{\beta(\alpha+2)(\alpha-1)}>1.
		\end{align}
		In the case $\beta=\frac{1}{3}$ or $\alpha \geq 4$, we see that \eqref{cond-a} is equivalent to
		\[
		\frac{2\alpha^2+2\alpha-10}{(\alpha+2)(\alpha-1)}>0
		\]
		which is satisfied for $\alpha\geq 4$. In the case $\beta=\frac{2}{\alpha+2}$ or $2<\alpha \leq 4$, \eqref{cond-a} is equivalent to
		\[
		\frac{\alpha^2-\alpha-2}{2(\alpha-1)}>0
		\]
		which is also satisfied for $2<\alpha\leq 4$. Therefore, \eqref{cond-a} is satisfied for all $\alpha>2$, and we can choose $a>1$ so that \eqref{est-eps-mu} holds. The proof is complete.
	\end{proof}
	
	\begin{proof}[Proof of Theorem \ref{theo-scat-inte-1}]
		It follows immediately from Lemma \ref{lem-smal-data-1} and Proposition \ref{prop-scat-1}.
	\end{proof}
	
	We are now able to prove Theorem \ref{theo-scat-inte}.
	
	\begin{proof}[Proof of Theorem \ref{theo-scat-inte}]
		We only need to consider the case $N\geq 3$ since the scattering for $N=2$ is included in Theorem \ref{theo-scat-inte-1}. We observe that the main elements needed for the proof of Theorem \ref{theo-scat-inte-1} are: the radial Sobolev embedding \eqref{est-Strauss} and $\frac{4}{N}<\alpha<\frac{4}{N-2}$. Now using \eqref{est-BL} instead of \eqref{est-Strauss}, the same argument as in the proof of Theorem \ref{theo-scat-inte} with $\alpha=p-1-\frac{2b}{N-2}$ gives the energy scattering for
		\[
		N\geq 3, \quad b>0, \quad 1+\frac{2b}{N-2} + \frac{4}{N} < p <1+\frac{2b}{N-2} +\frac{4}{N-2}.
		\]
		
		If $1+\frac{2b}{N-1} + \frac{4}{N-2} > 1+\frac{2b}{N-2} +\frac{4}{N}$ or $0<b<\frac{4(N-1)}{N}$, then we have the energy scattering for the full range $\frac{N+4}{N}+\frac{2b}{N-1}< p<\frac{N+2+2b}{N-2}$, and we are done.
		
		Now if $b\geq \frac{4(N-1)}{N}$, then there is a gap $1+\frac{2b}{N-1}+\frac{4}{N-2} \leq p \leq 1+\frac{2b}{N-2} +\frac{4}{N}$. We fill this gap as follows. As 
		\[
		1+\frac{2b}{N-1} +\frac{4}{N} < 1+\frac{2b}{N-1} +\frac{4}{N-2} \leq 1+\frac{2b}{N-2}+\frac{4}{N},
		\]
		there exists $s_1 \in \left(\frac{1}{2}, 1\right]$ such that
		\[
		1+\frac{2b}{N-1} +\frac{4}{N-2} = 1+\frac{2b}{N-2s_1} +\frac{4}{N}.
		\]
		Using \eqref{est-CO} with $s=s_1$ and repeating the same argument as in the proof of Theorem \ref{theo-scat-inte-1} with $\alpha=p-1-\frac{2b}{N-2s_1}$, we can prove the energy scattering for
		\[
		1+\frac{2b}{N-2s_1} + \frac{4}{N} \leq p <1+\frac{2b}{N-2s_1} + \frac{4}{N-2}.
		\]
		Note that the lower bound $p=1+\frac{2b}{N-2s_1} +\frac{4}{N}$ is admissible since we can write it as
		\[
		1+\frac{2b}{N-2s_1} +\frac{4}{N} = 1+\frac{2b}{N-1} +\frac{4}{N-2} = 1+ \frac{2b}{N-2s_\vareps} +\frac{4}{N-2+\vareps}
		\]
		with 
		\[
		s_\vareps = \frac{2N(N-1)\vareps + b(N-2)(N-2+\vareps)}{2b(N-2)(N-2+\vareps) + 4\vareps(N-1)}.
		\]
		Note that $s_\vareps \rightarrow \frac{1}{2}^+$ as $\vareps \rightarrow 0$. So by choosing $\vareps>0$ small enough, $s_\vareps$ is close to $\frac{1}{2}^+$. Using again \eqref{est-CO} with $s=s_\vareps$, the same argument as in the proof of Theorem \ref{theo-scat-inte-1} with $\alpha=p-1-\frac{2b}{N-2s_\vareps}$ gives the energy scattering for $p= 1+ \frac{2b}{N-2s_\vareps} +\frac{4}{N-2+\vareps}$, hence for $p=1+\frac{2b}{N-2s_1} +\frac{4}{N}$.
		
		If $1+\frac{2b}{N-2s_1} +\frac{4}{N-2} >1+\frac{2b}{N-2} +\frac{4}{N}$ or $\frac{4(N-1)}{N} \leq b<\frac{2(N-2s_1)}{N(1-s_1)}$, then the energy scattering holds for the whole range $\frac{N+4}{N}+\frac{2b}{N-1}<p<\frac{N+2+2b}{N-2}$.
		
		Otherwise, if $b\geq \frac{2(N-2s_1)}{N(1-s_1)}$, then we choose $s_2 \in (s_1, 1]$ such that
		\[
		1+\frac{2b}{N-2s_1} +\frac{4}{N-2} = 1+\frac{2b}{N-2s_2} +\frac{4}{N}.
		\]
		Note that it is possible due to
		\[
		1+\frac{2b}{N-2s_1}+\frac{4}{N}<1+\frac{2b}{N-2s_1} +\frac{4}{N-2} \leq 1+\frac{2b}{N-2}+\frac{4}{N}.
		\]
		Repeating the same reasoning, we get the energy scattering for 
		\[
		1+\frac{2b}{N-2s_2}+\frac{4}{N} \leq p <1+\frac{2b}{N-2s_2}+\frac{4}{N-2}.
		\]
		If $1+\frac{2b}{N-2s_2}+\frac{4}{N-2}>1+\frac{2b}{N-2}+\frac{4}{N}$ or $\frac{2(N-2s_1)}{N(1-s_1)} \leq b<\frac{2(N-2s_2)}{N(1-s_2)}$, we are done. Otherwise, we repeat the above argument until $\frac{2(N-2s_{k-1})}{N(1-s_{k-1})} \leq b <\frac{2(N-2s_k)}{N(1-s_k)}$. Note that $\frac{2(N-2s_k)}{N(1-s_k)} \rightarrow \infty$ as $s_k \rightarrow 1^-$ and $b>0$ is given, the above process will be stopped in finite steps. The proof is complete.
	\end{proof}

	\section{Blow-up solutions}
	\label{S5}
	\setcounter{equation}{0}
	In this section, we give the proofs of finite time blow-up given in Theorems \ref{theo-blow-mass-appl}, \ref{theo-blow-inter}, and \ref{theo-blow-ener}. The proof is based on an idea of Ogawa and Tsutsumi \cite{OT-JDE} using localized virial estimates and radial Sobolev embeddings.
	
	\subsection{Mass-critical nonlinearity}
	We shall show the existence of finite time blow-up solutions to \eqref{INLS} in the mass-critical regime. To this end, we introduce the real function $\vartheta: [0,\infty) \rightarrow [0,\infty)$ satisfying
	\[
	\vartheta(r)=
	\left\{
	\renewcommand*{\arraystretch}{1.2}
	\begin{array}{c c c}
	2r &\text{if} & 0 \leq r\leq 1,\\
	2[r-(r-1)^5] &\text{if} & 1<r\leq 1+\frac{1}{\sqrt[4]{5}},\\
	\vartheta'<0 &\text{if} & 1+\frac{1}{\sqrt[4]{5}}<r<2,\\
	0 &\text{if}  & r\geq2.
	\end{array}
	\right.
	\]
	We set $\theta(r):= \mathlarger{\int}_0^r \vartheta(\tau)d\tau$ and define the radial function
	\begin{align} \label{psi-R}
	\psi_R(x)= \psi_R(r) =R^2 \theta(r/R), \quad r =|x|.
	\end{align}
	It is straightforward to check that
	\[
	\psi''_R(r) \leq 2, \quad \frac{\psi'_R(r)}{r} \leq 2, \quad \Delta \psi_R(x) \leq 2N, \quad \forall r\geq 0, \quad \forall x \in \R^N.
	\]
	We need the following localized virial estimate in the mass-critical case.

	\begin{lemma}\label{lem-viri-est-mass}
		Let $N\geq 2$, $0<b < 2(N-1)$, and $p=\frac{N+4+2b}{N}$. Let $u \in C([0,T^*), H^1_{\rad}(\R^N))$ be a solution to \eqref{INLS}. Let $\psi_R$ be as in \eqref{psi-R} and define $V_{\psi_R}(t)$ as in \eqref{V-varphi}. Then for any $R, \vareps>0$, we have for all $t\in [0,T^*)$,
		\begin{align} \label{viri-est-mass}
		\begin{aligned}
		V_{\psi_R}''(t) \leq 16E(u_0)&+CR^{-2}+C \vareps R^{-2} + C\vareps^{-\frac{2+b}{2N-2-b}}R^{-2} \\
		&-2\int_{|x|>R}\Big(2\psi_{1,R}(r)-\frac{N \vareps}{2N+4+2b}(\psi_{2,R}(x))^{\frac{N}{2+b}}\Big)|\nabla u(t,x)|^2dx
		\end{aligned}
		\end{align}
		for some constant $C>0$, where
		\begin{align} \label{psi-12-R}
		\psi_{1,R}(r):=2-\psi_R''(r),\quad \psi_{2,R}(x):=\frac{4+2b}{N}(2N-\Delta\psi_R(x))-2b \left(2-\frac{\psi_R'(r)}{r}\right), \quad r=|x|.
		\end{align}
	\end{lemma}

	\begin{proof}
		Denote $p_*:=\frac{N+4+2b}{N}$. Thanks to Lemma \ref{lem-viri-iden} and $\psi_R(x)=|x|^2$ for $|x|\leq R$, we have for all $t\in [0,T^*)$,
		\begin{align*}
			V_{\psi_R}''(t)
			&=-\int \Delta^2\psi_R|u(t)|^2dx + 4\rea\sum_{j,k=1}^N\int\partial_{jk}^2\psi_R\partial_ku(t)\partial_j\bar u(t)dx\\
			&\mathrel{\phantom{=}}-\frac{2(p_*-1)}{p_*+1}\int\Delta\psi_R|x|^b|u(t)|^{p_*+1}dx+\frac4{p_*+1}\int\nabla\psi_R \cdot \nabla(|x|^b)|u(t)|^{p_*+1}dx\\
			&= 16 E(u(t)) -8\int_{|x|>R}|\nabla u(t)|^2dx + \frac{4N(p_*-1) -8b}{p_*+1} \int_{|x|>R} |x|^b |u(t)|^{p_*+1} dx \\
			&\mathrel{\phantom{=}} -\int \Delta^2\psi_R|u(t)|^2dx +4\rea\int_{|x|>R}\partial_{jk}^2\psi_R\partial_ku(t)\partial_j\bar u(t)dx\\
			&\mathrel{\phantom{=}}-\frac{2(p_*-1)}{p_*+1}\int_{|x|>R}\Delta\psi_R|x|^b|u(t)|^{p_*+1}dx +\frac4{p_*+1}\int_{|x|>R}\nabla\psi_R \cdot \nabla(|x|^b)|u|^{p_*+1} dx.
		\end{align*}
		Since $u(t)$ is radial, we have from the conservation of energy and the following identities
		\[
		\rea \sum_{j,k=1}^N\partial_{jk}^2\psi_R\partial_ku(t)\partial_j\bar u(t) =\psi_R''|\partial_ru(t)|^2=\psi''_R |\nabla u(t)|^2, \quad \nabla\psi_R \cdot \nabla(|x|^b)=b|x|^b\frac{\psi_R'}{r}
		\]
		that
		\begin{align*}
			V_{\psi_R}''(t)
			&=16E(u_0)-\int\Delta^2\psi_R|u(t)|^2dx - 4\int_{|x|>R}(2-\psi_R'')|\nabla u(t)|^2dx\\
			&\mathrel{\phantom{\leq 16E(u_0)}}+\int_{|x|>R} \left[\frac{2(p_*-1)}{p_*+1} (2N-\Delta\psi_R)-\frac{4b}{p_*+1}\left(2-\frac{\psi_R'}{r}\right)\right]|x|^b|u(t)|^{p_*+1}dx\\
			&\leq 16E(u_0)-\int\Delta^2\psi_R|u(t)|^2dx \\
			&\mathrel{\phantom{\leq 16E(u_0)}}-4\int_{|x|>R}\psi_{1,R}|\nabla u(t)|^2dx+\frac{2}{p_*+1}\int_{|x|>R}\psi_{2,R}|x|^b|u(t)|^{p_*+1}dx,
		\end{align*}
		where $\psi_{1,R}$ and $\psi_{2,R}$ are as in \eqref{psi-12-R}.
		
		As $\|\Delta^2\psi_R\|_{L^\infty}\lesssim R^{-2}$, the conservation of mass implies
		\[
		\left|\int\Delta^2\psi_R|u(t)|^2dx\right| \lesssim R^{-2}.
		\]
		Next, using \eqref{est-Strauss}, $\|\psi_R\|_{L^\infty}\lesssim 1$, and the conservation of mass, we estimate
		\begin{align*}
			\int_{|x|>R}\psi_{2,R}|x|^b|u(t)|^{p_*+1}dx
			&\leq \sup_{|x|>R} \left(|x|^{\frac{b}{p_*-1}} \psi_{2,R}^{\frac{1}{p_*-1}} |u(t,x)|\right)^{p_*-1}\|u(t)\|^2_{L^2}\\
			&\leq CR^{-\left(\frac{(N-1)(p_*-1)}{2}-b\right)} \sup_{|x|>R} \left(|x|^{\frac{N-1}{2}} \psi_{2,R}^{\frac{1}{p_*-1}} |u(t,x)|\right)^{p_*-1} \|u(t)\|^2_{L^2}\\
			&\leq CR^{-\left(\frac{(N-1)(p_*-1)}{2}-b\right)}\Big\|\nabla \Big(\psi_{2,R}^{\frac{1}{p_*-1}}u(t) \Big)\Big\|_{L^2}^{\frac{p_*-1}{2}}.
		\end{align*}
		Note that
		\[
		\frac{(N-1)(p_*-1)}{2}-b>0, \quad \frac{p_*-1}{2}<2
		\]
		due to $N\geq 2$ and $0<b<2(N-1)$. Now, with Young's inequality, we have for any $\vareps>0$,
		\begin{align*}
			\int_{|x|>R}\psi_{2,R}|x|^b|u(t)|^{p_*+1}dx
			\leq \vareps \Big\|\nabla \Big(\psi_{2,R}^{\frac{1}{p_*-1}}u(t)\Big)\Big\|^2_{L^2} + C \vareps^{-\frac{p_*-1}{5-p_*}} R^{-\frac{4}{5-p_*}\left(\frac{(N-1)(p_*-1)}{2}-b\right)}.
		\end{align*}
		From Lemma \ref{lem-cutoff}, we have $\Big\|\nabla\Big(\psi_{2,R}^{\frac{1}{p_*-1}}\Big) \Big\|_{L^\infty} \lesssim R^{-1}$. It follows that		
		\begin{align*}
			\int_{|x|>R}\psi_{2,R}|x|^b|u(t)|^{p_*+1}dx
			&\leq \vareps \Big\|\psi_{2,R}^{\frac{1}{p_*-1}} \nabla u(t)\Big\|^2_{L^2} + C\vareps R^{-2}+ C \vareps^{-\frac{p_*-1}{5-p_*}} R^{-\frac{4}{5-p_*}\left(\frac{(N-1)(p_*-1)}{2}-b\right)}.
		\end{align*}
		Collecting the above estimates, we obtain
		\begin{align*}
			V_{\psi_R}''(t)
			&\leq 16E(u_0)+CR^{-2} + C\vareps R^{-2} + C \vareps^{-\frac{p_*-1}{5-p_*}} R^{-\frac{4}{5-p_*}\left(\frac{(N-1)(p_*-1)}{2}-b\right)} \\
			&\mathrel{\phantom{\leq 16E(u_0)}}-2\int_{|x|>R}\Big(2\psi_{1,R}-\frac\varepsilon{p_*+1}\psi_{2,R}^{\frac{2}{p_*-1}}\Big)|\nabla u(t)|^2dx
		\end{align*}
		which completes the proof. Note that $\frac{4}{4-p_*}\left(\frac{(N-1)(p_*-1)}{2}-b\right)=2$.
	\end{proof}	

	\begin{lemma} \label{lem-cutoff}
		Let $N\geq 2$, $0<b<2(N-1)$, and $p_*=\frac{N+4+2b}{N}$. Let $\psi_{2,R}$ be as in \eqref{psi-12-R}. Then we have
		\begin{align} \label{est-deri}
		\Big\| \nabla \Big(\psi_{2,R}^{\frac{1}{p_*-1}}\Big) \Big\|_{L^\infty} \lesssim R^{-1}.
		\end{align}
	\end{lemma}
	
	\begin{proof}
		Using the fact that $\Delta \psi_R(x) = \psi''_R(r) +\frac{N-1}{r} \psi'_R(r)$, we see that
		\begin{align*}
		\psi_{2,R}(x)
		&=(p_*-1) (2N-\Delta \psi_R(x)) -2b\left(2-\frac{\psi'_R(r)}{r}\right)\\
		&=(p_*-1)(2-\psi_R''(r))+((N-1)(p_*-1)-2b)\left(2-\frac{\psi_R'(r)}{r}\right).
		\end{align*}
		
		$\bullet$ When $R<r\leq \left(1+\frac1{\sqrt[4]{5}}\right)R$, we have
		\[
		2-\frac{\psi'_R(r)}{r} = \frac{2(r/R-1)^5}{r/R}, \quad 2-\psi''_R(r) = 10(r/R-1)^4,
		\]
		hence
		\begin{align} \label{psi-2-R}
		\psi_{2,R}(x)=(r/R-1)^4 \left( 10(p_*-1) + 2((N-1)(p_*-1)-2b)\left(1-\frac{1}{r/R} \right) \right).
		\end{align}
		It follows that
		\[
		(\psi_{2,R}(x))^{\frac{1}{p_*-1}} = g(r/R),
		\]
		where
		\[
		g(\lambda) := (\lambda-1)^{\frac{4}{p_*-1}}\left(C+D - \frac{D}{\lambda}\right)^{\frac{1}{p_*-1}}
		\]
		with
		\[
		C=10(p_*-1), \quad D= 2((N-1)(p_*-1)-2b).
		\]
		Note that $C,D>0$ as $0<b<2(N-1)$. We have
		\[
		\partial_r \Big( \psi_{2,R}^{\frac{1}{p_*-1}}\Big) = \frac{1}{R} g'(r/R).
		\]
		We compute for $1<\lambda \leq 1+\frac{1}{\sqrt[4]{5}}$,
		\begin{align*}
		g'(\lambda) = \frac{4}{p_*-1} (\lambda-1)^{\frac{5-p_*}{p_*-1}} \left(C+D-\frac{D}{\lambda}\right)^{\frac{1}{p_*-1}} + \frac{D}{p_*-1} (\lambda-1)^{\frac{1}{p_*-1}} \lambda^{-2} \left(C+D-\frac{D}{\lambda}\right)^{\frac{2-p_*}{p_*-1}}.
		\end{align*}
		The function $h:\lambda \mapsto C+D-\frac{D}{\lambda}$ is strictly increasing on $\left(1, 1+\frac{1}{\sqrt[4]{5}}\right)$. Thus
		\[
		C=h(1)<h(\lambda) \leq h\left(1+\frac{1}{\sqrt[4]{5}}\right) = C+\frac{D}{1+\sqrt[4]{5}}.
		\]
		This shows that $h(\lambda)$ is bounded both from above and from below away from zero. As $1<p_*<5$ due to $0<b<2(N-1)$, we see that $g'(\lambda)$ is bounded from above by some positive constant. Thus we get \eqref{est-deri} for $R <r \leq \left(1+\frac{1}{\sqrt[4]{5}}\right) R$. 
		
		$\bullet$ When $\left(1+\frac{1}{\sqrt[4]{5}}\right) R< r <2R$, we have $\psi''_R(r) = \vartheta'(r/R) <0$. As $\vartheta(r/R)\geq 0$, we also have
		\[
		\frac{\psi'_R(r)}{r} = \frac{\vartheta(r/R)}{r/R} \in \left(\frac{\vartheta(2)}{2}, \frac{\vartheta(1+1/\sqrt[4]{5})}{1+1/\sqrt[4]{5}} \right).
		\]
		Thus we get
		\begin{align*}
		\psi_{2,R}(x)
			&=(p_*-1)(2-\psi_R''(r))+((N-1)(p_*-1)-2b)\left(2-\frac{\psi_R'(r)}{r}\right)\\
			&>2(p_*-1)+((N-1)(p_*-1)-2b)\left(2-\frac{\sqrt[4]{5}}{1+\sqrt[4]{5}}\vartheta(1+1/\sqrt[4]{5})\right)\\
			&>2(p_*-1)+\frac{2((N-1)(p_*-1)-2b)}{5(1+\sqrt[4]{5})},
		\end{align*}
		where we have used the fact that $\vartheta(1+1/\sqrt[4]{5}) = \frac{2(1+\sqrt[4]{5})}{\sqrt[4]{5}} - \frac{2}{5\sqrt[4]{5}}$. Now, because of $|\nabla \psi_{2,R}|\lesssim R^{-1}$ and $|\psi_{2,R}|\gtrsim 1$, we obtain \eqref{est-deri} for $\left(1+\frac{1}{\sqrt{3}}\right) R<r<2R$.
		
		$\bullet$ When $r\geq 2R$, we have $\psi'_R(r)=\psi''_R(r)=0$. It is straightforward to see that \eqref{est-deri} holds in this case. The proof is complete.
	\end{proof}
	
	\begin{lemma} \label{lem-vareps}
		Let $N\geq 3$ and $0<b\leq N-2$. Let $\psi_{1,R}$ and $\psi_{2,R}$ be as in \eqref{psi-12-R}. Then we have for $\vareps>0$ sufficiently small,
		\begin{align} \label{choi-vareps}
		2\psi_{1,R}(r) - \frac{N\vareps}{2N+4+2b} (\psi_{2,R}(x))^{\frac{N}{2+b}} \geq 0, \quad \forall r =|x|>R.
		\end{align}
	\end{lemma}
	
	\begin{proof}
		When $R<r \leq \left(1+\frac{1}{\sqrt[4]{5}}\right)R$, we have $\psi_{1,R}(r)=20(r/R-1)^4$ and from \eqref{psi-2-R},
		\begin{align*}
		\psi_{2,R}(x) &= 4(r/R-1)^4 \left( \frac{5(2+b)}{N} + \frac{2(N-1)-b}{N} \left(1-\frac{1}{r/R}\right)\right) \\
		&< 4(r/R-1)^4 \left( \frac{5(2+b)}{N} + \frac{2(N-1)-b}{(1+\sqrt[4]{5})N}\right).
		\end{align*}
		It follows that
		\[
		2\psi_{1,R}(r) - \frac{N\vareps}{2N+4+2b} (\psi_{2,R}(x))^{\frac{N}{2+b}} > 20(r/R-1)^4 - C\vareps (r/R-1)^{\frac{4N}{2+b}}.
		\]
		Since $0<r/R-1<1/\sqrt[4]{5}$ and $\frac{4N}{2+b} \geq 4$ as $0<b\leq N-2$, by taking $\vareps>0$ sufficiently small, we have \eqref{choi-vareps}. When $r>\left(1+\frac{1}{\sqrt[4]{5}}\right) R$, we see that $\vartheta'(r/R) \leq 0$, so $\psi_{1,R}(r) =2-\psi''_R(r)\geq 2$. We also have $\psi_{2,R}(x) \lesssim 1$. By choosing $\vareps>0$ small enough, we get \eqref{choi-vareps}.
	\end{proof}

	\begin{proof}[Proof of Theorem \ref{theo-blow-mass-appl}]
		By Lemma \ref{lem-viri-est-mass}, we have for all $t\in [0,T^*)$ and all $R,\vareps>0$,
		\begin{align*}
		V_{\psi_R}''(t) \leq 16E(u_0)&+CR^{-2}+C \vareps R^{-2} + C\vareps^{-\frac{2+b}{2N-2-b}}R^{-2} \\
		&-2\int_{|x|>R}\Big(2\psi_{1,R}(r)-\frac{N \vareps}{2N+4+2b}(\psi_{2,R}(x))^{\frac{N}{2+b}}\Big)|\nabla u(t,x)|^2dx,
		\end{align*}
		where $\psi_{1,R}$ and $\psi_{2,R}$ are as in \eqref{psi-12-R}. By Lemma \ref{lem-vareps}, there exists $\vareps>0$ sufficiently small so that
		\[
		2\psi_{1,R}(r)-\frac{N \vareps}{2N+4+2b}(\psi_{2,R}(x))^{\frac{N}{2+b}} \geq 0, \quad \forall r=|x|>R.
		\]
		By taking $R>0$ sufficiently large depending on $\vareps$, we get
		\[
		V_{\psi_R}''(t) \leq 8E(u_0)<0, \quad \forall t\in [0,T^*).
		\]
		The standard convexity argument (see e.g., \cite{Glassey}) shows $T^*<\infty$.
	\end{proof}
	
	\subsection{Inter-critical nonlinearity}
	We next establish the existence of blow-up solutions to \eqref{INLS} in the mass-supercritical and energy-sub-critical case.
	
	\begin{lemma} \label{lem-viri-est-inte}
		Let $N\geq 2$, $b>0$, $p> \frac{N+4+2b}{N}$, $p<\frac{N+2+2b}{N-2}$ if $N\geq 3$, and $p\leq 5$. Let $u \in C([0,T^*), H^1_{\rad}(\R^N))$ be a solution to \eqref{INLS}. Let $\psi_R$ be as in \eqref{psi-R} and define $V_{\psi_R}(t)$ as in \eqref{V-varphi}. Then for any $R>0$, we have for all $t\in [0,T^*)$,
		\begin{align} \label{viri-est-inte}
		\begin{aligned}
		V''_{\psi_R}(t) \leq 8 \|\nabla u(t)\|^2_{L^2} &- \frac{4N(p-1)-8b}{p+1} \int |x|^b |u(t,x)|^{p+1} dx \\
		&+ CR^{-2} + \left\{
		\begin{array}{ccc}
		C R^{-(2(N-1)-b)}\|\nabla u(t)\|^2_{L^2} &\text{if} & p=5, \\
		C R^{-\left(\frac{(N-1)(p-1)}{2}-b\right)} \left(\|\nabla u(t)\|^2_{L^2} + 1\right) &\text{if}& p<5,
		\end{array}
		\right.
		\end{aligned}
		\end{align}
		for some constant $C>0$.
	\end{lemma}
	
	\begin{proof}
		Thanks to Lemma \ref{lem-viri-iden} and $\psi_R(x)=|x|^2$ for $|x|\leq R$, we have for all $t\in [0,T^*)$,
		\begin{align*}
		V_{\psi_R}''(t)
		&=-\int \Delta^2\psi_R|u(t)|^2dx + 4\rea\sum_{j,k=1}^N\int\partial_{jk}^2\psi_R\partial_ku(t)\partial_j\bar u(t)dx\\
		&\mathrel{\phantom{=}}-\frac{2(p-1)}{p+1}\int\Delta\psi_R|x|^b|u(t)|^{p+1}dx+\frac{4}{p+1}\int\nabla\psi_R \cdot \nabla(|x|^b)|u(t)|^{p+1}dx\\
		&= 8 \|\nabla u(t)\|^2_{L^2} -\frac{4N(p-1)-8b}{p+1} \int |x|^b |u(t)|^{p+1} dx \\
		&\mathrel{\phantom{=}}-8\int_{|x|>R}|\nabla u(t)|^2dx + \frac{4N(p-1) -8b}{p+1} \int_{|x|>R} |x|^b |u(t)|^{p+1} dx \\
		&\mathrel{\phantom{=}} -\int \Delta^2\psi_R|u(t)|^2dx +4\rea\int_{|x|>R}\partial_{jk}^2\psi_R\partial_ku(t)\partial_j\bar u(t)dx\\
		&\mathrel{\phantom{=}}-\frac{2(p-1)}{p+1}\int_{|x|>R}\Delta\psi_R|x|^b|u(t)|^{p+1}dx +\frac4{p+1}\int_{|x|>R}\nabla\psi_R \cdot \nabla(|x|^b)|u|^{p+1} dx.
		\end{align*}
		Since $u(t)$ is radial, we have from the conservation of energy and the following identities
		\[
		\rea \sum_{j,k=1}^N\partial_{jk}^2\psi_R\partial_ku(t)\partial_j\bar u(t) =\psi_R''|\partial_ru(t)|^2=\psi''_R |\nabla u(t)|^2, \quad \nabla\psi_R \cdot \nabla(|x|^b)=b|x|^b\frac{\psi_R'}{r}
		\]
		that
		\begin{align*}
		V_{\psi_R}''(t)
		&= 8 \|\nabla u(t)\|^2_{L^2} -\frac{4N(p-1)-8b}{p+1} \int |x|^b |u(t)|^{p+1} dx \\
		&\mathrel{\phantom{=}}-\int\Delta^2\psi_R|u(t)|^2dx - 4\int_{|x|>R}(2-\psi_R'')|\nabla u(t)|^2dx\\
		&\mathrel{\phantom{=}}+\int_{|x|>R} \left[\frac{2(p-1)}{p+1} (2N-\Delta\psi_R)-\frac{4b}{p+1}\left(2-\frac{\psi_R'}{r}\right)\right]|x|^b|u(t)|^{p+1}dx.
		\end{align*}
		
		As $\|\Delta^2\psi_R\|_{L^\infty}\lesssim R^{-2}$, the conservation of mass implies
		\[
		\left|\int\Delta^2\psi_R|u(t)|^2dx\right| \lesssim R^{-2}.
		\]
		As $\psi''_R \leq 2$ and $\|2N-\Delta \psi_R\|_{L^\infty}, \left\|2-\frac{\psi'_R}{r}\right\|_{L^\infty} \lesssim 1$, we have
		\begin{align*}
		V''_{\psi_R}(t) \leq 8 \|\nabla u(t)\|^2_{L^2} &- \frac{4N(p-1)-8b}{p+1} \int |x|^b |u(t)|^{p+1} dx \\
		& + CR^{-2} + C \int_{|x|>R} |x|^b|u(t)|^{p+1} dx.
		\end{align*}
		By \eqref{est-Strauss} and the conservation of mass, we estimate
		\begin{align*}
		\int_{|x|>R} |x|^b |u(t)|^{p+1} dx &\leq \sup_{|x|>R} \left( |x|^{\frac{b}{p-1}} |u(t,x)|\right)^{p-1} \|u(t)\|^2_{L^2} \\
		&\leq C R^{-\left(\frac{(N-1)(p-1)}{2}-b \right)} \sup_{|x|>R} \left(|x|^{\frac{N-1}{2}} |u(t,x)|\right)^{p-1} \|u(t)\|^2_{L^2} \\
		&\leq C R^{-\left(\frac{(N-1)(p-1)}{2}-b \right)} \|\nabla u(t)\|^{\frac{p-1}{2}}_{L^2}.
		\end{align*}
		Note that
		\[
		\frac{(N-1)(p-1)}{2}-b>0, \quad \frac{p-1}{2} \leq 2.
		\]
		When $p=5$, we are done. When $p<5$, we have from Young's inequality that
		\begin{align*}
		\int_{|x|>R}\psi_{2,R}|x|^b|u(t)|^{p+1}dx
		\leq CR^{-\left(\frac{(N-1)(p-1)}{2}-b \right)} \left(\|\nabla u(t)\|^2_{L^2} + 1\right).
		\end{align*}
		This completes the proof.
	\end{proof}
	
	\begin{lemma} \label{lem-Q-blow}
		Let $N\geq 2$, $b>0$, $p>\max\left\{\frac{N+4+2b}{N}, 1+\frac{2b}{N-1}\right\}$, and $p<\frac{N+2+2b}{N-2}$ if $N\geq 3$. Let $u_0 \in H^1_{\rad}(\R^N)$ satisfy either $E(u_0)<0$ or if $E(u_0) \geq0$, we assume that \eqref{cond-blow-inte}. 
		Then there exist $\vareps, \nu>0$ such that for all $t\in [0,T^*)$,
		\begin{align} \label{est-blow}
		\|\nabla u(t)\|^2_{L^2} - \frac{N(p-1)-2b}{2(p+1)} \int |x|^b |u(t,x)|^{p+1} dx + \vareps \|\nabla u(t)\|^2_{L^2} \leq -\nu.
		\end{align}
	\end{lemma}
	
	\begin{proof}
		We consider separately two cases: $E(u_0)<0$ and $E(u_0)\geq 0$.
		
		If $E(u_0)<0$, we have from the conservation of energy that
		\begin{align*}
		\|\nabla u(t)\|^2_{L^2} &- \frac{N(p-1)-2b}{2(p+1)} \int |x|^b |u(t,x)|^{p+1} dx \\
		&= - \frac{N(p-1)-4-2b}{4} \|\nabla u(t)\|^2_{L^2} + \frac{N(p-1)-2b}{2} E(u_0)
		\end{align*}
		which shows \eqref{est-blow} with
		\[
		\vareps = \frac{N(p-1)-4-2b}{4}, \quad \nu = -\frac{N(p-1)-2b}{2} E(u_0).
		\]
		
		Let us consider the case $E(u_0)\geq 0$. In this case, we assume that \eqref{cond-blow-inte} holds. Arguing as in the proof of Lemma \ref{lem-coer}, we have
		\[
		\|\nabla u(t)\|_{L^2} \|u(t)\|^{\sigc}_{L^2} > \|\nabla Q\|_{L^2} \|Q\|^{\sigc}_{L^2}
		\]
		for all $t\in [0,T^*)$. Moreover, by taking $\vartheta= \vartheta(u_0,Q)>0$ as in \eqref{coer-est-prof-1}, we see that \eqref{est-G} holds for all $t\in [0,T^*)$. Consider the function
		\[
		G(\lambda) = \frac{N(p-1)-2b}{N(p-1)-4-2b} \lambda^2 - \frac{4}{N(p-1)-4-2b} \lambda^{\frac{N(p-1)-2b}{2}}, \quad \lambda>1.
		\]
		From \eqref{est-G}, we infer that there exists $\rho>0$ depending on $\vartheta$ such that
		\[
		G(\lambda)\leq 1-\vartheta\Longrightarrow\lambda>1+\rho.
		\] 
		Hence
		\begin{align} \label{est-blow-prof}
		\|\nabla u(t)\|_{L^2} \|u(t)\|^{\sigc}_{L^2} > (1+\rho)\|\nabla Q\|_{L^2} \|Q\|^{\sigc}_{L^2}
		\end{align}
		for all $t\in [0,T^*)$. Denote the left hand side of \eqref{est-blow} by $H(u(t))$. It follows that
		\begin{align*}
		H(u(t))(M(u(t)))^{\sigc} &= \frac{N(p-1)-2b}{2} E(u(t)) (M(u(t)))^{\sigc}  \\
		&\mathrel{\phantom{=}}+ \left(1+\vareps - \frac{N(p-1)-2b}{4}\right) \|\nabla u(t)\|^2_{L^2}(M(u(t)))^{\sigc}.
		\end{align*}
		By the conservation laws of mass and energy, \eqref{coer-est-prof-1}, \eqref{E-Q}, and \eqref{est-blow-prof}, we get
		\begin{align*}
		H(u(t)) (M(u_0))^{\sigc} &=\frac{N(p-1)-2b}{2} E(u_0) (M(u_0))^{\sigc} \\
		&\mathrel{\phantom{=}} + \left(1+\vareps - \frac{N(p-1)-2b}{4}\right) \|\nabla u(t)\|^2_{L^2}(M(u(t)))^{\sigc} \\
		&\leq \frac{N(p-1)-2b}{2} (1-\vartheta) E(Q) (M(Q))^{\sigc} \\
		&\mathrel{\phantom{=}}+ \left(1+\vareps - \frac{N(p-1)-2b}{4}\right) (1+\rho)^2 \left(\|\nabla Q\|_{L^2}\|Q\|^{\sigc}_{L^2}\right)^2 \\
		&= \left(\|\nabla Q\|_{L^2} \|Q\|^{\sigc}_{L^2}\right)^2 \left( -\frac{N(p-1)-4-2b}{4}(\vartheta + 2\rho + \rho^2) + \vareps (1+\rho)^2 \right) <0
		\end{align*}
		provided that $\vareps>0$ is taken sufficiently small. This shows \eqref{est-blow} with
		\[
		\nu = \|\nabla Q\|^2_{L^2} \left(\frac{M(Q)}{M(u_0)}\right)^{\sigc} \left(\frac{N(p-1)-4-2b}{4}(\vartheta + 2\rho + \rho^2) - \vareps (1+\rho)^2 \right) >0.
		\]
		The proof is complete.		
	\end{proof}
	
	\begin{proof}[Proof of Theorem \ref{theo-blow-inter}]
		By Lemma \ref{lem-viri-est-inte}, we have for all $t\in [0,T^*)$ and all $R>0$,
		\begin{align*}
		V''_{\psi_R}(t) \leq 8 \|\nabla u(t)\|^2_{L^2} &- \frac{4N(p-1)-8b}{p+1} \int |x|^b |u(t,x)|^{p+1} dx \\
		&+ CR^{-2} + \left\{
		\begin{array}{ccc}
		C R^{-(2(N-1)-b)}\|\nabla u(t)\|^2_{L^2} &\text{if} & p=5, \\
		C R^{-\left(\frac{(N-1)(p-1)}{2}-b\right)} \left(\|\nabla u(t)\|^2_{L^2} + 1\right) &\text{if}& p<5.
		\end{array}
		\right.
		\end{align*}
		From Lemma \ref{lem-Q-blow}, there exist $\vareps, \nu>0$ such that
		\[
		\|\nabla u(t)\|^2_{L^2} - \frac{N(p-1)-2b}{2(p+1)} \int |x|^b |u(t,x)|^{p+1} dx + \vareps \|\nabla u(t)\|^2_{L^2} \leq -\nu
		\]
		for all $t\in [0,T^*)$. It follows that
		\[
		V''_{\psi_R}(t) \leq -8\nu-8\vareps \|\nabla u(t)\|^2_{L^2} + CR^{-2} + \left\{
		\begin{array}{ccc}
		C R^{-(2(N-1)-b)}\|\nabla u(t)\|^2_{L^2} &\text{if} & p=5, \\
		C R^{-\left(\frac{(N-1)(p-1)}{2}-b\right)} \left(\|\nabla u(t)\|^2_{L^2} + 1\right) &\text{if}& p<5.
		\end{array}
		\right.
		\]
		Taking $R>0$ sufficiently large, we get
		\[
		V''_{\psi_R}(t) \leq -4\nu
		\]
		for all $t\in [0,T^*)$. The standard convexity argument yields $T^*<\infty$. The proof is complete.
	\end{proof}
	
	\subsection{Energy-critical nonlinearity} This subsection is devoted to show the existence of finite time blow-up solutions to \eqref{INLS} with the energy-critical nonlinearity. To this end, we recall some properties of $W$ (see \eqref{W}). One can readily check that $W$ is a solution to \eqref{equ-W}.
	
	\begin{proposition}[\cite{Lieb, Yanagida}]
		Let $N\geq 3$ and $b>0$. Then the optimal constant in the Sobolev inequality
		\begin{align} \label{Sobo-ineq}
		\int |x|^b |f(x)|^{\frac{2N+2b}{N-2}} dx \leq C_{\sha} \|\nabla f\|^{\frac{2N+2b}{N-2}}_{L^2}, \quad f \in \dot{H}^1_{\rad}(\R^N)
		\end{align}
		is achieved by $W$, i.e.,
		\[
		C_{\sha} = \int |x|^b |W(x)|^{\frac{2N+2b}{N-2}} dx \div \|\nabla W\|^{\frac{2N+2b}{N-2}}_{L^2}.
		\]
		Moreover, $W$ is the unique positive radial solution to \eqref{equ-W}.
	\end{proposition}
	
	\begin{proof} The existence of optimizers for \eqref{Sobo-ineq} was established by Lieb \cite[Theorem 4.3]{Lieb}. By a suitable scaling, we see that optimizers of \eqref{Sobo-ineq} are solutions to \eqref{equ-W}. Moreover, by \cite[Remark 2.1]{Yanagida}, $W$ is indeed a unique radial solution to \eqref{equ-W}.
	\end{proof}

	\begin{corollary}
		Let $N\geq 3$ and $b>0$. Let $W$ be the unique positive radial solution to \eqref{equ-W}. Then
		\begin{align} 
		\|\nabla W\|^2_{L^2} &= \int |x|^{-b} |W(x)|^{\frac{2N+2b}{N-2}} dx, \label{prop-W} \\
		E(W)&=\frac{b+2}{2N+2b}\|\nabla W\|_{L^2}^2. \label{prop-EW}
		\end{align}
	\end{corollary}
	
	\begin{lemma} \label{lem-W-blow}
		Let $N\geq 3$, $b>0$, and $p=\frac{N+2+2b}{N-2}$. Let $u_0 \in H^1_{\rad}(\R^N)$ satisfy either $E(u_0)<0$ or if $E(u_0) \geq0$, we assume that \eqref{cond-blow-ener} holds. Then there exist $\vareps, \nu>0$ such that for all $t\in [0,T^*)$,
		\begin{align} \label{est-blow-ener}
		\|\nabla u(t)\|^2_{L^2} -  \int |x|^b |u(t,x)|^{\frac{2N+2b}{N-2}} dx + \vareps \|\nabla u(t)\|^2_{L^2} \leq -\nu.
		\end{align}
	\end{lemma}
	
	\begin{proof}
		The case $E(u_0)<0$ is similar to that of Lemma \ref{lem-Q-blow}. Let us consider the case $E(u_0)\geq 0$. In this case, we assume that \eqref{cond-blow-ener} holds. By \eqref{Sobo-ineq}, we have
		\begin{align}
		E(u(t)) &= \frac{1}{2} \|\nabla u(t)\|^2_{L^2} - \frac{N-2}{2N+2b} \int |x|^b |u(t,x)|^{\frac{2N+2b}{N-2}}dx \nonumber\\
		&\geq \frac{1}{2} \|\nabla u(t)\|^2_{L^2} - \frac{N-2}{2N+2b} C_{\sha} \|\nabla u(t)\|^{\frac{2N+2b}{N-2}}_{L^2} \nonumber\\
		&=F\left(\|\nabla u(t)\|_{L^2}\right), \label{defi-F-ener}
		\end{align}
		where
		\[
		F(\lambda) := \frac{1}{2}\lambda^2 -\frac{N-2}{2N+2b} C_{\sha} \lambda^{\frac{2N+2b}{N-2}}.
		\]
		By the definition of $C_{\sha}$, we see that
		\begin{align*}
		F\left(\|\nabla W\|_{L^2}\right) &= \frac{1}{2} \|\nabla W\|^2_{L^2} -\frac{N-2}{2N+2b}\int |x|^b |W(x)|^{\frac{2N+2b}{N-2}} dx\\
		&= \frac{1}{2}\|\nabla W\|^2_{L^2} - \frac{N-2}{2N+2b} \int |x|^b |W(x)|^{\frac{2N+2b}{N-2}} dx = E(W).
		\end{align*}
		By the first assumption in \eqref{cond-blow-ener} and the conservation of energy, we have
		\[
		F\left(\|\nabla u(t)\|_{L^2}\right) \leq E(u_0) <E(W) = F\left(\|\nabla W\|_{L^2}\right)
		\]
		for all $t\in [0,T^*)$. From this and the second assumption in \eqref{cond-blow-ener}, the continuity argument yields $\|\nabla u(t)\|_{L^2}>\|\nabla W\|_{L^2}$ for all $t\in [0,T^*)$. We next take $\vartheta =\vartheta(u_0,W)>0$ such that
		\begin{align} \label{defi-vartheta}
		E(u_0) \leq (1-\vartheta) E(W).
		\end{align}
		By \eqref{prop-W}, we have
		\[
		E(W) =\frac{2+b}{2N+2b} \|\nabla W\|^2_{L^2} = \frac{2+b}{2N+2b} C_{\sha} \|\nabla W\|^{\frac{2N+2b}{N-2}}_{L^2}
		\]
		which together with \eqref{defi-F-ener} and \eqref{defi-vartheta} yield
		\[
		\frac{N+b}{2+b} \left(\frac{\|\nabla u(t)\|_{L^2}}{\|\nabla W\|_{L^2}} \right)^2 -\frac{N-2}{2+b} \left(\frac{\|\nabla u(t)\|_{L^2}}{\|\nabla W\|_{L^2}} \right)^{\frac{2N+2b}{N-2}} \leq 1-\vartheta
		\]
		for all $t\in [0,T^*)$. As $\|\nabla u(t)\|_{L^2}>\|\nabla W\|_{L^2}$ for all $t\in [0,T^*)$, we consider
		\[
		G(\lambda):= \frac{N+b}{2+b} \lambda^2 -\frac{N-2}{2+b} \lambda^{\frac{2N+2b}{N-2}}, \quad \lambda>1.
		\]
		As $G$ is strictly decreasing on $(1,\infty)$ with $G(1)=1$, there exists $\rho>0$ depending on $\vartheta$ such that 
		\[
		G(\lambda)\leq 1-\vartheta\Longrightarrow \lambda>1+\rho.
		\]
		In particular, we obtain
		\begin{align} \label{est-solu-ener}
		\|\nabla u(t)\|_{L^2} >(1+\rho) \|\nabla W\|_{L^2}
		\end{align}
		for all $t\in [0,T^*)$. Denote the left hand side of \eqref{est-blow-ener} by $H(u(t))$. Using the conservation of energy, \eqref{defi-vartheta}, \eqref{prop-EW}, and \eqref{est-solu-ener}, we have for $\vareps<\frac{2+b}{N-2}$,
		\begin{align*}
		H(u(t)) &= \frac{2N+2b}{N-2} E(u(t)) - \left(\frac{2+b}{N-2}-\vareps\right) \|\nabla u(t)\|^2_{L^2} \\
		& \leq \frac{2N+2b}{N-2} (1-\vartheta)E(W) - \left(\frac{2+b}{N-2}-\vareps\right) (1+\rho)^2 \|\nabla W\|^2_{L^2} \\
		&=\left( -\frac{2+b}{N-2} (\vartheta+2\rho +\rho^2) + \vareps(1+\rho)^2\right) \|\nabla W\|^2_{L^2}
		\end{align*}
		for all $t\in [0,T^*)$. Taking $\vareps>0$ sufficiently small, we obtain \eqref{est-blow-ener} with
		\[
		\nu = \left(\frac{2+b}{N-2}(\vartheta + 2\rho +\rho^2) -\vareps (1+\rho)^2\right) \|\nabla W\|^2_{L^2}>0.
		\]
		The proof is complete.		
	\end{proof}
	
	We have the following virial estimates in the energy-critical case whose proof is similar to that of Lemma \ref{lem-viri-est-inte}.
	\begin{lemma} \label{lem-viri-est-ener}
		Let $N\geq 4$, $0<b\leq 2(N-3)$, and $p=\frac{N+2+2b}{N-2}$. Let $u \in C([0,T^*), H^1_{\rad}(\R^N))$ be a solution to \eqref{INLS}. Let $\psi_R$ be as in \eqref{psi-R} and define $V_{\psi_R}(t)$ as in \eqref{V-varphi}. Then for any $R>0$, we have for all $t\in [0,T^*)$,
		\begin{align} \label{viri-est-ener}
		\begin{aligned}
		V''_{\psi_R}(t) \leq 8 \|\nabla u(t)\|^2_{L^2} &- 8 \int |x|^b |u(t,x)|^{\frac{2N+2b}{N-2}} dx \\
		&+ CR^{-2} + \left\{
		\begin{array}{ccc}
		C R^{-(2(N-1)-b)}\|\nabla u(t)\|^2_{L^2} &\text{if} & p=5, \\
		C R^{-\left(\frac{(2+b)(N-1)}{N-2}-b\right)} \left(\|\nabla u(t)\|^2_{L^2} + 1\right) &\text{if}& p<5,
		\end{array}
		\right.
		\end{aligned}
		\end{align}
		for some constant $C>0$.
	\end{lemma}

	\begin{proof}[Proof of Theorem \ref{theo-blow-ener}]
	The proof is similar to that of Theorem \ref{theo-blow-inter} using \eqref{est-blow-ener}. We thus omit the details.
	\end{proof}

	\section*{Acknowledgment}
	V.D.D. was partially supported by the ERC Grant CORFRONMAT No. 758620 (PI: Nicolas Rougerie). V.D.D. would like to express his deep gratitude to his wife - Uyen Cong for her encouragement and support.



\end{document}